%% file: duality.tex
\documentclass[a4paper,10pt]{amsart}

\usepackage[T1]{fontenc}
\usepackage[utf8]{inputenc}
\usepackage[active]{srcltx}
\usepackage{latexsym}
\usepackage{amsfonts}
\usepackage{amstext, amsthm}
\usepackage{amsmath}
\usepackage{graphicx, color}
\usepackage{enumitem}
\usepackage[all,dvips]{xy}
\input xy
\xyoption{all}
\xyoption{web}
\xyoption{poly}
\usepackage{amssymb,euscript,multicol,graphicx}
\usepackage[ec]{aeguill}
\usepackage{array, delarray, fancyhdr}

\input macros.tex %
\input englishsectionnement.tex %

\def\shIsom{{\mathcal I\!som}\,}
\def\shHom{{\mathcal{H}om}\,}
\def\shExt{{\mathcal{E}xt}\,}
\def\ch{\hbox{\rm ch}\,}

\def\Pic{\textrm{\rm Pic}}
\def\Ext{{\rm Ext}\,}
\def\Alb{\textrm{\rm Alb}}

\def\Spec{\textrm{\rm Spec}}
\def\picto{\Pic^{\tau}}
\def\champic{\mathcal{P}ic}
\def\id{\textrm{id}}
\def\cgs{\textrm{(CGS)}}
\DeclareMathOperator{\Triv}{Triv}
\renewcommand{\to}{%
\xymatrix@C=1pc{\ar[r] &}}
\def\lto{%
\xymatrix{\ar[r] &}}

\usepackage{a4wide}\setlist{noitemsep,topsep=0pt}
\setitemize{noitemsep,topsep=0pt}
\setenumerate{noitemsep,topsep=0pt}
\SetMathAlphabet\mathcal{normal}{U}{rsfs}{m}{n} 
\title{Duality for commutative group stacks}
\author{Sylvain Brochard}
\address{IMAG, Université de Montpellier, CNRS, Montpellier, France}
\email{sylvain.brochard@umontpellier.fr}

\begin{document}

\begin{abstract}
We study in this article the dual of a (strictly) commutative group stack $G$ and give some applications. Using the Picard functor and the Picard stack of~$G$, we first give some sufficient conditions for~$G$ to be dualizable. Then, for an algebraic stack $X$ with suitable assumptions, we define an Albanese morphism $a_X : X\to A^1(X)$ where $A^1(X)$ is a torsor under the dual commutative group stack $A^0(X)$ of $\Pic_{X/S}$. We prove that $a_X$ satisfies a natural universal property. We give two applications of our Albanese morphism. On the one hand, we give a geometric description of the elementary obstruction and of universal torsors (standard tools in the study of rational varieties over number fields). On the other hand we give some examples of algebraic stacks that satisfy Grothendieck's section conjecture.
\end{abstract}
\maketitle

\tableofcontents

\section{Introduction}

\noindent {\bf Motivation.}
If $X$ is a smooth and projective geometrically connected curve (over a field 
$k$), its Abel-Jacobi morphism $a : X\to \Pic^1_{X/k}\subset \Pic_{X/k}$ is the 
$X$-point corresponding to the effective divisor $\Delta_X\subset X\times_k X$. 
The morphism $a$ maps a point $x$ in $X$ to the class of the line 
bundle~$\Oc_X(x)$. The target $\Pic^1_{X/k}$ is a torsor under the jacobian 
$J=\Pic^0_{X/k}$, which is an abelian variety. If $X$ has a $k$-point we get a 
morphism from $X$ to $J$ itself. This morphism is useful in many contexts: it 
often allows to reduce some questions (\emph{e.g.} the injectivity in 
Grothendieck's section conjecture) to analogous questions about abelian 
varieties. Now let us consider the following orbifold curve $\Xc$ over $X$.
\begin{figure}[h]
\begin{center}
\vskip-0.5cm
\input{courbetordue.pstex_t}
\end{center}
\end{figure}
The morphism $\pi$ is an isomorphism above $X\setminus\{p\}$ and looks like 
$\left[\Spec\left(\frac{A[x]}{x^r-s}\right) / \mu_r \right] \to \Spec A$ in a 
neighboorhood $\Spec A$ of $p$.
It might be desirable to have at our disposal a natural morphism $\Xc \to 
\Pic_{\Xc/k}$ that could play the role of an Abel-Jacobi morphism. Let us look 
for such a morphism. Of course we have the obvious composition
$$\xymatrix{
\Xc \ar[r]^-{\pi} & X\ar[r]^-a &
\Pic_{X/k} \ar[r]^-{\pi^*} & \Pic_{\Xc/k}
}$$
but it completely overlooks the stacky structure of $\Xc$. 
If we try to mimic the construction of $a$ for $X$ in terms of the diagonal divisor, we come across the fact that $\Delta : \Xc \to \Xc \times_k \Xc$ is not a closed immersion (it is not a monomorphism). Actually, it can be seen (see~\cite[5.4]{Brochard_Picard}) that there is an exact sequence
$$\xymatrix{
0 \ar[r] &
\Pic_{X/k}    \ar[r]^-{\pi^*} &
\Pic_{\Xc/k}    \ar[r]^{} &
\Z/r\Z    \ar[r] & 0
}$$
so that $\Pic_{\Xc/k}$ is a disjoint union of $r$ copies  
of $\Pic_{X/k}$ and $\pi^*$ induces an isomorphism $\Pic^0_{X/k} \simeq \Pic^0_{\Xc/k}$. But the stack $\Xc$ itself is connected. So, in some sense, a ``natural'' morphism from $\Xc$ to $\Pic_{\Xc/k}$ \emph{must} overlook the stacky structure of $\Xc$.

There is another morphism from $X$ to (a torsor under) an abelian variety: the Albanese morphism. In the case of our curve~$X$, it turns out that it coincides with the Abel-Jacobi morphism \emph{via} the autoduality of the Jacobian $J^t\simeq J$. However it is better-behaved than Abel-Jacobi in at least two respects: it is functorial, and it exists for varieties of arbitrary dimension. In the classical setting, the Albanese variety associated to $X$ is dual to the Picard variety $\Pic^0_{X/k}$. As the above discussion suggests, to get a meaningful morphism for the stack $\Xc$, we need to replace the neutral component with the whole Picard functor $\Pic_{\Xc/k}$ (or at least its torsion component). The latter is not an abelian variety anymore. So this first raises the following question: what is the dual of such a group scheme?\bigskip

\noindent {\bf Duality.}
A dual of an arbitrary abelian sheaf is already defined in 
the framework of commutative group stacks. Roughly 
speaking, a commutative group stack (sometimes called a 
Picard stack) is a stack $G$ together with an addition 
morphism $\mu : G\times_k G\to G$ that satisfies some 
natural conditions (see~\ref{def_Picard_cat} 
and~\ref{def_cgs} below). Deligne has given in 
SGA~4~\cite[XVIII~1.4]{SGA4} a description of the 2-category 
of commutative group stacks in terms of length 1 complexes 
of sheaves of abelian groups (see~\ref{description_cgs}). 
For example, an abelian sheaf $F$ corresponds to the complex 
$[0\to F]$, the classifying stack $BF$ of $F$ corresponds to 
$[F\to 0]$, while a 1-motive (in the sense of Deligne 
\cite{Deligne_Theorie_Hodge_3}) is a complex $[C\to A]$ 
where $C$ is a twisted lattice and~$A$ is a semi-abelian 
scheme. For such a commutative group stack $G$, its dual is the stack 
of homomorphisms from $G$ to $\bgm$
$$D(G)=\shHom(G, \bgm).$$
The duality of Deligne's 1-motives 
\cite[10.2.10]{Deligne_Theorie_Hodge_3} is a particular case 
of this construction. For example, if $A$ is an abelian 
variety, the dual $D(A)$ is the classical dual abelian 
variety, the dual of $\Z$ is $\bgm$, the dual of $\gm$ 
is $B\Z$. More generally, if $F$ is a finite flat, a 
diagonal, or a constant group scheme, then $D(F)=BF^D$ and 
$D(BF)=F^D$, where $F^D=\Hom(F, \gm)$ denotes the Cartier 
dual of $F$. For any commutative group stack, there is a 
canonical evaluation map
$$e_G : G\lto D(D(G)).$$
If $e_G$ is an isomorphism, we say that $G$ is dualizable. 
It seems natural to look for sufficient conditions for a 
commutative group stack $G$ to be dualizable.
Of course Deligne's 1-motives are dualizable. Some other 
stable classes of dualizable commutative group stacks, inspired from 
1-motives, have been used recently (see for instance
\cite{Braverman_Bezrukavnikov_Geometric_Langlands_%
Correspondance},
\cite{Chen_Zhu_Geometric_Langlands_Prime}
\cite{Donagi_Pantev_Langlands_Duality_Hitchin},
\cite{Donagi_Pantev_Torus_Fibrations_Gerbes},
\cite{Jossen_On_The_Arithmetic}, 
\cite{Laumon_Transformation_de_Fourier} and
\cite{Russell_Albanese_Varieties_With}), \emph{e.g.} the 
class of commutative group stacks that are, locally on~$S$, finite 
products of copies of~$\Z$, $\bgm$ or abelian schemes (such 
commutative group stacks are sometimes called Beilinson 1-motives).

Two kinds of geometric conditions 
can be taken on the commutative group stack $G$, depending on how we 
think about it. One possibility is to require that there 
exists a presentation of $G$ as the stack associated to a 
length one complex $[G^{-1} \to G^0]$, where the sheaves 
$G^{-1}$ and $G^0$ satisfy certain conditions (this is what 
happens for 1-motives and the above-mentioned 
generalizations). Another possibility is to impose directly 
geometric conditions on the stack $G$, regardless of its 
presentations. It more or less amounts to impose conditions 
on the cohomology sheaves $H^i(G^{\bullet})$ for a length 
one complex $G^{\bullet}$ that represents $G$ (see for 
instance~\ref{prop_abelian_stacks}). In this spirit, we 
conjecture the following:
\begin{conjecture}
\label{conj}
Let $S$ be a base scheme with $2\in \Oc_S^{\times}$.
 Let $G$ be a proper, flat, and finitely presented algebraic commutative group stack over $S$, with finite and flat inertia stack. Then:
\begin{enumerate}
\item $D(G)$ is algebraic, proper, flat and finitely presented, with finite diagonal.
\item $G$ is dualizable.
\end{enumerate}

\end{conjecture}
We get partial results in this direction: under the 
assumptions we prove that (2) is a consequence of (1) 
(see~\ref{preuve_conj1_implique_conj2}), that $D(G)$ is 
algebraic and finitely presented with quasi-finite diagonal 
(\ref{thm_representabilite1} and~\ref{thm:inertie_finie}), and we prove the 
conjecture with the additional assumption that $H^0(G)$ is 
cohomologically flat~(\ref{thm_representabilite4} 
and~\ref{thm_dualizability2}). We give other results with 
the same flavour in sections~\ref{section_dual} 
and~\ref{section_dualizability}. The assumption that $2\in 
\Oc_S^{\times}$ is there because of the use of the results of 
Section~\ref{section_Ext}, and might be superfluous.
\bigskip

\noindent {\bf Albanese morphism.}
Now let us come back to our original goal: generalize the 
Albanese morphism to the context of algebraic stacks. This 
is done in section~\ref{section_Albanese_torsor}. Let $f : 
X\to S$ be an algebraic stack such that $\Oc_S\to f_*\Oc_X$ 
is universally an isomorphism (i.e. for any morphism $T\to S$, the morphism 
$\Oc_T\to 
(f_T)_*\Oc_{X_T}$ is an isomorphism, where $f_T : X_T\to T$ is the morphism 
obtained from $f$ after the base change $T\to S$) and such that $f$ locally has 
sections 
in the \emph{fppf} topology on $S$. Then we define a 
commutative group stack $A^0(X)$, a torsor $A^1(X)$ 
under~$A^0(X)$ and a canonical morphism $a_X : X\to A^1(X)$. 
If we assume that the torsion component of the Picard 
functor $\picto_{X/S}$ is proper, flat, cohomologically flat 
and finitely presented over $S$, then our Albanese morphism 
is universal among morphisms from $X$ to torsors under what 
we call ``abelian stacks'' (see~\ref{PU_albanese}).
Hence it generalizes the classical Albanese morphism given 
in~FGA~VI, théorème~3.3~(iii)~\cite[exp. 236]{FGA} in the 
classical setting of schemes. Note by the way that even in 
that classical setting, this yields a nice explicit
construction of the Albanese torsor and morphism, while the 
existence proof of FGA was not constructive. The 
construction of these objects is very simple and natural. 
Let us denote by $\champic(X/S)$ the Picard stack of $X$. 
Then the torsor $A^1(X)$ appears as a substack of the dual 
$D(\champic(X/S))$, and the morphism $a_X$ maps a point 
$x\in X(S)$ to the pullback morphism $x^* : \champic(X/S) 
\to \bgm$. The universal property~\ref{PU_albanese} 
mentioned above is one of the main results of the paper. 
In practice, our assumption that $\picto_{X/S}$ is 
representable more or less requires $X$ to be at least 
proper and flat. Note however that, unlike what happens for 
generalizations of the Albanese mapping to other contexts 
(\emph{e.g.} for an arbitrary variety over a field, see for 
instance~\cite{Barbieri-Viale_Srinivas}, 
\cite{Ramachandran_Duality_Of_Albanese}, 
\cite{Esnault_Srinivas_Viehweg}, 
\cite{Russell_Generalized_Albanese_And}, 
\cite{Russell_Albanese_Varieties_With}; note that some of 
them consider morphisms into non-compact group schemes, 
which causes difficulties very different from ours), we do 
not need to assume that $X$ has an~$S$-point to get our 
Albanese morphism, and the Albanese morphism is defined over 
the whole of~$X$ (compare 
with~\cite[7.3]{Barbieri-Viale_Srinivas}). Also, we do not 
need to assume that the base scheme $S$ is the spectrum of 
a field.

For the example of the orbifold curve from the beginning of this introduction, 
the Albanese stack $A^0(\Xc)$ is a gerbe 
over $A^0(X)$ banded by $\mu_r$, so that it does remember of the stacky 
structure of the orbifold curve $\Xc$. This example is
discussed in more details in~\ref{exemple:twisted_curves} (see 
also~\ref{exemple:root_stack} 
which is very similar).

\bigskip

\noindent {\bf Applications.}
In \cite{Borne_Vistoli_Fundamental_Gerbe}, Borne and Vistoli have extended Nori's theory of the fundamental group scheme to a theory of the fundamental gerbe, which applies to algebraic stacks even in absence of a rational point. When it exists (for a given algebraic stack $X$ over a field $k$), the fundamental gerbe is a morphism $X\to \Pi_{X/k}$ to a profinite gerbe, with a universal property for morphisms to finite stacks (see~\cite[5.1]{Borne_Vistoli_Fundamental_Gerbe}). Their formalism allows to reformulate Grothendieck's section conjecture as follows: the traditional ``section map'' is bijective if and only if the natural morphism $X\to \Pi_{X/k}$ induces a bijection on isomorphism classes of $k$-rational points, in which case we will say that GSC holds for $X$. In this paper we prove that GSC holds for some root stacks over Severi-Brauer varieties. As explained in \cite{Borne_Vistoli_Fundamental_Gerbe}, this result yields a general method for producing examples of smooth projective curves of genus at least 2 that satisfy GSC. We also prove that, if the torsion component $\picto_{X/k}$ is finite, then the Albanese torsor $A^1(X)$ coincides with the abelianization of the fundamental gerbe.

As another application, we give a geometric description of 
some classical arithmetical constructions. Let $X$ be a 
proper and smooth variety over a field $k$ and assume that 
$X$ is $\ov{k}$-rational. Assume that $X$ has ``universal 
torsors'' in the sense of Colliot-Thélène and Sansuc (we 
recall the definition in 
section~\ref{section_rational_varieties}).
If the field $k$ is of finite type over its prime subfield, then by \cite{Colliot-Thelene_Sansuc_Note1} there are up to isomorphism finitely many universal torsors $\pi_i : T^i \to X$ ($i=1\dots n$) and they provide a partition of the set of $k$-rational points of $X$:
$$X(k) = \coprod_i \pi_i(T^i(k)).$$
This description of the $k$-points is useful because, heuristically, the arithmetic of the universal torsors should be much simpler than the arithmetic of the variety~$X$. We explain in section~\ref{section_rational_varieties} how these universal torsors can be described in terms of the Albanese torsor~$A^1(X)$. In this case~$A^0(X)\simeq BG_0$ where $G_0$ is the Cartier dual of $Pic_{X/k}$, which means that $A^1(X)$ can be seen as a~$G_0$-gerbe. We prove moreover that the ``elementary obstruction'' -- an obstruction to the existence of universal torsors -- vanishes if and only if the gerbe $A^1(X)$ is trivial.
\bigskip

\noindent {\bf Contents.}
In section~\ref{section_group_stacks}, we briefly recall the 
definition of commutative group stacks and their description 
in terms of length one complexes. We also give preliminary 
results about some classes of commutative group stacks that will be used 
all along the paper. The dual of a commutative group stack is defined in 
section~\ref{section_dual}. Then we start computing the 
duals of some stacks, compare the dual $D(G)$ with the 
Picard stack of $G$ and use this comparison to get the main 
representability theorem mentioned above 
(\ref{thm_representabilite1}). The reader will also find 
along the way some results that might have independant 
interest (\emph{e.g.} Raynaud's devissage 
\ref{devissage_raynaud} for proper and flat commutative group schemes 
over Artin local rings).

In section~\ref{section_dualizability}, we address the 
question of the dualizability, that is: for a given 
commutative group stack~$G$, is $e_G : G\to DD(G)$ an 
isomorphism? We give a positive answer for some classes of 
stacks, mostly with properness assumptions 
(see~\ref{thm_dualizability}, \ref{thm_dualizability2}). 
This proves that the 2-functor $D(.)$ induces a 
2-antiequivalence for these classes of group stacks.

Sections~\ref{section_torsors} and~\ref{section_thm_square} are devoted
to necessary preliminary discussions before dealing with the Albanese morphism. 
The former recalls basic facts about torsors under a commutative group stack 
and gerbes banded by a sheaf of commutative groups, for the convenience of the 
reader, while the latter provides a variant of the famous ``theorem of the 
square'' for abelian stacks and for classifying stacks.

Section~7 is devoted to the definition and first properties of the Albanese torsor and Albanese morphism. In section~\ref{section_PU} we prove its universal properties and compute some examples.

Sections~\ref{section_rational_varieties} and~\ref{section_CS} are devoted to applications, respectively to rational varieties and in the context of the section conjecture.

The last section~\ref{section_Ext} recollects for the convenience of the reader some known results about $\shExt$ sheaves that are used throughout the paper.

\label{notations}

\bigskip

\noindent {\bf Notations and terminology.} Throughout the 
paper, the topology we use is the \emph{fppf} topology, 
unless we explicitly use another one. Let $S$ be a base 
scheme. If $X$ is a stack over~$S$, we denote by 
$\champic(X/S)$ its Picard stack, that is, the stack whose 
fiber category over an $S$-scheme $U$ is the category of 
invertibles sheaves on $X\times_S U$. The Picard functor is 
the \emph{fppf} sheafification of $U\mapsto \Pic(X\times_S 
U)$ and is denoted by $\Pic_{X/S}$. The coarse moduli sheaf (or coarse moduli 
space) of~$X$ is the \emph{fppf} sheafification of the presheaf of isomorphism 
classes of objects of $X$, e.g. $\Pic_{X/S}$ is the coarse moduli space of 
$\champic(X/S)$. It comes with a universal property, see~\cite[(3.19)]{LMB}. A 
morphism $f : X\to Y$ 
is said to be cohomologically flat in dimension zero (or 
just cohomologically flat, for short) if the formation 
of~$f_*\Oc_X$ commutes with base change. If $G$ is a sheaf 
of abelian groups over $S$, the sheaf 
$\Hom_{S-\textrm{gp}}(G, \gm)$ of group homomorphisms from 
$G$ to $\gm$ is denoted by $G^D$ and is called the Cartier 
dual of $G$. The \emph{fppf} sheaf $\shExt^i(G, \gm)$ will 
often be shortened as $E^i(G)$. We denote by~$G^0$ and 
$G^{\tau}$ respectively the neutral and torsion component 
of~$G$ (see \emph{e.g.}~\cite[3.3.1]{Brochard_finiteness}). 
If $A$ is an abelian scheme -- \emph{i.e.} a smooth and 
proper group scheme with connected fibers -- we denote by 
$A^t$ its dual~$\Pic^0_{A/S}$. It is an abelian scheme, 
isomorphic to $E^1(A)$ (\ref{dual_and_ext1_abelian}) and 
the 
isomorphism will often be used implicitly. We found it 
convenient to give (perhaps uncommon) names to some objects. 
We hope that it will not bother the reader. Here is a list 
of those terms with the place where the definition can be 
found: commutative group stack (\ref{def_cgs}), dual of a 
commutative group stack (\ref{def_dual_cgs}), abelian stack 
(\ref{def_abelian_stacks}), duabelian group 
(\ref{def_duabelian_groups}), Cartier group 
(\ref{Cartier_groups}), evaluation map $e_G$ 
(\ref{def_eval_map}), $H^{-1}(G)$ and $H^0(G)$ 
(\ref{def_hi}), exact sequence of commutative group stacks 
(\ref{def_suites_exactes}).
\bigskip

\noindent {\bf Acknowledgments.}
I warmly thank Michel Raynaud for Proposition~\ref{devissage_raynaud}, and 
Laurent Moret-Bailly who first gave me the idea of introducing the map $x\mapsto 
x^*$. I am also grateful to Jean Gillibert, Bertrand Toën, and Olivier 
Wittenberg for helpful conversations or comments. The content of 
section~\ref{section_CS} started out as joint work with Niels Borne and Angelo 
Vistoli, whom I thank heartily. I thank the anonymous 
referees for 
their very careful reading of the paper and their numerous comments and 
corrections that greatly clarified the exposition. The author was partially 
supported by ANR Arivaf (ANR-10-JCJC 0107).

\section{Commutative group stacks}
\label{section_group_stacks}

In order to fix some notations, and for the convenience of 
the reader, we recall here very briefly the basics about 
commutative group stacks, from SGA~4~\cite[XVIII~1.4]{SGA4}. These are 
the stacks called ``champs de Picard strictement commutatifs'' in 
\cite{SGA4}. For the 
details, the reader is refered to the original source. Then 
we define a notion of exact sequence of commutative group 
stacks, and two classes of commutative group stacks or sheaves 
that will be used in the sequel. Note that the group law is denoted additively 
in this section, but sometimes in the paper it will be more convenient to use a 
multiplicative notation.

\begin{defi}
\label{def_Picard_cat}
 A \emph{Picard category} is a groupoid $G$ (\emph{i.e.} a category in which all morphisms are isomorphisms) together with a functor
$$+ : G\times G \lto G,$$
a functorial associativity isomorphism
$$\xymatrix{\lambda_{x,y,z} : (x+y)+z \ar[r]^-{\sim}& x+(y+z)},$$
and a functorial commutativity isomorphism
$$\xymatrix{\tau_{x,y} : x+y \ar[r]^-{\sim}& y+x}$$
satisfying the following properties:
\begin{enumerate}
 \item the pentagon axiom and the hexagon axiom (see 
\cite[XVIII 1.4.1]{SGA4}),
\item for any objects $x,y$ of $G$, $\tau_{y,x}\circ \tau_{x,y}=\id_{x+y}$, and $\tau_{x,x}=\id_{x+x}$,
\item for any object $x$ of $G$ the functor $y\mapsto x+y$ is an equivalence of categories.
\end{enumerate}
\end{defi}

\begin{defi}
\label{def_cgs}
 Let $S$ be a scheme. A \emph{commutative group stack} over~$S$ is an $S$-stack~$G$ together with a morphism $+ : G\times_S G\to G$, and 2-isomorphisms $\lambda$, $\tau$ as above, such that for any $S$-scheme~$U$, the fiber category $G(U)$ equipped with the restrictions of $+, \lambda$ and $\tau$ is a Picard category.
\end{defi}

\begin{remarque}\rm
A commutative group stack $G$ always has a neutral object, 
\emph{i.e.} a pair~$(e,\eps)$ where $e$ is an object 
of~$G(S)$ and $\eps$ is an isomorphism $e+e\to e$. Moreover, 
this neutral object is unique up to a unique isomorphism, 
and is neutral on the left and right in a natural sense 
(see~\cite[XVIII~1.4.4]{SGA4}). For a fixed neutral 
object~$(e,\eps)$, there is an \emph{inverse} morphism $- : 
G\to G$ which is unique up to a unique 2-isomorphism (see 
\cite[14.4.1]{LMB}).
\end{remarque}

\begin{defi}
\label{def_morphismes_groupes}
Commutative group stacks over $S$ naturally form a 2-category $\cgs$ as follows. 
Let $G, H$ be two commutative group stacks over~$S$. A \emph{homomorphism} from 
$G$ to~$H$ is a morphism of stacks $f : G\to H$ with a 2-isomorphism $\alpha_f 
: f\circ + \Rightarrow +\circ(f\times f)$ such that the following diagrams 
commute:
$$\xymatrix@C=1,2pc@R=1,2pc{
f(x\!+\!y)\ar[r]^-{\alpha_f} \ar[d]_{f(\tau_G)} & f(x)\!+\!f(y)
\ar[d]^{\tau_H} \\
f(y\!+\!x)\ar[r]_-{\alpha_f} & f(y)\!+\!f(x)
} \qquad
\xymatrix@C=1,2pc@R=1,2pc{
f((x\!+\!y)\!+\!z)\ar[r]^-{\alpha_f} \ar[d]_{f(\lambda_G)} & 
f(x\!+\!y)\!+\!f(z) \ar[r]^-{\alpha_f+\id}& 
(f(x)\!+\!f(y))\!+\!f(z)\ar[d]^{\lambda_H} \\
f(x\!+\!(y\!+\!z))\ar[r]^-{\alpha_f} & f(x)\!+\!f(y\!+\!z) 
\ar[r]^-{\id+\alpha_f} & f(x)\!+\!(f(y)\!+\!f(z))\, .
}$$
If $f, g$ are homomorphisms from $G$ to $H$, a 2-isomorphism 
$u : f\Rightarrow g$ is a 2-isomorphism between the 
underlying morphisms of stacks, that is compatible with 
$\alpha_f$ and $\alpha_g$ (see \cite[XVIII~1.4.6]{SGA4}). 
We denote by~$\Hom_{cgs}(G,H)$ the category of 
homomorphisms from~$G$ to~$H$.
\end{defi}

Let us denote by $\ov{\cgs}$ the category whose objects are commutative group 
stacks and whose morphisms are isomorphism classes of homomorphisms. There is a 
very convenient description of $\cgs$ and $\ov{\cgs}$ in terms of length 1 
complexes of sheaves of abelian groups, as follows. Let $G^{\bullet} = [G^{-1} 
\to G^0]$ be a length 1 complex of sheaves of abelian groups. We denote by 
$\ch(G^{\bullet})$ the quotient stack $[G^0/G^{-1}]$. It is naturally a 
commutative group stack over $S$. Let $D^{[-1,0]}(S,\Z)$ denote the derived 
category of length 1 complexes of sheaves of abelian groups over~$S$ 
(equivalently, this is the full subcategory of the derived category of sheaves 
on~$S$ of complexes which are homologically concentrated in degrees -1 and 0).

\begin{thm}[{\cite[XVIII, 1.4.15 and 1.4.17]{SGA4}}]\ 
\label{description_cgs}

 \begin{enumerate}
  \item  The functor $\ch(.)$ induces an equivalence of categories from 
$D^{[-1,0]}(S,\Z)$ to $\ov{\cgs}$. In the sequel we denote by $G\mapsto 
G^{\flat}$ a quasi-inverse to this functor.
\item Let $C(S)$ denote the 2-category of complexes of abelian sheaves~$G^{\bullet}$ over~$S$, such that $G^i=0$ for $i\notin [-1, 0]$ and $G^{-1}$ is injective. Morphisms are morphisms of complexes and 2-morphisms are homotopies. Then $\ch(.)$ induces a 2-equivalence of 2-categories from $C(S)$ to $\cgs$.
 \end{enumerate}
\end{thm}

\begin{exemple}\rm
 Let $G$ be an abelian sheaf over~$S$. We can regard it as a commutative group stack (with trivial automorphism groups). Then $G^{\flat}\simeq [0\to G]$. We can also look at the classifying stack~$BG$ of~$G$ (which classifies~$G$-torsors). Then $(BG)^{\flat}\simeq [G\to 0]$.
\end{exemple}

\begin{exemple}\rm
 Let $X$ be an algebraic stack over~$S$. Then the Picard stack $\champic(X/S)$ is a commutative group stack, and
$\champic(X/S)^{\flat}\simeq \tau_{\leq 0}(Rf_*\gm[1])$
where $f : X\to S$ is the structural morphism of~$X$ and where, for a complex 
$K^{\bullet}$ with differential $\partial^i : K^i \to K^{i+1}$, the truncation 
$\tau_{\leq 0}(K^{\bullet})$ is the complex given by $\tau_{\leq 
0}(K^{\bullet})^i=K^i$ if $i<0$, $\tau_{\leq 0}(K^{\bullet})^i=0$ if $i>0$ and 
$\tau_{\leq 0}(K^{\bullet})^0=\Ker \partial^0$.
\end{exemple}

\begin{exemple}\rm
\label{def_hom_gp_stacks}
 Let $G, H$ be two commutative group stacks over~$S$. Then 
the stack of homomorphisms $U\mapsto \Hom_{cgs}(G\times_S 
U,H\times_S U)$ is naturally a commutative group stack, 
denoted by~$\shHom(G,H)$. Then 
by~\cite[XVIII~1.4.18]{SGA4}, $\shHom(G,H)^{\flat}\simeq 
\tau_{\leq 0}\RHom(G^{\flat},H^{\flat})$.
\end{exemple}

\begin{defi}
\label{def_hi}
 If $G$ is a commutative group stack over~$S$, we denote by $H^0(G)$ its coarse moduli sheaf (which is an abelian $S$-sheaf) and by $H^{-1}(G)$ the automorphism group of a neutral section of~$G$. Note that, for any complex~$G^{\bullet}$ as above, there are canonical isomorphisms $H^i(G^{\bullet})\to H^i(\ch(G^{\bullet}))$.
\end{defi}

The following useful proposition is stated without proof 
in~\cite{Chen_Zhu_Geometric_Langlands_Prime}.

\begin{prop}[{\cite[A.2.2]{Chen_Zhu_Geometric_Langlands_Prime}}]
 \label{prop:suites_exactes}
Let $j : A\to B$ and $\pi : B\to C$ be two morphisms of commutative group 
stacks over a base scheme $S$. The following are equivalent:
\begin{enumerate}
 \item The morphism $\pi$ is an \emph{fppf} epimorphism, and $j$ identifies $A$ 
with the kernel of $\pi$ (i.e. the fiber of $\pi$ above the neutral object 
of $C$).
\item There exists a morphism $C^\flat \to A^\flat[1]$ in $D(S,\Z)$ (the 
derived category of abelian sheaves on $S$)  such that
\[
 A^\flat \to B^\flat \to C^\flat \to A^\flat[1]
\]
is an exact triangle.
\end{enumerate}
\end{prop}
\begin{proof}
 Assume (2). The exact triangle induces a long exact sequence
\[
 0\to H^{-1}(A) \to H^{-1}(B) \to H^{-1}(C) \to
H^0(A)\to H^0(B)\to H^0(C) \to 0
\]
In particular $H^0(B)\to H^0(C)$ is an \emph{fppf} epimorphism, hence $\pi : 
B\to C$ is an epimorphism. To prove that $A$ is the kernel of $\pi$, we 
will use the 2-equivalence~\ref{description_cgs}~(2). Let 
$B^{\bullet}=[\xymatrix{B^{-1}\ar[r]^{d_B} & B^0}]$ and
$C^{\bullet}=[\xymatrix{C^{-1}\ar[r]^{d_C} & C^0}]$
be complexes in $C(S)$ corresponding to $B$ and $C$ (in particular $B^{-1}$ and 
$C^{-1}$ are injective) and let $\pi^{\bullet}$ be a morphism of 
complexes corresponding to $\pi$. Since we have the given exact triangle, we 
may assume that $A^{\flat}[1]$ is the mapping cone of $\pi^{\bullet}$, i.e. 
in degrees -1, 0 and 1, $A^{\flat}$ is the complex
\[
 \xymatrix{
B^{-1}\ar[r]^-{d_A^{-1}}& B^0\oplus C^{-1}\ar[r]^-{d_A^0} &C^0
}
\]
where $d_A^{-1}=(d_B,-\pi^{-1})$ and $d_A^0=(-\pi^0)\oplus(-d_C)$, and the 
other terms are zero. This complex is quasi-isomorphic to
$A^{\bullet}=[\xymatrix{B^{-1}\ar[r]^{d_A^{-1}} & \Ker d_A^0}]$, which belong 
to $C(S)$. 
Now the morphism $j : A\to B$ corresponds to the morphism of complexes 
$j^{\bullet} : A^{\bullet}\to B^{\bullet}$ where $j^{-1}$ is the identity of
$B^{-1}$ and $j^0$ is the composition $\Ker d_A^0\inj B^0\oplus C^{-1} \to 
B^0$.
Let $h : \Ker d_A^0 \to C^{-1}$ be the projection to $C^{-1}$. Then we have the 
relations 
$hd_A^{-1}=-\pi^{-1}j^{-1}$ and $d_Ch=-\pi^0j^0$ hence $h$ is a homotopy from 
$\pi^{\bullet}j^{\bullet}$ to 0. In other words, in $C(S)$ we have a 
2-commutative diagram
\[
 \xymatrix@R=0,2pc@C=0,2pc{
A^{\bullet} \ar[rr] \ar[dd]_{j^{\bullet}} && 0 \ar[dd]\\
&^h\FlecheNE \\
B^{\bullet} \ar[rr]_{\pi^{\bullet}} &&C^{\bullet}.
}
\]
We will prove that this diagram is 2-cartesian. In view of the 
2-equivalence~\ref{description_cgs}~(2) this will prove that $A$ is the kernel 
of $\pi$. We have to prove the following assertions:
\begin{enumerate}
 \item[(i)] For any object $D^{\bullet}$ of $C(S)$ and any pair 
$(f^{\bullet},\alpha)$, where $f^{\bullet} : D^{\bullet} \to B^{\bullet}$ is a 
morphism in $C(S)$ and $\alpha$ is a homotopy from $\pi^{\bullet}f^{\bullet}$ 
to 0, there exists a morphism $g^{\bullet} : D^{\bullet} \to A^{\bullet}$ and a 
homotopy 
$\beta$ from $f^{\bullet}$ to $j^{\bullet}g^{\bullet}$ such that 
$\alpha=(g^{\bullet}\star h)\circ(\pi^{\bullet}\star \beta)$ (viewed as 
2-morphisms in the 2-category $C(S)$), i.e. such that 
$\alpha=hg^0+\pi^{-1}\beta$ in terms of homotopies (i.e. morphisms from $D^0$ 
to $C^{-1}$).
\item[(ii)] For a given $(f^{\bullet},\alpha)$, the pair $(g^{\bullet},\beta)$ 
is unique up to a 
unique isomorphism. More precisely if $(g'^{\bullet},\beta')$ is another such 
pair, then 
there exists a unique homotopy $\gamma$ from $g^{\bullet}$ to $g'^{\bullet}$ 
such that $(j^{\bullet}\star\gamma)\circ \beta=\beta'$, i.e. such that 
$j^{-1}\gamma+\beta=\beta'$.
\end{enumerate}
\[
 \xymatrix@C=5pc{
D^{\bullet} \ar[d]_{g^{\bullet}} \ar@/_3pc/[dd]_{f^{\bullet}}&
D^{-1} \ar[r]^{d_D} \ar[d]_{g^{-1}} \ar@/_3pc/[dd]_{f^{-1}}&
D^0 \ar[d]^{g^0}  \ar@{.>}[dl]_(.6){\gamma} \ar@{.>}[ddl]_(.78){\beta} 
\ar@{.>}[dddl]^(.4){\alpha} \ar@/^3pc/[dd]^{f^{0}}\\
A^{\bullet} \ar[d]_{j^{\bullet}}&
B^{-1} \ar[r]^{d_A^{-1}} \ar[d]_{j^{-1}} & \Ker d^0_A \ar[d]^{j^{0}} 
\ar@{.>}[ddl]^(.69){h}\\
B^{\bullet} \ar[d]_{\pi^{\bullet}}&
B^{-1} \ar[r]^{d_B} \ar[d]_{\pi^{-1}} & B^0 \ar[d]^{\pi^0}\\
C^{\bullet}&
C^{-1}\ar[r]^{d_C}& C^0}
\]
Let us first prove (ii). Since $j^{-1}=\id$ we only have to check that 
$\gamma:=\beta'-\beta$ is indeed a homotopy from $g^{\bullet}$ to 
$g'^{\bullet}$, i.e. that it satisfies $\gamma d_D=g'^{-1}-g^{-1}$ and 
$d_A^{-1}\gamma=g'^0-g^0$. The first of these relations follows from $\beta 
d_D=j^{-1}g^{-1}-f^{-1}$ and $\beta'd_D=j^{-1}g'^{-1}-f^{-1}$. The second one 
follows from the equalities
\begin{align*}
 j^0d_A^{-1}\gamma&=d_Bj^{-1}\gamma=d_B(\beta'-\beta)
=(j^0g'^0-f^0)-(j^0g^0-f^0)=j^0(g'^0-g^0), \textrm{ and }\\
hd_A^{-1}\gamma &=-\pi^{-1}j^{-1}\gamma=\pi^{-1}(\beta-\beta')
=(\alpha-hg^0)-(\alpha-hg'^0)=h(g'^0-g^0)\, .
\end{align*}
Now let us prove (i). Let $(f^{\bullet},\alpha)$ be a pair as in (i). 
The morphism $(f^0,\alpha) : D^0\to B^0\oplus C^{-1}$ factorizes through 
$\Ker d_A^0$ because $d_C\alpha=-\pi^0f^0$ (since $\alpha$ is a homotopy from 
$\pi^{\bullet}f^{\bullet}$ to 0). Then the pair $(g^{\bullet},\beta)$ 
defined by $g^{-1}=f^{-1}, g^0=(f^0,\alpha)$ and $\beta=0$ has the 
required properties.
This finishes the proof of (1).

Conversely, assume (1). Let $C_{\pi}$ denote the mapping cone of $\pi^\flat : 
B^\flat \to C^\flat$. Then $H^0(C_{\pi})$ is the cokernel of $H^0(B)\to 
H^0(C)$, which is zero since $\pi : B\to C$ is an epimorphism. Hence $C_{\pi}$ 
is concentrated in degrees $[-2,-1]$, so that $C_{\pi}[-1]$ belongs to 
$D^{[-1,0]}(S,\Z)$ and we can consider its associated commutative group stack 
$\ch(C_{\pi}[-1])$. Since we have an exact triangle
\[
 C_{\pi}[-1]\to B\to C\to C_{\pi}
\]
by the above we know that $\ch(C_{\pi}[-1])$ is the kernel of $\pi : B\to C$, 
hence it is isomorphic to $A$. The result follows.
\end{proof}

\begin{defi}
\label{def_suites_exactes}
\begin{enumerate}
 \item A sequence of commutative group stacks $0\to A \to B \to C\to 0$ is said 
to be 
\emph{exact} if the equivalent conditions of~\ref{prop:suites_exactes} hold.
\item The sequence is called super-exact if moreover the connecting 
homomorphism 
from $H^{-1}(C)$ to $H^0(A)$ is zero, i.e. if both sequences of group sheaves
$$0\lto H^i(A) \lto H^i(B) \lto H^i(C)\lto 0$$ are exact ($i=-1, 0$).
\end{enumerate}
\end{defi}

\begin{remarque}\rm
If $A, B, C$ are sheaves of groups, the definition coincides with the usual 
notion of an exact sequence. However in the general case, be careful that the 
morphism $A\to B$ does not need to be a 
monomorphism for the sequence to be exact. For instance if $G$ is any (nonzero) 
sheaf then the kernel of the 
trivial morphism $S\to BG$ is $G\to S$ and the sequence $G\to S\to BG$ is 
exact. Similarly the kernel of the identity $BG\to BG$ is the trivial morphism 
$S\to BG$, which is not a monomorphism either (in this latter example, the 
sequence $S\to BG\to BG$ is even super-exact).
Note also that the exactness of both $H^i(\bullet)$ sequences does 
not imply the exactness of the sequence $0\to A \to B\to C\to 0$.
\details{
Par exemple, ....
} It does if you assume moreover that $H^{-1}(C)=0$, in which case the sequence 
is actually super-exact. This will often be the case in that paper.
\end{remarque}

\begin{exemple}\rm
\label{structural_exact_sequence}
 For any commutative group stack $G$ over $S$, there is a canonical short exact sequence:
$$\xymatrix{0\ar[r]& BH^{-1}(G) \ar[r]^-j& G  \ar[r]^-{\pi}& H^0(G)  \ar[r]& 0\, ,}$$
corresponding to the exact triangle $0\to H^{-1}(G)[1]\to G^{\flat} \to H^0(G) 
\to 0.$
\end{exemple}

\begin{prop}[abelian stacks]
\label{prop_abelian_stacks}
 Let $G$ be a commutative group stack over a base scheme~$S$. The following are equivalent:
\begin{enumerate}
 \item $G$ is algebraic, proper, flat, and of finite presentation, with 
connected and reduced (geometric) fibers. Its inertia stack $I_G$ is finite, 
flat and of finite presentation over~$G$.
\item $H^{-1}(G)$ is a finite, flat and finitely presented group scheme. $H^0(G)$ is an abelian scheme.
\end{enumerate}
\end{prop}
\begin{proof}
By definition of $I_G$ and $H^{-1}(G)$, there is a commutative diagram with 
cartesian squares
\[
 \xymatrix{
H^{-1}(G) \cartesien \ar[r] \ar[d] 
& I_G\cartesien \ar[r] \ar[d] & G \ar[d]^{\Delta} \\
S\ar[r] &G \ar[r]^-{\Delta} & G\times_S G
}
\]
where the morphism $S\to G$ corresponds to some neutral element $e$ of $G$. 
This proves that the condition on $I_G$ in (1) implies the 
condition on~$H^{-1}(G)$ in~(2).
Conversely, assume that $H^{-1}(G)$ is finite, flat and finitely presented.  
Then for any object $g$ of $G$, the ``translation by $g$'' isomorphism $\mu_g : 
G\to G$ induces an isomorphism of groups schemes $\Aut_G(e)\to 
\Aut_G(g)$. Since $\Aut_G(e)=H^{-1}(G)$, this proves that $I_G$ is finite, 
flat and of finite presentation over~$G$.
Hence the condition on $I_G$ in (1) is equivalent to the 
condition on~$H^{-1}(G)$ in~(2).

Under this condition, 
by~\cite[10.8]{LMB}, the coarse moduli sheaf~$H^0(G)$ is 
an algebraic space and the morphism $\pi : G\to H^0(G)$ is 
faithfully flat and locally of finite presentation. So $G$ 
is proper, flat, and of finite presentation if and only if 
the same holds for~$H^0(G)$.%
\details{La platitude est locale à la source pour fppf, de même que l'aspect ``localement de présentation finie''. Comme $\pi$ est une gerbe, c'est un homéomorphisme universel. Le diagramme commutatif
$$\xymatrix{
G\ar[r]^{\Delta_G} \ar[d] & G\times_S G
\ar[d] \\
H^0(G)\ar[r]_{\Delta_{H^0(G)}} & H^0(G)\times_S H^0(G)
}$$
montre que $\Delta_{H^0(G)}$ est universellement fermé si et 
seulement si $\Delta_G$ l'est. Donc $f : H^0(G) \to S$ est 
séparé (\cite[7.7]{LMB}) ssi $G$ l'est. Comme $\pi$ est 
universellement un homéomorphisme, $G$ quasi-compact 
équivaut à $H^0(G)$ quasi-compact. De même pour 
universellement fermé.}
It remains to prove that the fibers of $H^0(G)$ are 
connected and reduced if and only if those of~$G$ are. This 
is true for the connectedness since $\pi : G\to H^0(G)$ 
induces a homeomorphism of the underlying topological 
spaces. If the fibers of $G$ are reduced, 
by~\cite{EGA}~EGA~IV$_2$~2.1.13 so are those of~$H^0(G)$. 
Conversely assume~(2) and let us prove that the geometric 
fibers of~$G$ are reduced. By~\cite{EGA}~EGA~IV$_2$~6.6.1, 
it suffices to prove that the geometric fibers of~$\pi$ are 
reduced. This follows from the fact that, if~$F$ is a finite 
commutative group scheme over a field, then the stack~$BF$ is reduced. 
(It is algebraic by~\cite[10.13.1]{LMB}. Then use 
\cite{EGA} EGA IV$_2$ 2.1.13.)
\details{En effet les fibres géométriques de $\pi$ sont 
toutes de cette forme, avec $F=H^{-1}(G)$. Les gerbes sont 
bien neutres car on est sur un corps algébriquement clos 
donc elles sont des sections. Montrons maintenant 
l'assertion sur $BF$, plus précisément on montre que si $S$ 
est réduit, et si $F$ est un groupe séparé, plat et de 
présentation finie sur $S$, alors $BF$ est réduit. On a un 
morphisme canonique $S\to BF$ qui est fidèlement plat et de 
présentation finie. On sait aussi que $BF$ est algébrique 
(\cite[10.13.1]{LMB}). Soit $U\to BF$ une présentation 
lisse de $BF$. Il faut montrer que $U$ est réduit (c'est la 
définition). On regarde le carré cartésien
$$\xymatrix{
U\times_{BF} S\ar[r] \ar[d] &S
\ar[d] \\
U\ar[r] & BF\, .
}$$
Comme $S$ est réduit et $U\times_{BF} S \to S$ lisse, $U\times_{BF} S$ est réduit. Comme $U\times_{BF} S \to U$ est fidèlement plat, par EGA IV 2.1.13 on en déduit que $U$ est réduit.}
\end{proof}

\begin{defi}
\label{def_abelian_stacks}
 An abelian stack is a commutative group stack that satisfies the equivalent conditions of~\ref{prop_abelian_stacks}.
\end{defi}

\begin{prop}[duabelian groups]
\label{prop_duabelian_groups}
 Let $G$ be a sheaf of commutative groups over a base scheme~$S$. The following are equivalent:
\begin{enumerate}
 \item $G$ is a proper, flat, cohomologically flat and finitely presented algebraic space.
\item $G$ is an extension of a finite, flat and finitely presented group scheme~$F$ by an abelian scheme~$A$. In particular it is a scheme.
\end{enumerate}
\end{prop}
\begin{proof}
 Assume~(1) and let us prove~(2). Let $f : G\to S$ be the structural morphism. Then the~$\Oc_S$-module $f_*\Oc_G$ is of finite type and flat (because $f$ is proper and cohomologically flat). Hence $G_{\textrm{af}}:=\Spec(f_*\Oc_G)$ is finite and flat.%
\details{Pour montrer que $f_*\Oc_G$ est plat, voici une ébauche de preuve. (Mais ça semble bien connu de toute manière.) Par passage à la limite, on peut supposer $S$ noethérien. Par un résultat de EGA 4 dont je n'ai pas noté la référence (serait-ce 11.8.1 ?), OPS que $S$ est un trait. Puis on peut citer Raynaud, ``Faisceaux amples...'' VII 3.2.}
 By~SGA~3~\cite[VI$_B$~11.3.1]{SGA3_new} (beware that we use the new edition of 
SGA~3), there is a group 
structure on~$G_{\textrm{af}}$ such that the canonical 
morphism $\rho : G\to G_{\textrm{af}}$ is a homomorphism. 
This latter morphism~$\rho$ is finitely presented because 
$G$ is. It is moreover faithfully flat: this can be checked 
on the fibers (fiberwise flatness criterion) and over a 
field it follows from~SGA~3~\cite[VI$_B$~12.2]{SGA3_new}. We use 
here the fact that, since $f$ is cohomologically flat, the 
formation of $G_{\textrm{af}}$ commutes with any base 
change. Let~$N$ denote the kernel of $\rho$. Since 
$G_{\textrm{af}}$ is separated the inclusion $N\to G$ is a 
closed immersion so $N$ is proper. It is faithfully flat and 
finitely presented because $\rho$ is, and it has smooth and 
connected fibers by~SGA~3~\cite[VI$_B$~12.2]{SGA3_new}. Hence 
$N$ is an abelian algebraic space. It remains to prove 
that~$G$ is a scheme. By a theorem of Raynaud, we already 
know that $N$ is a scheme 
(see~\cite[I, Theorem 1.9]{Chai_Faltings_Degeneration}). 
Since the question is Zariski-local on $S$, we may assume 
that $G_{\textrm{af}}$ is free of rank~$n$. It is then 
killed by~$n$, and the $n$-power map $n : G\to G$ yields an 
\emph{fppf} epimorphism $G\to N$ whose kernel $F$ is finite. 
In particular $G\to N$ is finite, hence~$G$ is a scheme and 
this finishes the proof of~(2).

Conversely assume~(2). Let $\pi : G\to F$ be the given 
\emph{fppf} epimorphism with kernel~$A$. Then~ $G$ is an 
$A$-torsor over~$F$, so by descent it is a proper, flat and 
finitely presented algebraic space (over $F$, hence also 
over~$S$). Since $F$ is cohomologically flat (because it is affine) and 
$\Oc_F\to 
\pi_*\Oc_G$ is universally an isomorphism (by descent, and 
because an abelian scheme has this property) it follows 
that~$G$ is cohomologically flat over~$S$, whence~(1).
\end{proof}

\begin{defi}
\label{def_duabelian_groups}
 A duabelian group is a sheaf of commutative groups that 
satisfies the equivalent conditions 
of~\ref{prop_duabelian_groups}.
\end{defi}

We will see later (\ref{thm_representabilite4}) that 
duabelian groups are precisely the duals of abelian stacks.

\begin{prop}
\label{prop:E0_E1_des_duabeliens}
Let $G$ be a duabelian group scheme over a base scheme $S$. 
Then~$G^D$ is finite, flat and finitely presented and $\shExt^1(G, 
\gm)$ is an abelian scheme.
\end{prop}
\begin{proof}
 There is an exact sequence
$\xymatrix@C=1pc{0 \ar[r] &A \ar[r]^i &G \ar[r] & F \ar[r]& 
0}$
where~$A$ is an abelian scheme and $F$ is a finite and flat 
group scheme. The statement is local so we may assume 
that $F$ is free of constant rank $n$. Since $A^D=0$ this 
sequence induces an isomorphism $F^D\simeq G^D$ hence 
$G^D$ 
is a finite and flat group scheme. Moreover, since 
$E^1(F)=0$ by~%
\ref{vanishing_ext_finite_or_multiplicative}, the morphism
$E^1(i) : E^1(G)\to E^1(A)$ is a monomorphism. On the other 
hand, since $F$ is killed by $n$ by \cite[\S 
1]{Oort_Tate_Group_Schemes}, the $n$-power map 
$[n]_G : G\to G$ induces a morphism
$\pi : G\to A$. The composition $\pi\circ i$ is the
\emph{fppf} epimorphism $[n]_A : A\to A$. Its kernel $_nA$ 
is a finite and flat group scheme over $S$ hence 
$E^1(_nA)=0$ and $E^1([n]_A) : E^1(A) \to E^1(A)$ is an 
epimorphism. Since $E^1([n]_A)=E^1(i)\circ E^1(\pi)$ this 
proves that $E^1(i)$ is an epimorphism. Now $E^1(i)$ is 
both a monomorphism and an epimorphism of \emph{fppf} 
sheaves, hence it is an isomorphism and this proves that 
$E^1(G)$ is an abelian scheme.
\end{proof}
\details{vieille preuve du corollaire :

Now we return to a general base scheme and 
assume that $G$ is cohomologically flat. Then forming 
$G_{\text{af}}$ commutes to \emph{any} base change so $\rho$ 
induces an isomorphism~$G_{\text{af}}^D\simeq G^D$. (Note 
that  in this case~$f_*\Oc_G$ is automatically flat.) This 
yields the flatness of $G^D$. By Artin's theorem 
\cite[10.8]{LMB}, $E^1(G)$, which 
by~\ref{particular_case_of_a_sheaf} is the coarse 
\emph{fppf} sheaf associated with $D(G)$, is an algebraic 
space locally of finite presentation. With the assumption 
that $2\in \Oc_S^{\times}$ we know moreover that $E^1(G)$ is 
flat and quasi-compact (because so is $D(G)$) and that its 
fibers are abelian varieties. It remains to prove 
that~$E^1(G)$ is proper (then it is an abelian scheme). For 
this question we can assume that~$S$ is the spectrum of a 
discrete valuation ring. But then by 
SGA~3~\cite[VI$_B$~12.10]{SGA3_new} the morphism $\rho : G\to 
G_{\text{af}}$ is faithfully flat and of finite 
presentation. Since $G$ is cohomologically flat we see that 
the kernel $N$ of $\rho$ is an abelian scheme. Then the 
exact sequence
$$
0\lto N\lto G\lto G_{\text{af}}\lto 0 
$$
induces an isomorphism $E^1(G)\to E^1(N)$ and this proves 
that $E^1(G)$ is proper.}

\section{Dual of a commutative group stack}
\label{section_dual}

\begin{defi}
\label{def_dual_cgs}
 Let $G$ be a commutative group stack over a base scheme 
$S$. We define its dual to be the stack of homomorphisms of 
commutative group stacks (\ref{def_hom_gp_stacks}) from $G$ to $B\gm$.
$$D(G)=\shHom(G, B\gm).$$
\end{defi}

\begin{remarque}\rm
\label{dual_dun_produit_de_morphismes}
If $\varphi : G\to H$ is a homomorphism of commutative group 
stacks, the composition with $\varphi$ defines a 
homomorphism $D(\varphi) : D(H)\to D(G)$. 
If $\varphi'$ is another such homomorphism, any 
2-isomorphism $\alpha : \varphi \Rightarrow \varphi'$ 
induces a 2-isomorphism $D(\alpha) : D(\varphi) \Rightarrow 
D(\varphi')$. This makes $D(.)$ a strict 2-functor from the 
2-category of commutative group stacks to itself. Note that 
$D(.)$ is additive on maps in the following sense. If 
$\varphi_1, \varphi_2$ are homomorphisms from $G$ to $H$, then 
there is a functorial isomorphism $D(\varphi_1+\varphi_2) 
\simeq D(\varphi_1)+D(\varphi_2)$ of additive homomorphisms 
from $D(H)$ to $D(G)$ in $\Hom_{cgs}(D(H), D(G))$.
\end{remarque}

\begin{remarque}\rm Forming the dual of a commutative group stack commutes with 
base change: for $G$ as above and for any morphism of schemes $S'\to S$, there 
is a canonical isomorphism 
$D(G\times_S S') \simeq D(G)\times_S S'$.
\end{remarque}

The next lemma gives a description of the group sheaves $H^{-1}(D(G))$ and 
$H^0(D(G))$ attached to the dual $D(G)$ in terms of those attached to $G$.

\begin{lem}%
\label{description_dual}
 Let $G$ be a commutative group stack. Then:
\begin{itemize}
 \item[a)] $H^{-1}(D(G))\simeq H^0(G)^D$ ;
\item[b)] There is an exact sequence:
$$0\to E^1(H^0(G)) \to H^0(D(G)) \to H^{-1}(G)^D \to 
E^2(H^0(G))\, .$$
\end{itemize}
(Recall that $G^D$ denotes the sheaf $\Hom(G,\gm)$ of group 
homomorphisms and $E^i(G)$ denotes the \emph{fppf} sheaf 
$\fExt^i(G,\gm)$.)
\end{lem}
\begin{proof}
By~\ref{def_hom_gp_stacks}, $D(G)^{\flat}\simeq \tau_{\leq 0}\RHom(G^{\flat}, 
\gm[1])$. Hence for $i=-1$ or $i=0$, $H^i(D(G))\simeq 
\shExt^{i+1}(G^{\flat},\gm)$. For $i=-1$ this yields $H^{-1}(D(G))\simeq 
\Hom(H^0(G),\gm)$, whence a). For $i=0$ we get $H^0(D(G))\simeq 
\shExt^{1}(G^{\flat},\gm)$. 
 There is a spectral sequence (see e.g.~\cite[Tag 
07A9]{Stacks_Project}):
\[
 E_2^{p,q}=E^p(H^{-q}(G^\flat))\Rightarrow \shExt^{p+q}(G^{\flat},\gm).
\]
The low degree exact sequence associated to this spectral sequence yields b).
\end{proof}

\begin{cor}
\label{particular_case_of_a_classifying_stack}
 Let $G$ be a sheaf of commutative groups on $S$. Then there is a canonical 
isomorphism of commutative group stacks
$$\xymatrix{D(BG)\ar[r]^-{\sim}& G^D}$$
that maps a point $\varphi : BG \to B\gm$ of $D(BG)$ to the point 
$H^{-1}(\varphi) : G \to \gm$ of $G^D$ (using the canonical isomorphisms 
$H^{-1}(BG)\simeq G$ and $H^{-1}(B\gm)\simeq \gm$).
\end{cor}
\begin{proof}
 We have $H^0(BG)=0$ and $H^{-1}(BG)=G$ hence, by the lemma, the exact sequence~\ref{structural_exact_sequence} for~$D(BG)$ reduces to the desired isomorphism. 
\end{proof}

\begin{cor}
\label{particular_case_of_a_sheaf}
 Let $G$ be a sheaf of commutative groups on $S$ and regard it as a 
commutative group stack. Then~$H^{-1}(D(G))\simeq G^D$ and~$H^0(D(G))\simeq 
E^1(G)$. In particular there is a canonical homomorphism
$$BG^D\lto D(G)\, .$$
This is an isomorphism if and only if the sheaf $E^1(G)$ is zero. Note that 
this morphism can be described as follows: if $T$ is a $G^D$-torsor, its image 
in $D(G)$ is the morphism $G\to B\gm$ that maps a point $g$ of $G$ to the 
$\gm$-torsor $c(g)_*T$ where $c(g) : G^D\to \gm$ is the evaluation at $g$.
\end{cor}
\begin{proof}
 This is a consequence of~\ref{description_dual}. The 
canonical morphism is the morphism~$j$ 
of~\ref{structural_exact_sequence}.
\end{proof}

\begin{exemple}\rm
\label{Cartier_groups}
Recall that if $G$ is a (commutative) constant group scheme, we say that it is 
finitely generated if the ordinary group that 
defines the constant group scheme is a finitely generated 
abelian group. This is not equivalent to $G$ being of finite 
type. We say that a sheaf of abelian groups $G$ is a finitely generated twisted 
constant group if, \emph{fppf}-locally on~$S$, it is a finitely generated 
constant group scheme.
Let $G$ be a sheaf of abelian groups that, \emph{fppf}-locally on $S$, is 
built up by successive extensions from diagonalizable groups 
of finite type, finitely generated commutative constant 
groups, and finite locally free commutative 
group schemes. 
Such a group will be called a \emph{Cartier group scheme} 
(or we will say that $G$ is Cartier). Recall that 
groups of multiplicative type and of finite type, and 
finitely generated twisted constant groups, are actually 
étale-locally trivial (SGA~3~\cite[X, 4.5 and 5.9]{SGA3_2}). By induction on 
the 
number 
of extensions, we deduce 
from~\ref{vanishing_ext_finite_or_multiplicative} 
and~\ref{vanishing_ext_constant} that $E^1(G)=0$ and that 
the Cartier dual $G^D$ is still Cartier. (See \cite[X, \S 5]{SGA3_2} for the 
duality theory of groups of multiplicative type and of twisted constant 
groups, and see \cite[\S 4]{Shatz_Group_schemes} for Cartier duality of finite 
locally free group schemes.) In particular, 
by~\ref{particular_case_of_a_sheaf} we see that $D(G)\simeq 
BG^D$. On the other hand, if $A$ is an abelian scheme over 
$S$, then $D(A)\simeq A^t$, the classical dual of $A$ as an 
abelian scheme (use~\ref{dual_and_ext1_abelian}).
\end{exemple}

\begin{exemple}\rm
\label{dual_motif}
If $G$ is a 1-motive, that is, if $G^{\flat}$ is quasi-isomorphic to a complex of the form $[G^{-1}\to G^0]$ where $G^{-1}$ is a twisted lattice and $G^0$ is a semi-abelian variety, then $D(G)$ is the classical dual of $G$ as a 1-motive, as described in~\cite{Deligne_Theorie_Hodge_3}. More generally, if $G^{-1}$ is a sheaf such that $E^1(G^{-1})=0$, and $G^0$ fits into an exact sequence $0\to F\to G^0\to A\to 0$ where $A$ is an abelian scheme and $E^1(F)=0$, then $D(G)^{\flat}$ is quasi-isomorphic to a complex $[F^D\to \widetilde{G^0}]$ where~$\widetilde{G^0}$ fits into an exact sequence $0\to {G^{-1}}^D \to \widetilde{G^0} \to A^t\to 0$. We will not need this fact in the sequel, except in the particular case where $G^{-1}=0$. For the convenience of the reader we include a short proof  in this case below.
\end{exemple}

\begin{lem}
\label{dual_comme_quotient}
 Let $G$ be a sheaf of abelian groups over a base scheme $S$. Assume that there is an exact sequence:
$$0\lto F \lto G \lto A\lto 0$$
where $A$ is an abelian scheme over~$S$ and $E^1(F)=0$ . Let $\delta : F^D \to A^t$ be the natural map given by the $\Hom(\, .\,, \gm)$ sequence. Then the stack $D(G)$ is naturally isomorphic to the quotient stack $[A^t/F^D]$ where $F^D$ acts on $A^t$ \emph{via} $\delta$.
\end{lem}
\begin{proof}
By~\ref{description_cgs} it suffices to prove that 
$D(G)^{\flat}\simeq [F^D \to A^t]$ in the derived category 
$D^{[-1, 0]}(S, \Z)$. By \cite[XVIII~1.4.18]{SGA4}, 
$D(G)^{\flat}\simeq\tau_{\leq 0}\RHom(G, \gm[1])$. Viewing 
the given exact sequence as a triangle in $D^{[-1, 0]}(S, 
\Z)$ and applying the functor $\RHom(\, .\,, \gm[1])$, we 
get a triangle:
$$\RHom(A, \gm[1]) \to \RHom(G, \gm[1])\to \RHom(F, \gm[1])\to \RHom(A, \gm[2])\, .$$
Let $C=\RHom(F, \gm[1])$. Since, by~\ref{particular_case_of_a_sheaf}, 
$H^0(C)=E^1(F)=0$, we can truncate the above triangle in degrees $\leq 0$ 
using the Lemma~\ref{lemme_troncation_triangle} below. Since moreover the 
truncations of $\RHom(A, \gm[1])$ and $\RHom(F, \gm[1])$ are $A^t$ and $F^D[1]$ 
we get a triangle:
$$\xymatrix{
F^D \ar[r]^{\delta} & A^t \ar[r]& \tau_{\leq 0}\RHom(G, \gm[1]) \ar[r] & F^D[1] \, .
}$$
This proves that $\tau_{\leq 0}\RHom(G, \gm[1])$ is isomorphic to the cone of $\delta$, which is precisely the complex $[F^D \to A^t]$.
\details{utiliser les axiomes ($\Delta$ 3)'' et ($\Delta$ 3)* des catégories triangulées, notes de Lipman sur la dualité de Grothendieck, p. 13.}
\end{proof}

\begin{lem}
 \label{lemme_troncation_triangle}
 Let $\Ac$ be an abelian category and let
$\xymatrix{A\ar[r]^u &  B\ar[r]^v & C 
\ar[r]& A[1]}$
be a triangle 
in $D(\Ac)$. Then $H^n(B)\to H^n(C)$ is surjective (\emph{i.e.} its 
cokernel is zero) if and only if there exists a morphism  $w : \tau_{\leq n}C 
\to (\tau_{\leq n}A)[1]$ such that
\[
  \tau_{\leq n}A \lto \tau_{\leq n}B
\lto \tau_{\leq n}C \lto (\tau_{\leq n}A)[1]
\]
is an exact triangle.
\end{lem}
\begin{proof}
If we have an exact triangle as above, the associated long exact sequence 
proves that $H^n(B)\to H^n(C)$ is surjective. Conversely, assume that $H^n(B)\to 
H^n(C)$ is surjective and let us construct the desired $w$. For this we may 
assume that $C$ is the mapping cone of $u$, i.e. $C^i=B^i\oplus A^{i+1}$ for 
all $i$. Let $C_{\tau_{\leq n}u}$ denote the mapping cone of $\tau_{\leq 
n}u$. There is a natural morphism $\alpha : C_{\tau_{\leq n}u} \to \tau_{\leq 
n} C$ given in degree $n-1$ (resp. $n$) by the inclusion $B^{n-1}\oplus 
\Ker(A^n\to A^{n+1}) \inj B^{n-1}\oplus A^n$ (resp. $\Ker(B^n\to B^{n+1})\inj 
\Ker(C^n\to C^{n+1})$). Then $H^i(\alpha)$ is an isomorphism for all $i\neq n$, 
and $H^n(\alpha)$ identifies with the morphism $\Coker(H^n(A)\to H^n(B))\inj 
H^n(C)$. Since $H^n(B)\to H^n(C)$ is surjective, $H^n(\alpha)$ is an 
isomorphism hence $\alpha$ is a quasi-isomorphism and this concludes the proof.
\end{proof}

\begin{prop}[Raynaud]
\label{devissage_raynaud}
 Let $S$ be the spectrum of an Artin local ring with algebraically closed 
residue field $k$, and let $G$ be a proper and flat commutative group scheme 
over $S$. Then there is an exact sequence
$$0\lto F \lto G \lto A\lto 0$$
where $A$ is an abelian scheme and $F$ is a finite flat group scheme.
\end{prop}
\begin{proof}
We first assume that the field $k$ has characteristic $p>0$. 
Let us denote by $S_0$ the spectrum of $k$ and by~$G_0$ the 
reduction of $G$ to $S_0$. Since $k$ is algebraically 
closed, by SGA~3~\cite[VI$_A$~5.5.1 and~5.6.1]{SGA3_new}, 
$G_0$ is the extension of a finite $k$-group $M_0$ by an 
abelian scheme $A_0$. Let $n$ be an integer that kills 
$M_0$. Then the $n$-power map $[n] : G_0\to G_0$ factorizes 
through $A_0$ and yields a morphism $u_0 : G_0 \to A_0$. 
Moreover the composition $A_0 \to G_0 \to A_0$ is the 
isogeny $[n]$ of $A_0$, hence it is finite, flat and 
surjective. This proves that $u_0$ is an \emph{fppf} 
epimorphism. Let $F_0$ be its kernel. Let $K$ be the kernel 
of $[n]: A_0\to A_0$. Then $K$ is also the kernel of the 
composition $F_0\to G_0\to M_0$, which is an \emph{fppf} 
epimorphism. Hence $F_0$ is an extension of $M_0$ by $K$ and 
is thus finite.
\details{Notons $K$ le noyau de $[n]: A_0\to A_0$. C'est aussi le noyau du composé $F_0\to G_0\to M_0$. Donc on a une suite exacte
$$0\lto K\lto F_0\lto M_0$$
et pour montrer que $F_0$ est fini il suffit de montrer que la dernière flèche est un épimorphisme \emph{fppf}. Soit $m$ une section de $M_0$. Il faut la relever fppf-localement à $F_0$. Quitte à localiser, on peut déjà la relever en une section $g$ de $G_0$. On cherche alors une section $a$ de $A_0$ telle que $n(g-a)=0$. Alors $g-a$ sera une section de $F_0$ qui relève $m$. Mais $n(g-a)=0$ équivaut à $na=u(g)$ dans $A_0$ et comme $[n] : A_0\to A_0$ est un épimorphisme fppf le problème a localement une solution.}
So $G_0$ is extension of the abelian scheme $A_0$ by a 
finite group scheme.
$$0\lto F_0\lto G_0 \lto A_0\lto 0$$
  By a theorem of Grothendieck (see~\cite{FGA} for the 
statement without proof, a detailed proof is given 
in~\cite[8.5.23]{FGA_Explained}), there exists an abelian 
scheme~$A$ over~$S$ lifting~$A_0$. Now, the morphism $u_0$ 
does not necessarily lift to a morphism of schemes $G\to A$, 
but it does if we increase the integer $n$ used above. 
Indeed, let $S_0\to S_1 \to \dots \to S_m=S$ be a factorization of the 
morphism $S_0\to S$, where each morphism $S_{i}\to S_{i+1}$ is an extension of 
spectra of Artin local rings defined by a 
square-zero ideal~$I_{i+1}$. We denote by a subscript $i$ the objects above 
$S_i$ obtained after the base change $S_i\to S$ from those above $S$. 
Assume that there exists a morphism of group schemes $u_i : G_i \to A_i$ 
lifting $u_0$. Then 
by~Illusie~\cite[VII~3.3.1.1]{Illusie_CCD2}, there is a 
class $\omega\in H^1(G_i, u_i^*\Omega^1_{A_i/S_i}\otimes 
I_{i+1})$ that vanishes if and only if $u_i$ lifts to a morphism 
of schemes $u_{i+1}$. The obstruction to lift $u_i\circ [p^\ell]$ 
is~$p^\ell.\omega$
\details{Ici le $\Ext^1$ est le $\Ext^1$ au sens des groupes. En particulier, le morphisme obtenu en appliquant le foncteur $\Ext^1(\, .\,, u^*\Omega^1_{A_0/S_0}\otimes I)$ au morphisme $[p^\ell] : G_0\to G_0$ est bien la multiplication par $p^\ell$.}
 hence vanishes for $\ell$ large enough (because the group 
$H^1$ above is actually a $\Gamma(S_i, \Oc_{S_i})$-module). The 
resulting morphism of schemes $u_{i+1} : G_{i+1}\to A_{i+1}$ is not 
necessarily a group morphism, but once again, for $n$ large 
enough it is so. (We can also 
use~SGA~3~\cite[III~2.1]{SGA3_new} instead of Illusie.)
By induction we get a homomorphism $u : G\to A$ lifting $u_0$. 
By~SGA~1~\cite[IV~5.9]{SGA1}~$u$ is automatically flat. It 
is surjective and finite because so is $u_0$. This 
yields the desired exact sequence, with $F$ the kernel 
of~$u$.

If the field $k$ has characteristic zero, this is much 
easier: let $G^0$ denote the connected component of $0$ in $G$. Since $G$ is 
of finite type, $G^0$ is an open and closed subscheme of $G$ hence it is proper 
and flat over $S$. Since the residue field has characteristic zero, the 
special fiber of $G$ is smooth hence $G$ itself is smooth because it is flat 
with smooth fibers. This proves that $G^0$ is an abelian scheme over $S$. By 
SGA~3~\cite[VI$_A$~5.5.1]{SGA3_new}, the quotient $G/G^0$ is finite and étale, 
so that $G$ is the 
extension of a finite étale group by an abelian scheme. 
Using the $n$-power map $[n] : G\to G$ as in the beginning 
of the proof, we directly get the desired exact sequence 
over $S$.
\end{proof}

\begin{cor}
\label{dual_dun_schema_en_groupes_sur_une_base_artinienne}
 Let $S$ be the spectrum of an Artin ring $R$ and let $G$ be a proper and flat 
commutative group scheme over $S$. Then the dual $D(G)$ is proper and flat.
\end{cor}
\begin{proof}
Since an Artin ring is a product of Artin local rings, we may assume that $R$ 
is local. Let us prove that there exists a faithfully flat morphism $S'\to S$ 
where $S'$ is the spectrum of an Artin local ring with algebraically closed 
residue field. We can write $R$ as a quotient of $W[X_1, \dots, X_n]$ 
for some integer $n$ and some discrete valuation ring $W$ (use e.g. Cohen's 
theory). By~\cite[I, 5.5.3]{EGA1} there exists a discrete valuation ring $W'$ 
that dominates $W$ with algebraically closed residue field. Then $R\otimes_W 
W'$ is an Artin local ring with algebraically closed residue field and is 
faithfully flat over $R$.
Now by faithfully flat descent, we may assume that 
the residue field of $R$ is algebraically closed.
 By \ref{devissage_raynaud}, \ref{vanishing_ext_finite_or_multiplicative} and 
\ref{dual_comme_quotient}, 
$D(G)$ is isomorphic to a quotient stack $[A/F]$ where $A$ 
is an abelian scheme acted on by a finite and flat commutative group 
scheme $F$. By \cite[10.13.1]{LMB} the stack $[A/F]$ is 
algebraic. Moreover the canonical morphism $A\to [A/F]$ is 
finite and faithfully flat. It follows that $D(G)$ is proper 
and flat.
\details{
La platitude est locale à la source pour fppf. Pour propre, plus généralement si $p : X \to Y$ est un morphisme de champs algébriques avec $X$ propre, alors :
\begin{itemize}
 \item si $p$ est surjectif, $Y\to S$ est quasi-compact et universellement fermé ;
\item si $p$ est fidèlement plat et localement de présentation finie, $Y\to S$ est localement de type fini ;
\item si $p$ est surjectif et universellement fermé, $Y\to S$ est séparé.
\end{itemize}
}\end{proof}

Let $G$ be a commutative group stack over a base scheme $S$. If we 
identify the category of morphisms of stacks from $G$ to 
$B\gm$ with the category of invertible sheaves on $G$, we 
see that there is a natural forgetful morphism
$$\omega : D(G) \lto \champic(G/S)\, .$$
We can translate Definition~\ref{def_morphismes_groupes} 
in terms of invertible sheaves to give 
an alternative and useful description of the stack $D(G)$. 
It will be helpful to study the 
representability of the stack~$D(G)$, 
see~\ref{thm_representabilite1}. Let us first fix some 
notation. Let $Sym : G\times_S G \to G\times_S G$ be the 
isomorphism that exchanges the two factors, and let $\mu_G : 
G\times_S G \to G$ be the group law in~$G$. Let $p_i$ 
(resp. $q_i$) be the projection of $G\times_S G$ (resp. 
$G\times_S G\times_SG$) onto the $i$-th factor. By 
definition of $G$, there is a 2-isomorphism $\tau : \mu_G 
\Rightarrow \mu_G\circ Sym$ (commutativity) and a 
2-isomorphism of associativity
$$\lambda : \mu_G\circ (\mu_G\times \id_G) \Rightarrow  \mu_G\circ ( \id_G\times\mu_G).$$
If $\Lc$ is an invertible sheaf on $G$, then $\tau$ and $\lambda$ induce isomorphisms of invertible sheaves on $G\times_S G$ (resp. on $G\times_S G\times_S G$)
\begin{eqnarray*}
\tau(\Lc) &:& \mu_G^*\Lc \lto Sym^*\mu_G^* \Lc\\
\lambda(\Lc) &:& (\mu_G\times \id_G)^*\mu_G^*\Lc \lto 
( \id_G\times\mu_G)^*\mu_G^*\Lc\, .
\end{eqnarray*}
Now the description of~$D(G)$ is as follows. For any $S$-scheme $U$, an object of the fiber category $D(G)(U)$ is a couple $(\Lc,\alpha)$ where $\Lc$ is an invertible sheaf on $G\times_S U$ and $\alpha$ is an isomorphism
$$\alpha : \mu_G^*\Lc \lto p_1^*\Lc \otimes p_2^*\Lc$$
such that the two following diagrams commute
$$\xymatrix{\mu_G^*\Lc \ar[r]^-{\alpha} \ar[d]_{\tau(\Lc)} \ar@{}[dr]|{(A)} & p_1^*\Lc\otimes p_2^*\Lc \ar[d]^{\textrm{can.}} \\
Sym^*\mu_G^*\Lc \ar[r]_-{Sym^*\alpha} &
Sym^*(p_1^*\Lc\otimes p_2^*\Lc)}$$

$$\xymatrix@!0@R=2.2pc@C=8pc{
(\mu_G\times \id_G)^*\mu_G^*\Lc \ar[rr]^{\lambda(\Lc)}
    \ar[d]_{(\mu_G\times \id_G)^*\alpha}&&
(\id_G\times\mu_G)^*\mu_G^*\Lc
    \ar[d]^{(\id_G\times\mu_G)^*\alpha} \\
(\mu_G\times \id_G)^*(p_1^*\Lc \otimes p_2^*\Lc)
    \ar[d]_{\textrm{can.}}^{\wr} \ar@{}[rrd]|{(B)}&&
(\id_G\times\mu_G)^*(p_1^*\Lc \otimes p_2^*\Lc)
    \ar[d]^{\textrm{can.}}_{\wr} \\
((q_1\times q_2)^*\mu_G^*\Lc)\otimes q_3^*\Lc
    \ar[d]_{((q_1\times q_2)^*\alpha)\otimes \id_{q_3^*\Lc}} &&
q_1^*\Lc\otimes (q_2\times q_3)^*\mu_G^*\Lc
    \ar[d]^{\id_{q_1^*\Lc}\otimes(q_2\times q_3)^*\alpha } \\
(q_1\times q_2)^*(p_1^*\Lc\otimes p_2^*\Lc)\otimes q_3^*\Lc
\ar[rd]_{\textrm{can.}}
&&q_1^*\Lc\otimes (q_2\times q_3)^*(p_1^*\Lc \otimes p_2^*\Lc)\ar[ld]^{\textrm{can.}}
\\
&q_1^*\Lc\otimes q_2^*\Lc \otimes q_3^*\Lc&
}$$
If $x=(\Lc, \alpha)$ and $x'=(\Lc', \alpha')$ are two such objects, a morphism $x\to x'$ is an isomorphism $\beta : \Lc \to \Lc'$ such that the diagram
$$\xymatrix{\mu_G^*\Lc \ar[r]^-{\alpha} \ar[d]_{\mu_G^*\beta} &
p_1^*\Lc\otimes p_2^*\Lc \ar[d]^{(p_1^*\beta)\otimes (p_2^*\beta)}\\
\mu_G^*\Lc' \ar[r]_-{\alpha'} &
p_1^*\Lc'\otimes p_2^*\Lc'
}$$
commutes.

\begin{remarque}\rm
 This description also shows that $D(G)$ is the stack of 
extensions of $G$ by $\gm$. To see this, use \emph{e.g.}~%
\cite[I, 2.3.10]{LMB_Pinceaux_varietes_abeliennes}.
In the language of 
\cite{LMB_Pinceaux_varietes_abeliennes}, an isomorphism 
$\alpha : \mu_G^*\Lc \to p_1^*\Lc\otimes p_2^*\Lc$ 
corresponds to a section $\sigma$ of the $\gm$-torsor 
$\theta_2(\Lc)$ above $G\times_S G$. Then the commutativity 
of the diagram (A) (resp. (B)) is equivalent to the 
condition I, 2.3.8 (resp. I, 2.3.9) of 
\cite{LMB_Pinceaux_varietes_abeliennes}. 
\end{remarque}

\begin{thm}
\label{thm_representabilite1}
Let $S$ be a base scheme. Let $G$ be an algebraic 
commutative group stack which is proper, flat and finitely 
presented over $S$. Then:
\begin{enumerate}
\item The morphism $\omega$ is affine and of finite 
presentation.
\item The stack $D(G)$ is algebraic and of finite 
presentation, with affine diagonal. Its fibers are proper.
\item If $G$ is an algebraic 
space (\emph{i.e.} $H^{-1}(G)=0$), then $D(G)$ is flat.
\end{enumerate}
\details{Comme $D(G)$ est localement de présentation finie, 
la diagonale l'est aussi, donc comme elle est affine elle 
est automatiquement de présentation finie.}
\end{thm}
\begin{proof}
 Let us prove (1). Since $G$ is of finite presentation, 
by standard limit arguments we can assume that $S$ is 
noetherian. Let $U$ be an $S$-scheme and 
$U\to \champic(G/S)$ a morphism, corresponding to an 
invertible sheaf $\Lc$ on $G\times_S U$. 
By~\cite[2.1.1]{Brochard_finiteness} and 
\cite[Thm D]{Hall_Cohomology_and_base} the sheaf 
$I=\shIsom(\mu_G^*\Lc,p_1^*\Lc\otimes p_2^*\Lc)$ is an 
affine scheme of finite presentation over $U$. By the above 
description, we see that the fiber product 
$D(G)\times_{\champic(G/S)} U$ identifies with the closed 
subspace of $I$ defined by the conditions (A) and (B), 
hence 
it is also an 
affine scheme of finite presentation over $U$.

Let us prove (2). By~\cite{Brochard_Picard}, 
\cite[2.1.1]{Brochard_finiteness} and 
\cite[Thm D]{Hall_Cohomology_and_base}, the stack 
$\champic(G/S)$ is algebraic and locally of finite 
presentation, with affine 
diagonal. Hence so is $D(G)$ by (1). Let us now prove that 
$D(G)$ has proper fibers. For this we may assume that $S$ 
is the spectrum of an 
algebraically closed field $k$. If $G$ is a group scheme 
the assertion follows from~%
\ref{dual_dun_schema_en_groupes_sur_une_base_artinienne}.
Now, let $G$ be a proper commutative group 
stack over~$k$. Then $H^{-1}(G)$ and $H^0(G)$ are duabelian 
$k$-group schemes (note that a proper group scheme over the spectrum of a 
field is obviously flat and finitely presented, but also cohomologically flat 
because the formation of direct images commutes with flat base 
change). In particular, 
by~\ref{prop:E0_E1_des_duabeliens},
$H^{-1}(D(G))$ (which is isomorphic to $H^0(G)^D$ by~\ref{description_dual}) 
and $H^{-1}(G)^D$ are 
finite and $E^1(H^0(G))$ is an abelian variety.
By Artin~\cite[10.8]{LMB} the coarse moduli 
sheaf $H^0(D(G))$ is an algebraic space locally of finite 
presentation (hence actually a group scheme) and the 
projection $\pi : D(G) \to H^0(D(G))$ is faithfully flat and 
locally of finite presentation. 
By~\ref{description_dual} there is an exact 
sequence
$$\xymatrix{
0\ar[r] & E^1(H^0(G)) \ar[r]^-{\delta} & H^0(D(G)) \ar[r] & 
H^{-1}(G)^D \ar[r] & E^2(H^0(G)).}$$
By SGA~3~\cite[VI$_A$ 3.2]{SGA3_new} the cokernel $K$ of 
$\delta$ is a group scheme. But it has a monomorphism to 
the finite group scheme $H^{-1}(G)^D$. Thus $K$ is 
finite and this proves that $H^0(D(G))$ is proper. This 
implies that $D(G)$ itself is proper, because it is a gerbe over 
$H^0(D(G))$, banded by $H^{-1}(D(G))$ which is finite.  

To prove (2) it only remains to prove that $D(G)$ 
is quasi-compact for a general base $S$. To this end, we 
use the notion of 
quasicompactness introduced in~\cite{Brochard_finiteness} 
for non necessarily representable stacks (or morphisms of 
stacks). The fibers of~$D(G)$ are proper by the above. In 
particular the morphism $D(G)\to \Pic_{G/S}$ factorizes 
through $\Pic^{\tau}_{G/S}$. 
By~\cite[3.3.3]{Brochard_finiteness}, we know that 
$\Pic^{\tau}_{G/S}$ is quasi-compact over $S$. Moreover the 
morphism $\champic^{\tau}(G/S)\to \Pic^{\tau}_{G/S}$ is an 
\emph{fppf} gerbe hence quasi-compact. To conclude, it 
suffices to prove that the morphism from~$D(G)$ 
to~$\champic^{\tau}(G/S)$ is quasi-compact. 
By~\cite[3.3.3]{Brochard_finiteness} the inclusion 
$\champic^{\tau}(G/S) \to \champic(G/S)$ is an open 
immersion hence its diagonal is quasi-compact. Since 
$\omega$ is quasi-compact, by~\cite[3.1.3 
(vii)]{Brochard_finiteness}  we see that $D(G)$ is 
quasi-compact over~$S$ and this finishes the proof of~(2).


Since the 
flatness of a locally noetherian stack can be checked on 
Artin rings (SGA~1~\cite[IV~5.6]{SGA1}), the assertion~(3) 
follows from~%
\ref{dual_dun_schema_en_groupes_sur_une_base_artinienne}.
\end{proof}

\begin{thm}
\label{thm:inertie_finie}
Let $G$ be an algebraic 
commutative group stack which is proper, flat and finitely 
presented over a base scheme $S$. Assume that $H^{-1}(G)$ is flat.
Then $H^{-1}(D(G))$ is a finite group scheme.
\end{thm}
\begin{proof}
 Since $H^{-1}(G)$ is flat, $H^0(G)$ is 
a proper, flat and finitely presented algebraic space.
Since $H^{-1}(D(G))=H^0(G)^D$ we may assume that $G$ is 
an algebraic space and we have to prove that $G^D$ is 
finite. By~\ref{thm_representabilite1}, we already know that 
$G^D$ is affine and finitely presented, with 
finite fibers. We now prove that it is proper. By 
\cite[A.2.1]{LMB}  this will imply that it is finite. For 
this question we can assume that $S$ is the spectrum of a 
discrete valuation ring $R$ with fraction field~$K$. 
Let~$G_{\text{af}}$ be the spectrum of~$f_*\Oc_G$. By 
\cite[VII~3.2]{Raynaud_Faisceaux_amples_sur} 
and~SGA~3~\cite[VI$_B$~11.3.1]{SGA3_new} it is a finite flat 
group scheme, the canonical morphism $\rho : G \to 
G_{\text{af}}$ is a homomorphism, and it is universal for 
homomorphisms to affine $R$-groups. In particular it induces 
a bijection $G_{\text{af}}^D(R)\to G^D(R)$. Since forming 
$G_{\text{af}}$ commutes to any flat base change, the 
natural map $G_{\text{af}}^D(K)\to G^D(K)$ is also 
bijective. Lastly, $G_{\text{af}}^D$ is finite hence 
satisfies the valuative criterion thus 
$G^D$ is finite.
\end{proof}

\begin{cor}
\label{repres_dual_Cartier}
 Let~$G$ be a proper, flat and finitely presented 
commutative group algebraic space over a base scheme~$S$. 
Then $G^D$ is a finite group scheme.
\end{cor}
\begin{proof}
 This is the particular case of~\ref{thm:inertie_finie} where $H^{-1}(G)=0$.
\end{proof}

\begin{thm}
\label{thm_representabilite4}
Let $S$ be a base scheme and let $G$ be a
commutative group stack over $S$.
\begin{enumerate}
\item If $G$ is a duabelian group, then~$D(G)$ is an 
abelian stack.
\item If $G$ is an abelian stack, then $D(G)$ is a 
duabelian group.
\item Assume that 2 is invertible in $S$. If $H^0(G)$ is 
duabelian and $H^{-1}(G)$ is 
finite and flat, then $H^0(D(G))$ is duabelian 
and $H^{-1}(D(G))$ is finite and flat.
\end{enumerate}
\end{thm}
\begin{proof}
(1) is an immediate consequence of~\ref{prop:E0_E1_des_duabeliens} 
and~\ref{description_dual}. Let us prove (2) and (3). There is an exact 
sequence:
$$\xymatrix{
0\ar[r] & E^1(H^0(G)) \ar[r] & H^0(D(G)) \ar[r] & H^{-1}(G)^D \ar[r] & E^2(H^0(G)).
}$$
The group $H^{-1}(G)^D$ is finite and flat. 
By~\ref{prop:E0_E1_des_duabeliens}, $E^1(H^0(G))$ is an abelian scheme.
Moreover 
by~\ref{vanishing_ext23_abelian} in both cases the morphism 
from $H^{-1}(G)^D$ to $E^2(H^0(G))$ is zero. Hence 
$H^0(D(G))$ is duabelian and this proves (3). If $G$ is an 
abelian stack then $H^0(G)$ is an abelian scheme. 
By~\ref{dual_and_ext1_abelian} its Cartier dual vanishes 
hence $D(G)=H^0(D(G))$ and it is a duabelian group.
%
\end{proof}


It is a natural question to ask whether the above duality operation preserves short exact sequences. In general this is not the case. The following proposition gives a positive answer with suitable assumptions.

\begin{prop}
\label{exactness_of_the_dual_sequence}
 \label{prop:cns_exactness_dual_sequence}
Let $(*) : \xymatrix{
0\ar[r]& A\ar[r]^j &B\ar[r]^{\pi} & C\ar[r] & 0
}$ be a short exact sequence of commutative group stacks.
\begin{enumerate}[label=(\alph*)]
 \item The following are equivalent:
\begin{enumerate}[label=(\roman*)]
 \item The dual sequence $D(*)$ is exact.
\item The morphism of stacks $D(j) : D(B)\to D(A)$ is an epimorphism.
\item The morphism of sheaves $H^{0}(D(B))\to H^0(D(A))$ is an epimorphism.
\end{enumerate}
\item If $E^2(H^0(C))=E^1(H^{-1}(C))=0$, then $D(*)$ is exact.
\item Assume that $(*)$ is super-exact. If $E^2(H^0(C))=0$ and if the morphism 
from $H^{-1}(A)^D$ to $E^1(H^{-1}(C))$ vanishes, then $D(*)$ is exact.
\item Assume that $(*)$ is super-exact and that $D(*)$ is exact. Then $D(*)$ is 
super-exact if and only if the morphism $H^0(A)^D \to E^1(H^0(C))$ is zero.
\end{enumerate}
\end{prop}
\begin{proof}
 (a) The implications $(i)\Rightarrow (ii)\Rightarrow (iii)$ are obvious, and 
$(iii)\Rightarrow (i)$ follows from~\ref{lemme_troncation_triangle}.

(b) The exact triangle
\[\RHom(C^\flat,\gm[1]) \to \RHom(B^\flat,\gm[1]) \to 
\RHom(A^\flat,\gm[1]) \to \RHom(C^\flat,\gm[2])
\]
induces an exact sequence $H^0(D(B))\to H^0(D(A))\to \Ext^2(C^\flat, \gm)$. By
~\cite[Tag 07A9]{Stacks_Project}) there is a spectral sequence:
\[
 E_2^{p,q}=E^p(H^{-q}(C^\flat))\Rightarrow \Ext^{p+q}(C^{\flat},\gm).
\]
Since $C^\flat$ is concentrated in degrees $-1,0$, we have $E_2^{p,q}=0$ for 
$q\geq 2$ and the spectral sequence yields a long exact sequence :
\[
 \dots \to E^2(H^0(C))
\to \Ext^2(C^\flat,\gm) \to E^1(H^{-1}(C))\to \dots
\]
so that $\Ext^2(C^\flat,\gm)=0$.

The assertions (c) and (d) are proven by chasing through 
the commutative diagram
$$\xymatrix@R=0.7pc{
&H^0(A)^D \ar[d] && 0 \ar[d] & E^1(H^0(A))\ar[d] \\
0\ar[r]&E^1(H^0(C)) \ar[d] \ar[r]&
H^0(D(C)) \ar[d] \ar[r]& H^{-1}(C)^D \ar[d] \ar[r]&
E^2(H^0(C))\ar[d]  \\
0\ar[r]&E^1(H^0(B)) \ar[d] \ar[r]&
H^0(D(B)) \ar[d]\ar[r]& H^{-1}(B)^D \ar[d]\ar[r]&
E^2(H^0(B))\ar[d]  \\
0\ar[r]&E^1(H^0(A))\ar[d] \ar[r]&
H^0(D(A)) \ar[r]&H^{-1}(A)^D \ar[d] \ar[r]&
E^2(H^0(A))\ar[d]  \\
& E^2(H^0(C)) && E^1(H^{-1}(C))& E^3(H^0(C)) 
}$$
in which the rows (obtained from Lemma~\ref{description_dual}), and the first, 
third and fourth columns are exact.
\end{proof}

\begin{remarque}\rm
\label{dual_sec_can}
Let $G$ be an arbitrary commutative group stack. 
Apply~\ref{exactness_of_the_dual_sequence} to the canonical short exact 
sequence~\ref{structural_exact_sequence}, i.e. with $A, B$ and $C$ respectively 
equal to $BH^{-1}(G), G$ and $H^0(G)$. Then $H^0(A)=0$ and $H^{-1}(C)=0$. 
By~\ref{exactness_of_the_dual_sequence} a) the dual sequence 
$D(\ref{structural_exact_sequence})$ is 
exact if and only if $H^{0}(D(B))\to H^0(D(A))$ is an epimorphism. Using the 
commutative diagram in the proof of~\ref{exactness_of_the_dual_sequence}, we 
can identify this latter morphism with $H^0(D(B))\to H^{-1}(B)^D$, hence the 
sequence
$$\xymatrix{0\ar[r]& D(H^0(G)) \ar[r]& D(G)  \ar[r]& D(BH^{-1}(G))  \ar[r]& 0}$$
is (super-)exact if and only if the natural morphism 
$H^{-1}(G)^D\to E^2(H^0(G))$ is trivial.
\end{remarque}

We now give three lemmas that will be used to compare the dual $D(G)$ of an abelian stack with the torsion component $\picto_{G/S}$ of the Picard functor (see~\ref{comp_dual_picto}).

\begin{lem}
 \label{lemme_isomorphisme_fibre_a_fibre}
Let $u : G\to H$ be a morphism of commutative group algebraic spaces over a  base scheme $S$. Assume that:
\begin{itemize}
\item[(i)] For any geometric point $s : \Spec k \to S$ of $S$, the induced morphism $u_s$ is an isomorphism.
 \item[(ii)] $G$ is flat and of finite presentation over $S$.
\item[(iii)] $H$ is locally of finite type over $S$.
\end{itemize}
Then $u$ is an isomorphism.
\end{lem}
\begin{proof}
 Note that~SGA~3~\cite{SGA3_new}~VI$_B$~2.10 and~2.11 also hold 
for algebraic spaces. Then by~VI$_B$~2.11 $u$ is a 
monomorphism. Let $Q$ denote the \emph{fppf} quotient sheaf. 
By Artin~\cite[10.4]{LMB} it is an algebraic space locally 
of finite type. Then by~\cite[VI$_B$~2.10]{SGA3_new}, the group 
$Q$ is trivial.
\end{proof}
\details{
Si l'on préfère supposer seulement $u$ quasi-compact, plutôt que $G$, au prix d'une hypothèse noethérienne supplémentaire, c'est possible. On suppose comme ci-dessus que $u$ est un morphisme d'espaces algébriques en groupes. On garde les hypothèses (i) et (iii), et la platitude de $G$. On suppose de plus que $u$ est de présentation finie (c'était automatique avec les hypothèses ci-dessus) et que $H$ est localement noethérien. Montrons alors que $u$ est un isomorphisme. On sait que $u$ est un monomorphisme par le critère fibre à fibre. Pour appliquer le critère fibre à fibre pour les isomorphismes (cf fourre-tout) il suffit de montrer que $u$ est universellement ouvert, ou universellement fermé. On va montrer qu'il est plat (donc universellement ouvert). Pour cette question on peut supposer que la base est un anneau local artinien. Mais dans ce cas vu l'hypothèse sur la fibre il est clair que $u$ est un homéomorphisme universel, donc un isomorphisme d'après le critère fibre à fibre déjà évoqué plus haut, cqfd.
}

\begin{lem}
\label{morphisme_constant}
 Let $S$ be a scheme and let $X$ be an algebraic stack over~$S$ such that $X(S)$ is nonempty and $\Oc_S\to f_*\Oc_X$ is an isomorphism (where $f$ is the structural morphism of $X$). Let~$Y\to S$ be an affine morphism of schemes. Then any morphism~$g~:~X\to Y$ is \emph{constant}, \emph{i.e.} it factorizes through $f$.
\end{lem}
\begin{proof}
 Let $x\in X(S)$. It suffices to prove that the maps $g$ and 
$g\circ x\circ f$ are equal. We can assume that $S$ and $Y$ 
are affine. Now the set $\Hom(X,Y)$ can be identified with 
$\Hom(\Gamma(Y,\Oc_Y), \Gamma(X,\Oc_X))$ and the result 
follows.
\end{proof}

\begin{lem}
\label{lemme_monomorphisme}
 Let $G$ be an algebraic commutative group stack. Assume that $\Oc_S 
\to f_*\Oc_G$ is universally an isomorphism. Then $H^{-1}(D(G))=0$, and the 
natural morphism from $D(G)$ to~$\Pic_{G/S}$ is a monomorphism.
\end{lem}
\begin{proof}
 Let $\varphi : H^0(G)\to \gm$ be a morphism of group 
sheaves. By~\ref{morphisme_constant} the induced morphism 
from $G$ to $\gm$ is trivial and this implies that $\varphi$ 
is trivial. Since the assumptions are stable under base 
change, it follows that $H^0(G)^D=0$, hence 
$H^{-1}(D(G))=0$.

To prove that $D(G)\to \Pic_{G/S}$ is a monomorphism, it 
suffices to prove that the induced map from $D(G)(S)$ to 
$\Pic_{G/S}(S)$ is injective (again because the assumptions 
are stable under base change). Let $\sigma : G\to B\gm$ be 
an $S$-point of $D(G)$, such that the corresponding 
invertible sheaf is mapped to 0 in $\Pic_{G/S}(S)$. 
By~\cite[2.2.6]{Brochard_Picard} it is mapped to 0 in 
$\Pic(G)/\Pic(S)$, which means that the morphism of stacks 
underlying $\sigma$ factorizes through $S$. This in turn 
implies that the morphism of commutative group stacks $\sigma$ is 
trivial. 
\end{proof}

\begin{prop}
\label{comp_dual_picto}
 Let $G$ be an abelian stack over a base scheme $S$. Then 
the forgetful morphism from $D(G)$ to $\champic(G/S)$ 
induces a functorial isomorphism
$$\xymatrix{D(G)\ar[r]^-{\sim} & \picto_{G/S}\, .}$$
\end{prop}
\begin{proof}
 By standard limit arguments we may assume that $S$ is 
noetherian. The structural morphism $f : G \to S$ is proper, 
flat, finitely presented, and with geometrically connected 
and geometrically reduced fibers 
(\ref{prop_abelian_stacks}). In particular, $\Oc_S \to 
f_*\Oc_G$ is universally an isomorphism. Then the Picard 
functor $\Pic_{G/S}$ is representable by a quasiseparated 
algebraic space 
\cite[Theorem~2.1.1~(2)]{Brochard_finiteness} and the 
subfunctor $\picto_{G/S}$ is representable by an open 
subscheme, which is of finite presentation over~$S$ 
\cite[Theorem~3.3.3]{Brochard_finiteness}. On the other 
hand, by~\ref{thm_representabilite4}, $D(G)$ is a proper 
and flat algebraic space over~$S$. The natural morphism 
$\omega : D(G) \to \Pic_{G/S}$ is a monomorphism 
by~\ref{lemme_monomorphisme}. Since $D(G)$ is proper it 
factorizes through~$\picto_{G/S}$. We still denote 
by~$\omega$ the resulting monomorphism $D(G)\to 
\picto_{G/S}$. To prove that it is an isomorphism, 
by~\ref{lemme_isomorphisme_fibre_a_fibre} we may assume 
that~$S$ is the spectrum of an algebraically closed field. 

Let us first consider the case where $G$ is an abelian variety.
Then, by~\ref{particular_case_of_a_sheaf}, we know that $D(G)$ is isomorphic to the abelian variety $\picto_{G/S}$ hence $\omega$ is necessarily an isomorphism since it is injective.

Now let us consider the case where $H^0(G)=0$. Then 
$\Pic_{G/S}\simeq H^{-1}(G)^D$ (see for 
instance~\cite[5.3.7]{Brochard_Picard} 
or~\ref{Picard_classifiant} below) and the whole Picard 
functor is torsion. On the other hand, the commutative group stack 
$D(G)$ is also isomorphic to~$H^{-1}(G)^D$ 
by~\ref{particular_case_of_a_classifying_stack}. The 
morphism $\omega$ is a proper monomorphism, hence a closed 
immersion, and since both sides are finite and isomorphic it 
must be an isomorphism.

In the general case, the canonical exact sequence~\ref{structural_exact_sequence} induces a commutative diagram
$$\xymatrix{0\ar[r] & D(H^0(G)) \ar[r] \ar[d]&
D(G) \ar[r] \ar[d]^{\omega}& D(BH^{-1}(G))\ar[r] \ar[d]& 0 \\
0\ar[r]& \picto_{H^0(G)/S} \ar[r]& \picto_{G/S}  \ar[r] & \Pic_{BH^{-1}(G)/S}\, .}$$
The morphism $H^{-1}(G)^D \to E^2(H^0(G))$ is trivial 
(\ref{vanishing_ext23_abelian}), hence the first row is exact 
by~\ref{dual_sec_can}. The second row is exact too (\ref{Picard_classifiant} 
and~\cite[3.3.2]{Brochard_finiteness}). The left and right vertical maps are 
isomorphisms by the above particular cases, hence by the 5-lemma (which holds 
in any abelian category and in particular in the category of sheaves, see e.g. 
\cite[VIII, \S 4, Lemma4]{MacLane_Categories_For_The}; note by the way that in 
loc. cit. the assumptions on $f_1$ and $f_5$ are too strong: it suffices 
to have $f_1$ epic and $f_5$ monic) the middle one is also an isomorphism.
\end{proof}


\begin{prop}
 \label{Picard_classifiant}\ 

\begin{enumerate}
 \item Let $F$ be a sheaf of commutative groups over a base scheme~$S$. There is a canonical isomorphism
$\xymatrix@C=1pc{
\Pic_{BF/S} \ar[r]^-{\sim} & F^D\, .
}$
\item Let $F$ be a separated, flat and finitely presented commutative group 
algebraic space over~$S$, and let~$\Gc$ be an~$F$-gerbe (see~\ref{def:gerbe}) 
over an~$S$-scheme $X$. There is an exact sequence:
$$0\lto \Pic_{X/S} \lto \Pic_{\Gc/S} \lto F^D\, .$$
\end{enumerate}

\end{prop}
\begin{proof}
 (1) is proved in~\cite[5.3.7]{Brochard_Picard} (the sheaf 
$F$ was supposed to be a scheme in \emph{loc. cit.} but this 
assumption was useless). The isomorphism maps an invertible 
sheaf~$L$ on~$BF$ to the unique character~$\chi_{L} : F\to 
\gm$ such that the natural action of $F$ on $L$ is induced 
through $\chi_{L}$ by that of~$\gm$. The inverse of this 
isomorphism maps a character $\chi$ to the class of the 
invertible sheaf~$\Lc(\chi)$ corresponding to the induced 
map $BF\to B\gm$.
\details{
Ceci est encore vrai pour une gerbe non neutre. On peut le montrer facilement en utilisant la description du foncteur de Picard suggérée par Niels. Avec cette description, pour un champ $\Gc$ sur une base $S$, l'ensemble $\Pic_{\Gc/S}(S)$ est en bijection avec les classes d'isomorphie de couples $(\Hc, h : \Gc \to \Hc)$ où $\Hc$ est une $\gm$-gerbe et $h$ un $S$-morphisme de champs. Si $\Gc$ est une $F$-gerbe, le caractère associé au point $l\in \Pic_{\Gc/S}(S)$ qui correspond à $(\Hc, h : \Gc \to \Hc)$ est induit par $h$ (regarder les automorphismes d'un objet, la construction descend). Il faut alors montrer que, si $\Gc$ est une $F$-gerbe sur $S$ et $\chi : F\to \gm$ est un caractère de $F$, il existe une $\gm$-gerbe $\Hc$ et un morphisme $\Gc\to \Hc$ qui induit $\chi$. Il suffit pour cela de voir $\Gc$ comme un $BF$-torseur et de prendre pour $\Hc$ le $B\gm$-torseur obtenu à partir de $\Gc$ par extension des scalaires le long du morphisme $BF\to B\gm$ induit par $\chi$ (voir paragraphe sur les torseurs dans ce papier). Ceci montre que pour $\Gc$ une $F$-gerbe sur $S$, le morphisme $\Pic_{\Gc/S}\to F^D$ est surjectif. Pour l'injectivité, voir (2) ci-dessous.
}

(2) By~\cite[5.3.6]{Brochard_Picard} the sequence $\Pic(X) 
\to \Pic(\Gc)\to F^D(S)$ is exact. Let $\pi : \Gc \to X$ 
denote the structural morphism of~$\Gc$. The Leray spectral 
sequence of $\pi$ yields an injection $H^1(X, \pi_*\gm) \inj 
H^1(\Gc,\gm)$. But the canonical map $\Oc_X\to 
\pi_*\Oc_{\Gc}$ is universally an isomorphism, hence 
$\pi_*\gm=\gm$ and we see that $\pi^* : \Pic(X)\to 
\Pic(\Gc)$ is injective. Sheafifying, we get the exactness 
of the sequence $0\to \Pic_{X/S}\to \Pic_{\Gc/S}\to F^D$.
\end{proof}

\begin{remarque}\rm
\label{rem_dual_classifiant_via_Picard}
 We let the reader check that the isomorphisms~\ref{particular_case_of_a_classifying_stack} and~\ref{Picard_classifiant}~(1) are compatible with the forgetful morphism~$\omega : D(BG)\to \Pic_{BG/S}$ in the sense that the following diagram commutes.
$$\xymatrix@R=1pc@C=1pc{
D(BG) \ar[rr]^{\omega} \ar[rd]_{(\ref{particular_case_of_a_classifying_stack})}&& \Pic_{BG/S}\ar[ld]^{(\ref{Picard_classifiant})}\\
&G^D &
}$$
\end{remarque}

\begin{cor}
 \label{comp_dual_pic}
Let $G$ be a sheaf of commutative groups over a base scheme $S$. Then the forgetful morphism from $D(BG)$ to $\champic(BG/S)$ induces an isomorphism
$$\xymatrix{\omega : D(BG)\ar[r]^-{\sim} & \Pic_{BG/S}\, .}$$
\end{cor}

\section{Dualizability}
\label{section_dualizability}

\begin{defi}
\label{def_eval_map}
 Let $G$ be a commutative group stack over a base scheme $S$. There is a natural evaluation homomorphism
$$e_G : G \lto DD(G)$$
that maps a point $g$ of $G$ to the homomorphism $e_G(g) : D(G)\to B\gm$ 
defined by $e_G(g)(\varphi)=\varphi(g)$ for any object $\varphi : G\to B\gm$ of 
$D(G)$.
We say that $G$ is dualizable if $e_G$ is an isomorphism.
\end{defi}

\begin{prop}
 \label{fonctorialite_evaluation_map}
 \label{dual_evaluation_map}\ 
\begin{itemize}
 \item[(a)] The evaluation map $e_G$ is functorial in the following sense. If $f : G \to H$ is a morphism of commutative group stacks, then the square
$$\xymatrix@R=1pc{
G\ar[r]^-{e_G} \ar[d]_{f} &
DD(G)\ar[d]^{DD(f)} \\
H\ar[r]_-{e_H} &DD(H)
}$$
 is \emph{strictly} commutative.
\item[(b)] The composition $D(e_G)\circ e_{D(G)}$
$$\xymatrix{
D(G) \ar[r]^-{e_{D(G)}}& DDD(G) \ar[r]^-{D(e_G)} &D(G)
}$$
is equal to the identity of~$D(G)$.
\item[(c)] Let~$G$ be a dualizable commutative group stack. Then~$D(G)$ is 
dualizable.
\item[(d)] Forming the evaluation morphism commutes with base change.
\item[(e)] For commutative group stacks, the property of being dualizable is stable under base change.
\end{itemize}
\end{prop}
\begin{proof}
(b) and (d) are straightforward verifications and (c), (e) are immediate 
consequences. To prove (a), we just observe that both morphisms $e_H\circ f$ and 
$DD(f)\circ e_G$ map an object $x$ of $G$ to the morphism of commutative group 
stacks
$$\begin{array}{rcl}
\shHom(H, B\gm)& \lto & B\gm \\
\varphi &\xymatrix{\ar@{|->}[r]&} & \varphi(f(x))
\end{array}\, .$$
\end{proof}

\begin{prop}
 \label{antiequivalence}
The 2-functor $D(.)$ induces a 2-antiequivalence from the 2-category of 
dualizable commutative group stacks to itself.
\end{prop}
\begin{proof}
 If $G$ is dualizable, then so is $D(G)$, and $G$ is by definition isomorphic to $D(D(G))$, hence $D(.)$ is 2-essentially surjective. It remains to prove that
for $G$ and $H$ dualizable, the functor
$$D(.) : \Hom(G,H) \lto \Hom(D(H), D(G))$$
is an equivalence of categories. Using~\ref{fonctorialite_evaluation_map} (a), we observe that the functors $DD(.)\circ e_G$ and $(e_H\circ .)$ from $\Hom(G,H)$ to $\Hom(G, DD(H))$ are equal. Since $(.\circ e_G)$ and $(e_H\circ .)$ are equivalences, we deduce that $DD(.)$ is an equivalence. The result then follows from the lemma below, whose proof is straightforward.
\end{proof}

\begin{lem}
 Let $F : A \to B$ and $G : B \to C$ be functors. Assume that $G\circ F$ is an equivalence. Then:
\begin{itemize}
 \item[(i)] $F$ is faithful and $G$ is essentially surjective.
\item[(ii)] If $F$ is essentially surjective (resp. if $G$ is faithful) then $G$ (resp. $F$) is full.
\item[(iii)] $F$ is an equivalence if and only if $G$ is an equivalence. 
\end{itemize}
\end{lem}

\begin{prop}
\label{prop_dualisabilite}
 Let $G$ be a commutative group stack over a base scheme $S$. Let $S'\to S$ be an \emph{fppf} cover. The following are equivalent:
\begin{itemize}
 \item[(i)] $G$ is dualizable.
\item[(i')] The morphisms $H^i(e_G) : H^i(G) \to H^i(DD(G))$ are isomorphisms ($i=-1, 0$). 
\item[(ii)] $G\times_S S'$ is dualizable.
\item[(ii')] The morphisms $H^i(e_{G\times_SS'})$ are isomorphisms ($i=-1, 0$). 
\end{itemize}
\end{prop}
\begin{proof}
 The equivalences $(i)\Leftrightarrow (i')$ and 
$(ii)\Leftrightarrow (ii')$ follow from Deligne's 
equivalence between the category of commutative group stacks 
and the derived category $D^{[-1, 0]}(S, \Z)$ of length~1 
complexes of \emph{fppf} sheaves of commutative groups. The 
equivalence $(i')\Leftrightarrow (ii')$ follows from the 
fact that, for a morphism of \emph{fppf} sheaves, being an 
isomorphism is local in the \emph{fppf} topology.
\end{proof}
\details{Si les $H^i(G)$ sont des faisceaux fpqc, et si les $H^i(DD(G))$ sont des préfaisceaux séparés pour fpqc, alors la question de la dualisabilite est locale sur $S$ pour fpqc. De même si $G$ et $DD(G)$ sont des champs fpqc, alors la question de la dualisabilite est locale sur $S$ pour fpqc. Question subsidiaire : y a-t-il un lien entre ``(1) : $G$ est un champ fpqc'' et ``(2) : les $H^i(G)$ sont des faisceaux fpqc'' ? Je crois que (1) implique (2) n'est pas clair du tout. Inversement supposons (2). Au brouillon j'ai montré (à vérifier...) que pour tous $x,y$ $\Isom(x,y)$ est un faisceau fpqc. Il reste à montrer la condition de recollement... (pas fait).}

\begin{lem}
 \label{compatibility_of_can_isoms_with_ev_maps}
Let $G$ be a sheaf of abelian groups. Then the diagrams
$$\xymatrix{
G \ar[r]^-{e_G} \ar[d]_{c} 
   &
DD(G) \ar[d]^{D(\ref{particular_case_of_a_sheaf})} \\
G^{DD} & D(BG^D) \ar[l]_-{(\ref{particular_case_of_a_classifying_stack})}
}
\qquad
\textrm{ and }
\qquad
\xymatrix{
BG \ar[r]^-{e_{BG}} \ar[d]_{B(c)} 
   &
DD(BG) \\
B(G^{DD})\ar[r]^-{(\ref{particular_case_of_a_sheaf})} & D(G^D) \ar[u]_{D(\ref{particular_case_of_a_classifying_stack})}
}
$$
where $c$ stands for the evaluation map of Cartier duality, commute.
\end{lem}
\begin{proof}
Let us prove that the first diagram commutes. It suffices to do it on 
$S$-points, so let $g\in G(S)$. On the one hand, $c(g) : G^D \to \gm$ is the 
morphism that maps an element $\varphi\in G^D(S)$ to $\varphi(g)\in \gm(S)$. The 
other side is a little bit more complicated to describe: it is the morphism that 
we get by applying $H^{-1}$ to the morphism $e_G(g)\circ 
(\ref{particular_case_of_a_sheaf})$. The morphism 
$H^{-1}(\ref{particular_case_of_a_sheaf})$ identifies with the canonical 
isomorphism $G^D\simeq H^{-1}(D(G))$, in other words it maps an element 
$\varphi\in G^D(S)$ to the automorphism of the neutral element $e : G\to B\gm$ 
of $D(G)$ that maps a point $h$ of $G$ to the automorphism of the trivial 
$\gm$-torsor that corresponds to $\varphi(h)$. It remains to describe 
$e_G(g) : D(G) \to B\gm$ on automorphisms. Let $\psi : G\to B\gm$ be an object 
of $D(G)$ (actually we are only interested in the neutral object). Then 
$e_G(g)$ maps an automorphism $u$ of $\psi$ (that is, the data, for each element 
$h$ of $G$ of an automorphism $u(h)$ of $\psi(h)\in B\gm$) to the automorphism 
$u(g)$ of $\psi(g)$. In the end, the morphism $G^D\to \gm$ that we get is indeed 
equal to the morphism that maps an element $\varphi \in G^D$ to $\varphi(g)$. 
The commutativity of the second diagram is a similar definition-chasing.
\details{Faisons le premier par exemple. Il suffit de le faire sur les $S$-points. Soit $g\in G(S)$. D'un côté, $c(g)$ est le morphisme qui envoie un $\varphi\in G^D(S)$ sur $\varphi(g)\in \gm(S)$. De l'autre côté, c'est un peu plus compliqué mais on y arrive quand même. $e_G(g)$ est le morphisme $D(G)\to B\gm$ qui envoie un $\psi\in D(G)$ sur $\psi(g)$. En fait plus bas on aura surtout besoin de décrire $e_G(g)$ sur les automorphismes, donc autant le faire tout de suite. $e_G(g)$ envoie un automorphisme $u$ de $\psi : G\to B\gm$ (on rappelle que se donner un automorphisme $u$ de $\psi$ revient à se donner pour chaque élément $h$ de $G$ un automorphisme $u(h)$ de $\psi(h)\in B\gm$) sur l'automorphisme $u(g)$ de $\psi(g)$. $D(\ref{particular_case_of_a_sheaf})$ envoie $e_G(g)$ sur le morphisme composé $BG^D \to D(G) \to B\gm$ où la première flèche est \ref{particular_case_of_a_sheaf} et la seconde est $e_G(g)$. Enfin appliquer \ref{particular_case_of_a_classifying_stack} revient à prendre le morphisme induit sur les groupes d'automorphismes de la section neutre. On obtient bien le morphisme qui à $u\in G^D$ associe $u(g)\in \gm(S)$. Pour l'autre diagramme c'est du même acabit.
}
\end{proof}

\begin{prop}
\label{prop_dualisabilite_Cartier_et_abeliens}
 Let $G$ be a sheaf of commutative groups over $S$.
\begin{enumerate}
 \item The stack $BG$ is dualizable if and only if the natural map $G\to G^{DD}$ is an isomorphism (in other words, $G$ is dualizable in the sense of Cartier duality) and $E^1(G^D)=0$.
\item If $E^i(E^1(G))=0$ for $i=0, 1, 2$ and $G\to G^{DD}$ 
is an isomorphism, or if $G$ is an abelian scheme, then $G$ 
is dualizable as a commutative group stack.
\end{enumerate}
\end{prop}
\begin{proof}
 (1) is a consequence 
of~\ref{particular_case_of_a_classifying_stack}, 
\ref{particular_case_of_a_sheaf}, 
and~\ref{compatibility_of_can_isoms_with_ev_maps}.

Let us prove (2). Assume first that $E^i(E^1(G))=0$ for 
$i=0, 1, 2$ and $G\to G^{DD}$ is an isomorphism. 
By~\ref{exactness_of_the_dual_sequence} the dual of the 
sequence $0\to BG^D\to D(G)\to E^1(G)\to 0$ is exact. Since 
$E^i(E^1(G))=0$ for $i=0,1$ we see that $D(E^1(G))=0$. Hence 
the canonical morphism $D(\ref{particular_case_of_a_sheaf}) 
: DD(G) \to D(BG^D)$ is an isomorphism and 
by~\ref{compatibility_of_can_isoms_with_ev_maps} this proves 
that $e_G$ is an isomorphism.

Now let $A$ be an abelian scheme. By~\ref{comp_dual_picto} we know that $D(A)$ and $DD(A)$ are abelian schemes. We let the reader check that the following diagram commutes, which proves that $e_A$ is an isomorphism.
$$\xymatrix{
A\ar[r]^-{e_A} \ar[d]_{\textrm{can.}} & DD(A)
\ar[d]^{(\ref{comp_dual_picto})} \\
A^{tt}\ar[r]_-{(\ref{comp_dual_picto})^t} & D(A)^t
}$$
\end{proof}

\begin{exemple}\rm
\label{les_Cartier_sont_dualisables}
   If $G$ is a Cartier group (see~\ref{Cartier_groups}), then $E^1(G)=E^1(G^D)=0$ and $G\to G^{DD}$ is an isomorphism. Hence $G$ and $BG$ are dualizable.
\end{exemple}

 Let $G$ be a commutative group stack over a base scheme $S$. Let us denote 
$A:=BH^{-1}(G)$ and $C:=H^0(G)$. By~\ref{structural_exact_sequence} there is a 
canonical exact sequence 
$0\to A\to G\to C\to 0$. Assume that the 
morphism $H^{-1}(G)^D \to E^2(H^0(G))$ is trivial. Then by~\ref{dual_sec_can}, 
the sequence $\xymatrix@C=1pc{
0\ar[r] & D(C)\ar[r] &D(G)\ar[r] & D(A)\ar[r] &0
}$ is super-exact. Applying $\Hom(\, .\, , \gm)$ to the $H^0$ sequence we get a 
morphism $H^0(D(C))^D\to E^1(D(A))$. But
by~\ref{particular_case_of_a_classifying_stack}, 
$D(A)\simeq H^{-1}(G)^D$ and by~\ref{description_dual}, $H^0(D(C))\simeq 
E^1(H^0(G))$, so that we get a morphism 
$E^1(H^{0}(G))^D \to 
E^1(H^{-1}(G)^D)$. 

\begin{prop}
\label{devissage_dualisabilite}
 Let $G$ be a commutative group stack over a base scheme $S$. Assume that the  
morphism $H^{-1}(G)^D \to E^2(H^0(G))$ is trivial and that the induced morphism 
(described above) 
$E^1(H^{0}(G))^D \to E^1(H^{-1}(G)^D)$ is trivial as well. If 
both~$BH^{-1}(G)$ and $H^0(G)$ are dualizable, then so is $G$.
\end{prop}
\begin{proof}
With the above notations, using the 
diagram in the proof of~\ref{exactness_of_the_dual_sequence} we see that 
the sequences ($i=-1, 0$)
$$\xymatrix{
0\ar[r] &  H^i(DD(A))\ar[r] &H^i(DD(G)) \ar[r] &
H^i(DD(C))
}$$
are exact. Hence, for $i=-1,0$ we have a commutative diagram with exact lines:
$$\xymatrix{
0 \ar[r]&H^i(A)\ar[d]_{H^i(e_A)} \ar[r]&H^i(G)\ar[d]^{H^i(e_G)} \ar[r]&H^i(C)\ar[d]^{H^i(e_C)} \ar[r]& 0\\
  0\ar[r]&H^i(DD(A)) \ar[r]&H^i(DD(G)) \ar[r]&H^i(DD(C))
}$$
Since the left and right vertical arrows are isomorphisms, so is the middle one.
\end{proof}

\begin{exemple}\rm
 Let $G$ be an abelian stack over $S$. 
By~\ref{prop_dualisabilite_Cartier_et_abeliens}, 
$BH^{-1}(G)$ and $H^0(G)$ are dualizable. The morphism 
$H^{-1}(G)^D\to E^2(H^0(G))$ is trivial 
by~\ref{vanishing_ext23_abelian}, and $E^1(H^0(G))^D=0$ 
by~\ref{dual_and_ext1_abelian}. Hence~$G$ is dualizable.
\end{exemple}

\begin{thm}
\label{thm_dualizability}
 Let $S$ be a regular base scheme in which 2 is invertible. Let $G$ be a commutative group stack over $S$. Assume that \'etale-locally on~$S$:
\begin{itemize}
 \item[(i)] $H^0(G)$ fits in an exact sequence
$$0\lto A\lto H^0(G) \lto F\lto 0$$
where $A$ is an abelian scheme over $S$, and $F$ is built up, by successive extensions, from finite locally free group schemes and constant free group schemes of finite rank, and
\item[(ii)] $H^{-1}(G)$ is built up by successive extensions 
from finite locally free group schemes and split tori.
\end{itemize}
Then $G$ is dualizable, and $D(G)$ satisfies the same assumptions as $G$. More precisely, as soon as~(i) and~(ii) hold for~$G$, then $H^{-1}(D(G))\simeq F^D$, and $H^0(D(G))$ fits in an exact sequence
$$0\lto A^t \lto H^0(D(G)) \lto H^{-1}(G)^D\lto 0\, .$$  
\end{thm}
\begin{proof}
We can assume that~(i) and~(ii) hold. By~\ref{les_Cartier_sont_dualisables}, 
since $H^{-1}(G)$ is Cartier, we know that $E^1(H^{-1}(G)^D)=0$ and that 
$BH^{-1}(G)$ is dualizable. By~\ref{devissage_dualisabilite}, to prove that $G$ 
is dualizable it suffices to prove that $E^2(H^0(G))=0$ and $H^0(G)$ is 
dualizable. The exact sequence given in~(i) induces a long exact sequence:
$$0\to F^D\to H^0(G)^D\to A^D\to E^1(F) \to E^1(H^0(G)) \to E^1(A)\to E^2(F)\dots$$
Since the sheaves $A^D$, $E^1(F)$, $E^2(F)$ and $E^2(A)$ all vanish (section~\ref{section_Ext}), we see that $E^2(H^0(G))$ is zero and we get isomorphisms $F^D\simeq H^0(G)^D$ and $E^1(H^0(G))\simeq A^t$. Using the description of~$D(H^0(G))$ from~\ref{particular_case_of_a_sheaf} and applying~\ref{exactness_of_the_dual_sequence} twice, we see that the sequence
$$0\lto DD(A) \lto DD(H^0(G)) \lto DD(F)\lto 0$$
is super-exact. Hence the dualizability of $H^0(G)$ follows from that of $A$ 
and $F$. 
This proves that $G$ is dualizable. Moreover, by~\ref{description_dual} we have 
an isomorphism $H^{-1}(D(G))\simeq H^{0}(G)^D$ and an exact sequence
\[
 0\to E^1(H^0(G)) \to H^0(D(G)) \to H^{-1}(G)^D \to 
E^2(H^0(G))=0\, ,
\]
whence the assertions about $D(G)$.
\end{proof}

\begin{exemple}\rm Assume that the base scheme is the 
spectrum of an algebraically closed field~$k$ of 
characteristic different from 2. Let~$G$ be an 
algebraic commutative $k$-group stack locally of finite 
type. Assume that~$(H^{-1}(G))_{\text{red}}^0$ is a torus 
and that $(H^0(G))_{\text{red}}^0$ is an abelian variety. 
Assume moreover that the groups of connected components of 
$H^0(G)$ and $H^{-1}(G)$ are of finite type as ordinary 
abelian groups. Then by~SGA~3~\cite[VI$_A$~5.5.1 
and~5.6.1]{SGA3_new}, $G$ satisfies the 
assumptions of~\ref{thm_dualizability}, hence it is 
dualizable.
\end{exemple}

The assumptions on $S$ in~\ref{thm_dualizability} might be 
superfluous. We can drop them if we restrict the class of 
commutative group stacks, using the results of 
Section~\ref{section_dual}.

\begin{thm}
 \label{thm_dualizability2}
Let $G$ be a commutative group stack over a base scheme~$S$. 
Assume that one of the following holds:
\begin{itemize}
 \item[(i)] $G$ is a duabelian group.
\item[(ii)] $H^{-1}(G)$ is a 
finite flat group scheme, $H^0(G)$ is a duabelian 
group, and ${2\in\Oc_S^{\times}}$.
\end{itemize}
 Then $G$ is dualizable.
\end{thm}
\begin{proof}
Assume (i). By~\ref{prop_duabelian_groups}, the group $G$ 
fits in an extension $0\to A\stackrel{i}{\to} G\to F\to 0$ 
where~$A$ is an abelian scheme and $F$ is a finite flat 
group scheme. We have seen in the proof 
of~\ref{prop:E0_E1_des_duabeliens} that $E^1(i) : 
E^1(G)\to E^1(A)$ is an isomorphism. 
This implies that $H^0(D(G))\to 
H^0(D(A))$ is an isomorphism (\ref{description_dual}). 
By~\ref{exactness_of_the_dual_sequence} a) this 
shows that the sequence $$0\lto D(F)\lto D(G)\lto D(A)\lto 
0$$ is exact. Recall that $D(F)\simeq BF^D$ and $D(A)\simeq 
A^t$. By~\ref{vanishing_ext23_abelian} any morphism $F\to E^2(A^t)$ is trivial. 
Hence by~\ref{dual_sec_can} the sequence 
$$\xymatrix{
0\ar[r]&DD(A)\ar[r]&DD(G)\ar[r]&DD(F)\ar[r] &0
}$$
is super-exact.
Since $A$ and $F$ are dualizable 
(\ref{prop_dualisabilite_Cartier_et_abeliens}), it follows 
that $G$ is dualizable as well.

Now assume (ii). By \ref{prop_dualisabilite}, to prove that 
$e_G$ is an isomorphism it suffices to prove that $H^i(e_G)$ 
is an isomorphism for $i=-1, 0$. Applying 
\ref{thm_representabilite4} twice, we see that the stacks 
$D(G)$ and $DD(G)$ are algebraic and satisfy the same 
assumptions as $G$. Hence we can 
apply~\ref{lemme_isomorphisme_fibre_a_fibre} and we may 
assume that~$S$ is the spectrum of an algebraically closed 
field. But in this case, \ref{thm_dualizability} applies 
hence $e_G$ is an isomorphism.
\details{On peut aussi procéder ainsi. 
By~\ref{vanishing_ext23_abelian} the morphism 
$H^{-1}(G)^D\to E^2(H^0(G))$ is trivial. Moreover 
by~\ref{vanishing_ext_finite_or_multiplicative}, 
$E^1(H^{-1}(G)^D)=0$. Hence~\ref{devissage_dualisabilite} 
applies and~$G$ is dualizable.
}
\end{proof}

We summarize in the following table some classes of stacks 
which are known to be dualizable so far. For each line of 
this table, the 2-functor $D(.)$ induces a 2-antiequivalence 
between the class on the left and the class on the right. 
In the last line we assume that $2\in \Oc_S^{\times}$.

\smallskip
\renewcommand{\arraystretch}{1.5}
\newcolumntype{M}{>{\raggedright}m{0.47\textwidth}}
\begin{center}
\begin{tabular}{|M|M|}
\hline
\multicolumn{2}{|c|}{
$$D(.)$$
} \tabularnewline
\multicolumn{2}{|c|}{
$$\xymatrix@C=15pc{
\ar@{<->}@/^15pt/[r]&
}$$
} \tabularnewline
\hline

Cartier group schemes (see~\ref{Cartier_groups}) & classifying stacks of Cartier group schemes \tabularnewline
\hline
abelian schemes & abelian schemes \tabularnewline
\hline
1-motives & 1-motives \tabularnewline
\hline
abelian stacks (see~\ref{def_abelian_stacks}) & duabelian group schemes (see~\ref{def_duabelian_groups}) \tabularnewline
\hline
commutative group stacks $G$ with $H^{-1}(G)$ finite flat and $H^0(G)$ 
duabelian &
commutative group stacks $G$ with $H^{-1}(G)$ finite flat and $H^0(G)$ 
duabelian \tabularnewline
\hline
\end{tabular}
\end{center}
\smallskip

\begin{remarque}\rm
\label{preuve_conj1_implique_conj2}
 Assume that \ref{conj}~(1) is true. Then \ref{conj}~(2) also holds. (Note 
however that $H^{-1}(D(G))$ is not flat in general.) Indeed, let $G$ be a 
proper, flat and finitely presented commutative group stack, with $H^{-1}(G)$ 
finite and flat. By~\ref{thm_representabilite1}~$D(G)$ is algebraic and of 
finite 
presentation. By our assumption it is even proper and flat. Hence 
applying~\ref{thm_representabilite1} again $DD(G)$ is algebraic and of finite 
presentation. Since~$H^{-1}(G)$ is flat, and since the result is known over an 
algebraically closed field (\ref{thm_dualizability}), 
by~\ref{lemme_isomorphisme_fibre_a_fibre} the morphism~$H^{-1}(e_G)$ is an 
isomorphism. In particular $H^{-1}(DD(G))$ is flat. Then by Artin's 
theorem~\cite[10.8]{LMB} the coarse moduli sheaves~$H^0(G)$ and~$H^0(DD(G))$ are 
algebraic spaces. Moreover~$H^0(G)$ is flat (because~$G$ is flat, flatness is 
local at the source~\cite[IV, 2.2.11]{EGA}, and $G\to H^0(G)$ is an \emph{fppf} 
epimorphism by~\cite[10.8]{LMB}) and 
by~\ref{lemme_isomorphisme_fibre_a_fibre} again $H^0(e_G)$ is an isomorphism. 
Hence $e_G$ is an isomorphism by~\ref{prop_dualisabilite}.
\end{remarque}

\section{Torsors under a commutative group stack}
\label{section_torsors}

In this whole section, $G$ is a commutative group stack over a base scheme $S$. 
We denote by~$e : S \to G$ the neutral section. The definition of a torsor 
under~$G$ was given by Breen in~\cite{Breen_Bitorseurs}. Note that we switched 
to a multiplicative notation for the ``addition'' of the group 
stack.

\begin{defi}
\label{def_action}
 (i) An action of $G$ on an $S$-stack $T$ is a pair $(\mu, \alpha)$ where $\mu : G\times_S T\to T$
is a morphism of $S$-stacks, and $\alpha$ is a 2-isomorphism 
making the following diagram 2-commutative.
$$\xymatrix@R=0.5pc@C=0.9pc{G\times_S G\times_S T 
\ar[rr]^-{m_G\times \id_T} \ar[d]_{\id_G\times \mu} & 
\raisebox{-3ex}{$^{\alpha} \FlecheNE$} &
G\times_S T \ar[d]^-{\mu} \\
G\times_S T \ar[rr]_-{\mu} && T}$$
In other words, there is a functorial collection of isomorphisms
$$\alpha_{g,h}^x : g.(h.x) \lto (gh).x$$
for all objects $x$ of $T$ and $g, h$ of $G$, where a notation like $h.x$ 
stands for $\mu(h,x)$. Moreover, we require the following two conditions:
\begin{itemize}
 \item[a)] For all objects $x$ of $T$ and $g, h, k$ of $G$, we have a commutative diagram of 2-isomorphisms:
$$\def\MyNode{\ifcase\xypolynode\or
      g.((hk).x)
    \or
      g.(h.(k.x))
    \or
      (gh).(k.x)
    \or
      ((gh)k).x
    \or
      (g(hk)).x
    \fi
  }%
  \xy/r7pc/: (0,.3)::
    \xypolygon5{~>{}\txt{\ \ \strut\ensuremath{\MyNode}}} 
    \ar "2";"3" _{\alpha_{g,h}^{k.x}}
    \ar "3";"4" _<{\alpha_{gh,k}^{x}}
    \ar "4";"5" ^{(\lambda_{g,h,k}).x}
    \ar "2";"1" ^{g.(\alpha_{h,k}^{x})}
    \ar "1";"5" ^<{\alpha_{g,hk}^{x}}
  \endxy$$
\item[b)] For any $g\in G(S)$, the translation $\mu_g : T\to T$ defined by $\mu_g(t)=g.t$ is an equivalence of categories.
\end{itemize}
(ii) Let $f_0 : G \to G'$ be a homomorphism from~$G$ to another 
commutative group stack $G'$ and let~$(T',\mu', \alpha')$ be an $S$-stack with 
an action of $G'$. An $f_0$-equivariant morphism\footnote{If $G=G'$ and 
$f_0=\id_G$ we will talk about a $G$-equivariant morphism.} from $T$ to $T'$ is 
a pair $(f_1, \sigma)$ where $f_1 : T\to T'$ is a morphism of $S$-stacks and 
$\sigma$ is a 2-isomorphism making the following diagram commutative.
$$
\xymatrix@R=0.3pc@C=0.3pc{
G\times_S T \ar[rr]^-{\mu} \ar[dd]_{f_0\times f_1} &&
T\ar[dd]^{f_1} \\
&^\sigma\FlecheNE & \\
G'\times_S T' \ar[rr]_-{\mu'} && T'
}
$$
In other words, $\sigma$ is a functorial collection of isomorphisms:$$\xymatrix{\sigma_g^x : f_0(g).f_1(x) \ar[r]& f_1(g.x).}$$
We moreover require that these isomorphisms satisfy a compatibility condition with the other data, \emph{i.e.}\! for all objects $g, h$ in $G$ and $x$ in $T$, the following diagram of 2-isomorphisms is commutative.
$$
\def\MyNode{\ifcase\xypolynode\or
      f_0(gh).f_1(x)
    \or
      (f_0(g)f_0(h)).f_1(x)
    \or
      f_0(g).(f_0(h).f_1(x))
    \or
      f_0(g).f_1(h.x)
    \or
      f_1(g.(h.x))
    \or
      f_1((gh).x)
    \fi
  }%
  \xy/r10.5pc/: (0,.22)::
    \xypolygon6{~>{}\txt{\ \ \strut\ensuremath{\MyNode}}} 
    \ar "3";"4" _>{f_0(g).\sigma_h^x}
    \ar "4";"5" _{\sigma_g^{h.x}}
    \ar "5";"6" _{f_1(\alpha_{g,h}^x)}
    \ar "3";"2" ^{{\alpha'}_{f_0(g),f_0(h)}^{f_1(x)}}
    \ar "1";"2" 
    \ar "1";"6" ^{\sigma_{gh}^x}
  \endxy
$$
(iii) If $(f_1,\sigma)$ and $(f_1', \sigma')$ are two $f_0$-equivariant 
morphisms as in~(ii), a 2-isomorphism from~$(f_1, \sigma)$ to $(f_1', \sigma')$ 
is a 2-isomorphism $\tau : f_1 \Rightarrow f_1'$ that is compatible with the 
$\sigma$'s, \emph{i.e.} such that for any objects $x$ of $T$ and $g$ of $G$, 
$\tau^{g.x}\circ \sigma_{g}^x = {\sigma'_g}^{x}\circ (f_0(g).\tau^x)$.
\end{defi}

\begin{remarque}\rm
Given an $S$-stack~$T$ and a pair~$(\mu, \alpha)$ satisfying the pentagon 
condition~(i)~a) of~\ref{def_action}, the following requirements are equivalent 
(but not automatic, e.g. if $\mu$ is a constant morphism then it satisfies (i) 
a) but not (i) b)):
\begin{itemize}
 \item[(i)]  b) For any $g\in G(S)$, the translation $\mu_g$ is an equivalence of categories.
\item[(i)]  b') For \emph{some} $g\in G(S)$, the translation $\mu_g$ is an equivalence of categories.
\item[(i)]  c) For some neutral object $(e,\eps)$, there is 
a (automatically unique) 2-isomorphism\vskip-0.3cm
$$\xymatrix@R=0.9pc@C=0.9pc{T \ar[rr]^-{e\times \id_T} \ar[rd]_{\id_T}  &
\raisebox{-3ex}{$^{\mathfrak{a}_e} \FlecheSO$} &
G\times_S T \ar[ld]^{\mu}\\
& T &}$$
in other words a functorial collection of isomorphisms $\mathfrak{a}_e^x : e.x \to x,$
and for all objects $g$ of $G$ and $x$ of $T$ the following diagrams of 2-isomorphisms commute:
$$\def\MyNode{\ifcase\xypolynode\or
      g.(e.x)
    \or
      (ge).x
    \or
      g.x
    \fi
  }%
 \xy/r2.4pc/: (0,1)::
    \xypolygon3{~>{}\txt{\ \ \strut\ensuremath{\MyNode}}} 
    \ar "1";"3" ^{g.\mathfrak{a}_e^x}
    \ar "2";"3" 
    \ar "1";"2" _{\alpha_{g,e}^{x}}
  \endxy
\qquad \quad
\def\MyNode{\ifcase\xypolynode\or
      e.(g.x)
    \or
      (eg).x
    \or
      g.x
    \fi
  }%
 { \xy/r2.4pc/: (0,1)::
    \xypolygon3{~>{}\txt{\ \ \strut\ensuremath{\MyNode}}} 
    \ar "1";"3" ^{\mathfrak{a}_e^{g.x}}
    \ar "2";"3" 
    \ar "1";"2" _{\alpha_{e,g}^{x}}
  \endxy}
$$
where the bottom maps are uniquely determined by $\eps$.
\item[(i)] c') For \emph{any} neutral object $(e,\eps)$ there is a (unique) 2-isomorphism $\mathfrak{a}_e$ as in (i) c).
\end{itemize}
\end{remarque}

\begin{remarque}\rm
\label{rem:action_du_neutre_et_morph_equivariant}
 Let $(f_1, \sigma)$ be an $f_0$-equivariant morphism as in~(ii). Let $(e, 
\eps)$ be a neutral object of~$G$. Its image $e'=f_0(e)$ is a neutral object 
of~$G'$. Let $\mathfrak{a}_e$ and $\mathfrak{b}_{e'}$ be the associated 
2-isomorphisms as in (i) c) above. Then for any~$x$ in~$T$, the following 
diagram commutes:
$$\xymatrix@R=1.5pc@C=0.3pc{f_0(e).f_1(x) \ar[rr]^-{\mathfrak{b}_{e'}^{f_1(x)}}
\ar[rd]_{\sigma_e^x} &&
f_1(x) \\
&f_1(e.x) \ar[ru]_{f_1(\mathfrak{a}_e^x)}&
.}$$
\end{remarque}

The following lemma is straightforward.

\begin{lem}
 \label{conditions_torseur}
Let $T$ be an $S$-stack, with an action $(\mu, \alpha)$ of~$G$. The following are equivalent:
\begin{itemize}
 \item[(i)] \emph{Fppf}-locally on~$S$, there is a 
$G$-equivariant isomorphism $f_1 : G \to T$. 
\item[(ii)] The natural morphism
$(\mu, p_2) : G\times_S T\to T\times_S T,\ (g,t)\mapsto (g.t, t)$ is an equivalence, and the morphism $T\to S$ is an \emph{fppf} epimorphism, \emph{i.e.} there is an \emph{fppf} covering $S'\to S$ such that $T(S')$ is nonempty.
\end{itemize}
\end{lem}

\begin{defi}
  A $G$-torsor is an $S$-stack~$T$ with an action of $G$ satisfying the conditions of~\ref{conditions_torseur}.
\end{defi}

For any $G$-torsor $T$, and any stack~$P$ with an action of $G$, Breen defines in~\cite{Breen_Bitorseurs} a contracted product $T\wedge^G P$ that inherits a natural action of $G$. Let us recall some properties of this construction.

\begin{prop}
 \label{prop_produit_contracte}
\begin{itemize}
 \item[a)] If $T_1$ and $T_2$ are two $G$-torsors, then $T_1\wedge^G T_2$ is again a $G$-torsor. This defines a group law on the set $H^1(S, G)$ of isomorphism classes of $G$-torsors, where the neutral element is the class of the trivial torsor $G$.
\item[b)] Let $T$ be a $G$-torsor and $\varphi : G\to H$ a morphism of 
commutative group stacks. This induces a natural $G$-action on $H$. Then $H$ 
naturally acts on the stack $T\wedge^G H$ and makes it an $H$-torsor (which we 
denote by $T\wedge^{G,\varphi} H$ if there is an ambiguity on the morphism 
$\varphi$). This defines a group morphism
$H^1(\varphi) : H^1(S, G) \to H^1(S, H).$
\end{itemize}
\end{prop}

\begin{prop}
 \label{prop_extension_scalaires_multiplicative_en_phi}
Let $G$ and $H$ be two commutative group stacks over a base scheme $S$ and let 
$\varphi_1$ and $\varphi_2$ be two homomorphisms of commutative group stacks 
from $G$ to $H$. Let $\psi$ denote their product, defined functorially by 
$\psi(g)=\varphi_1(g)\varphi_2(g)$. Then for a $G$-torsor $T$, there is a 
functorial isomorphism
$$T\wedge^{G, \psi} H\simeq (T\wedge^{G, \varphi_1} H)\wedge^H (T\wedge^{G, \varphi_2} H).$$
\end{prop}
 
We can also describe $G$-torsors in terms of extensions of $\Z$ by $G$. We define a 2-category $Ext^1(\Z, G)$ as follows.
\begin{enumerate}
 \item An object is an exact sequence of commutative group stacks 
$\xymatrix@C=1pc{0\ar[r]& G \ar[r]^j& E \ar[r]^{\pi}& \Z 
\ar[r]& 0.}$
\item A morphism between two such objects $(E_1, j_1, \pi_1)$ and $(E_2, j_2, \pi_2)$ is a pair $(\varphi, \beta)$ where $\varphi : E_1 \to E_2$ is a homomorphism of commutative group stacks such that $\pi_2\circ \varphi =\pi_1$ and $\beta : \varphi\circ j_1 \Rightarrow j_2$ is a 2-isomorphism of additive morphisms (see \ref{def_morphismes_groupes}).
\item A $2-$isomorphism from $(\varphi, \beta)$ to 
$(\varphi', \beta')$ is a 2-isomorphism $\delta : \varphi 
\Rightarrow \varphi'$ of additive morphisms (see 
\ref{def_morphismes_groupes}), that is compatible with 
$\beta$ and $\beta'$.
\end{enumerate}
If $\xymatrix@C=1pc{0\ar[r]&G \ar[r]^j&E \ar[r]^{\pi}&\Z \ar[r]&0}$ is an object of $Ext^1(\Z, G)$, then $\pi^{-1}(1)$ is naturally a $G$-torsor. This construction extends to a 2-functor $t$ from $Ext^1(\Z, G)$ to $(G-Tors)$. We leave to the reader the proof of the following fact.

\begin{prop}
\label{equiv_extensions_torseurs}
 The 2-functor $t$ from $Ext^1(\Z, G)$ to $(G-Tors)$ is a 2-equivalence of 2-categories.
\end{prop}

To conclude this section, we recall a few facts about $F$-gerbes for a fixed 
sheaf of commutative groups $F$, in particular the 
equivalence between the notions of an 
$F$-gerbe and a $BF$-torsor. 

\begin{defi}
 \label{def:gerbe}
\begin{enumerate}
 \item[(i)] Let $f : X\to S$ be a stack over an algebraic space $S$. We say 
that $f$ is a gerbe (or that $X$ is a gerbe over $S$) if $f$ and its diagonal 
are \emph{fppf} epimorphisms.
\item[(ii)] Let $F$ be a sheaf of commutative groups over $S$. An $F$-gerbe, or 
a gerbe banded by $F$, is a gerbe $f : X\to S$ together with an isomorphism of 
sheaves of groups
\[
 c_x : F_U \lto \Aut_{X}(x)
\]
for every $S$-scheme $U$ and every object $x\in X(U)$, such that the following 
conditions hold:
\begin{itemize}
 \item[(G1)] For any $S$-scheme $U$ and any isomorphism $\varphi : x\to x'$ in 
$X(U)$, we have $c_{\varphi}\circ c_x=c_{x'}$ where $c_{\varphi} : \Aut_X(x)\to 
\Aut_X(x')$ is the isomorphism induced by $\varphi$.
\item[(G2)] For any $S$-morphism $V\to U$ and any object $x\in X(U)$, the 
pullback of $c_x$ along $V\to U$ is equal to $c_{(x_{|_V})}$.
\end{itemize}
\item[(iii)] A morphism of $F$-gerbes from $(X', \{c_{x'}'\})$ to $(X,\{c_x\})$ 
is a morphism of stacks $g : X'\to X$ such that for every object $x'$ of $X'$ 
the diagram
\[
 \xymatrix{& F \ar[ld]_{c'_{x'}} \ar[rd]^{c_{g(x')}}
\\
\Aut_{X'}(x') \ar[rr]_{g_*} && \Aut_X(g(x'))
}
\]
commutes.
\item[(iv)] If $g$ and $h$ are two morphisms of $F$-gerbes, a 2-isomorphism 
from $g$ to $h$ is a 2-isomorphism $g\Rightarrow h$ of morphisms of stacks.
\end{enumerate}
\end{defi}

\begin{remarque}\rm
 Equivalently, an $F$-gerbe is a gerbe $X\to S$ together with an isomorphism of 
sheaves of groups over $X$ :
\[
 c : F_X \lto I_{X/S}
\]
where $F_X=F\times_S X$ and $I_{X/S}$ is the inertia stack of $X$ over 
$S$.
\end{remarque}

We refer 
to~\cite{Giraud} 
or~\cite{Olsson_Algebraic_spaces_and} for some elementary facts about gerbes 
banded by a sheaf of commutative groups $F$. In particular there is a natural 
bijection (\cite[IV, 3.4.2]{Giraud} 
or~\cite[12.2.4]{Olsson_Algebraic_spaces_and})
\[
 \gamma_F : H^2(S, F) \lto \{\text{isomorphism classes of } F-\text{gerbes}\} 
\]
where the group on the left is the \emph{fppf} derived functor cohomology group. 
The result below is folklore, but we were unable to find a suitable reference.

\begin{prop}
\label{prop:equivalence_BG_torseurs_G_gerbes}
 \begin{enumerate}
  \item There is a natural 2-equivalence of 2-categories between the 2-category 
of $BF$-torsors and the 2-category of $F$-gerbes.
\item If the set of isomorphism classes of $F$-gerbes is equipped via this 
equivalence with the group law induced by the contracted product of 
$BF$-torsors from~\ref{prop_produit_contracte}, then the above 
bijection $\gamma_F$ is a group isomorphism.
 \end{enumerate}
\end{prop}
\begin{proof}
 Since we did not even describe $\gamma_F$, we only give 
the proof of (1). Let us construct a 2-equivalence $\Phi$ from $BF$-torsors to 
$F$-gerbes. Let $T$ be a $BF$-torsor. By definition $T\to S$ is an \emph{fppf} 
epimorphism. Moreover since $T$ is locally isomorphic to $BF$, and since two 
$F$-torsors are locally isomorphic (because they are both locally trivial), the 
diagonal of $T\to S$ is also an \emph{fppf} epimorphism. Hence 
$T\to S$ is a gerbe. Let $t$ be an object of $T(U)$ for some $S$-scheme~$U$. 
Then the morphism $\mu_t : g\mapsto g.t$ is an isomorphism of stacks $(BF)_U 
\to 
T_U$, which induces an isomorphism 
of group sheaves $\Aut_{(BF)_U}(e) \simeq \Aut_{T_U}(t)$ where $e$ denotes 
the neutral object of~$BF$, i.e. the trivial $F$-torsor. Composing with the 
canonical isomorphism $\Aut_{BF}(e)\simeq F$ we get an isomorphism of group 
sheaves over $U$
\[
 c_t : F_U \lto \Aut_T(t)\, .
\]
Now the stack $T$ together with the collection of all the $c_t$'s is an 
$F$-gerbe. This defines $\Phi$ on objects. The definition of $\Phi$ on 
1-morphisms and on 2-morphisms is straightforward: if $(f : T_1\to T_2, 
\sigma)$ is a morphism of $BF$-torsors then $\Phi(f,\sigma)$ is just $f$. 
Using the functoriality of $\sigma$ and 
Remark~\ref{rem:action_du_neutre_et_morph_equivariant} we check that the 
diagrams of~\ref{def:gerbe}~(iii) commute. If $\tau : f \Rightarrow f'$ is a 
2-isomorphism from $(f,\sigma)$ to $(f', \sigma')$ we define $\Phi(\tau)=\tau$.

It remains to prove that $\Phi$ is indeed a 2-equivalence of categories, i.e. 
it is 2-essentially surjective, and for any two $BF$-torsors $T_1$ and $T_2$, 
the functor $\Phi_{T_1, T_2} : \Hom(T_1,T_2) \to \Hom(\Phi(T_1), \Phi(T_2))$ is 
an equivalence.
Let us first prove that $\Phi_{T_1, T_2}$ is fully faithfull. This amounts to 
say that if $(f,\sigma)$ and $(f', \sigma')$ are two morphisms of 
$BF$-torsors, then any 2-isomorphism $\tau : f \Rightarrow f'$ is 
compatible with $\sigma$ and $\sigma'$ in the sense of~\ref{def_action}~(iii), 
i.e. for any objects $g$ of $BF$ and $x$ of $T_1$, we have
$\tau^{g.x}\circ \sigma_{g}^x = {\sigma'_g}^{x}\circ (g.\tau^x)$. Using the 
fact that we have a functorial identification of the automorphism groups of 
the objects of $T_2$ with $F$, we can define a morphism of stacks $BF\times_S 
T_1 \to F$ that maps a pair $(g,x)$ to the section of $F$ that corresponds to
$\tau^{g.x}\circ \sigma_{g}^x \circ ({\sigma'_g}^{x}\circ (g.\tau^x))^{-1}$.
Now the key fact is that the coarse moduli space of $BF\times_S T_1$ is $S$, so 
that this morphism factorizes through~$S$, by the universal property of the 
coarse moduli space. It then suffices to check the 
required equality when $g=e$ (the neutral object of $BF$), in which case this 
boils down to the functoriality of $\tau$ and 
Remark~\ref{rem:action_du_neutre_et_morph_equivariant}.

Let us prove that $\Phi_{T_1, T_2}$ is essentially surjective. Let $f : T_1 \to 
T_2$ be a morphism of stacks that satisfies the 
condition~\ref{def:gerbe}~(iii). We have to prove that there exists a 
2-isomorphism $\sigma$ such that $(f, \sigma)$ is a morphism of $BF$-torsors.
For this, we recall the fact that if $T$ is an $F$-gerbe, then for any object 
$x$ of $T$, we have an isomorphism of $F$-gerbes:
\[
 \Triv_x : T\lto BF
\]
that maps an object $y$ of $T$ to the sheaf $\Isom(y,x)$. Moreover, if $T$ is a 
$BF$-torsor and $x$ is an object of $T$, then $\Triv_x : T \to BF$ is a 
quasi-inverse to the equivalence $\mu_x : BF \to T$. [This amounts to say that 
any $F$-torsor $g$ is isomorphic to $\Triv_x(\mu_x(g))$, i.e. to 
$\Isom(g.x,x)$. To see this, notice that a section of $g$ corresponds to a 
trivialization $g\simeq e$ where $e$ is the trivial $F$-torsor, hence 
to an isomorphism $g.x\simeq x$, \emph{via} $\mu_x$ and the canonical 
isomorphism $e.x\simeq x$.]
For a fixed object $x$ of $T_1$, the following diagram is 2-commutative
\[
 \xymatrix{
T_1 \ar[rr]^f \ar[rd]_{\Triv_x}
&& T_2\ar[ld]^{\Triv_{f(x)}} \\
&BF
}
\]
where the 2-isomorphism is given by the collection of the isomorphisms of 
$F$-torsors $$\Isom(y,x)\to \Isom(f(y),f(x))$$ induced by $f$.
It follows that the diagram
\[
 \xymatrix{
T_1 \ar[rr]^f 
&& T_2 \\
&BF \ar[ul]^{\mu_x} \ar[ur]_{\mu_{f(x)}}
}
\]
is 2-commutative as well. Now the 2-isomorphism $\mu_{f(x)}\Rightarrow f\circ 
\mu_x$ yields the desired functorial collection of isomorphisms $\sigma_g^x : 
g.f(x)\to f(g.x)$.

Finally let us prove that $\Phi$ is 2-essentially surjective. Let $(T, 
\{c_t\})$ be an $F$-gerbe. We have to construct a morphism $\mu : BF\times_S 
T\to T$ together with a 2-isomorphism $\alpha$ that satisfies the conditions 
of~\ref{def_action} such that $\Phi(T,\mu, \alpha)\simeq (T,\{c_t\})$. Let 
$(g,t)$ be a point of $BF\times_S T$. Recall that $\Triv_t : T\to BF$ is an 
equivalence and choose a quasi-inverse $\mu_t : BF\to T$. We define $\mu(g,t)$ 
to be $\mu_t(g)$. The remaining details (constructing $\alpha$ and checking 
that $\Phi(T,\mu, \alpha)\simeq (T,\{c_t\})$) are left to the reader.
\details{Voir la construction de $\alpha$ en C8, p.5v.}%
\end{proof}

\section{A theorem of the square}
\label{section_thm_square}

For an abelian variety~$A$ over a field~$k$, the classical ``theorem of the 
square'' asserts that for all $x,y\in A(k)$, and for any line bundle~$L$ on~$A$, 
there is an isomorphism $(\mu_{x+y}^*L)\otimes L \simeq 
(\mu_x^*L)\otimes(\mu_y^*L)$, where for a point $x$ of $A$, the morphism $\mu_x 
: A\to A$ is the translation $a\mapsto a+x$. We will need similar facts for 
some 
commutative group stacks: abelian stacks on the one hand, and classifying 
stacks on the other hand. This section is a short interlude devoted to the proof 
of these facts.

\begin{defi}
 Let~$G$ be a commutative group stack over a base scheme~$S$. Let~$L$ be an 
invertible sheaf on~$G$. We denote by $\Lambda(L)$ the so-called ``Mumford 
bundle'' on $G\times_S G$
$$\Lambda(L)=(\mu^*L)\otimes(p_1^*L)^{-1}\otimes(p_2^*L)^{-1}$$
where $\mu$ is the product map and $p_1, p_2$ are the projections from $G\times_S G$ to~$G$. We denote by~$\varphi_L$ the induced morphism of stacks:
$$\varphi_L : G \lto \Pic_{G/S}.$$
Functorially, $\varphi_L$ maps a point $g\in G(S)$ to the class~$[\mu_g^*L\otimes L^{-1}]$.
\end{defi}

\begin{thm}
\label{thm_phi_L_nul}
 Let~$G$ be a commutative group stack over a base scheme~$S$ and let~$L$ be an 
invertible sheaf on~$G$. Assume that one of the following holds:
\begin{itemize}
 \item[(a)] $G$ is an abelian stack and~$[L]\in \Pic_{G/S}^{\tau}(S)$.
\item[(b)] $G$ is the classifying stack~$BF$ of a commutative group scheme~$F$. 
\end{itemize}
Then $\varphi_L=0$.
\end{thm}
\begin{proof}
  By the universal property of the coarse moduli space 
\cite[(3.19)]{LMB}, $\varphi_L$ 
factorizes through the coarse moduli space of~$G$. In the 
case~(b), this moduli space is trivial, hence $\varphi_L$ is 
constant, and equal to 0 since~$\varphi_L(e)=0$.

In the case~(a), let us denote by $A=H^0(G)$ and $F=H^{-1}(G)$. Then there is an 
exact sequence of commutative group stacks:
$$\xymatrix{
0\ar[r]& BF \ar[r]^{i} &G \ar[r]^{\pi} & A\ar[r] & 0\, .
}$$
By assumption, the group~$F$ is finite and flat over~$S$ and~$A$ is an abelian 
scheme over~$S$. We may assume that~$F$ is locally free of rank $n$, so that 
$nF^D=0$. 
By~\ref{Picard_classifiant} there is an exact sequence of group schemes:
$$\xymatrix{
0 \ar[r] &
 \Pic_{A/S}   \ar[r]^{\pi^*} &
 \Pic_{G/S}   \ar[r]^{\chi} &
 F^D\, .
}$$

Let us first assume that there is an invertible sheaf~$M$ 
on~$A$ such that $L\simeq \pi^*M$. Since~$[L]\in\picto_{G/S}$ there exists an 
integer $m>0$ such that $[L]^m\in \Pic^0_{G/S}$. 
The morphism $x\mapsto x^n$ from $\Pic^0_{G/S}$ to itself factorizes 
through $\Pic_{A/S}$ because of the above exact sequence, hence through 
$\Pic^0_{A/S}$ because $\Pic^0_{G/S}$ has connected fibers. This proves 
that $[M]^{nm}\in \Pic^0_{A/S}$ hence $[M]\in \picto_{A/S}$. Using the 
theorem for the abelian scheme~$A$ (see 
\cite[chap.~6~\S 2]{Mumford_GIT}) we see that for any 
object $g$ of $G$, the class of~$\mu_{\pi(g)}^*M\otimes 
M^{-1}$ is trivial in~$\Pic_{A/S}$, hence its pullback 
$[\mu_g^*L\otimes L^{-1}]$ is trivial in $\Pic_{G/S}$ and 
this proves the theorem in this case.

In the general case, by~\ref{morphisme_constant}, the composition~$\chi\circ 
\varphi_L$ from $G$ to $F^D$ must be constant, hence trivial since 
$\varphi_L(e)=0$. This proves that $\varphi_L$ factorizes through $\Pic_{A/S}$, 
and even through~$A^t=\Pic^0_{A/S}$ since $G$ has geometrically 
connected fibers. Let us still denote by~$\varphi_L$ the induced morphism~$G\to 
A^t$. Then 
$\chi(L^n)=0$ in $F^D(S)$ and it follows that $L^n$ comes from~$A$. By the 
previous case we deduce that $\varphi_{L^n}=0$. But $\varphi_{L^n}$ is equal to 
$(\varphi_L)^n$ so $\varphi_L$ factorizes through the kernel $A^t_n$ of the 
isogeny $[n] : A^t\to A^t$. Since~$A^t_n$ is finite,%
\details{il est propre car c'est un sous-schéma fermé de $A^t$ et il est quasi-fini car ses fibres sont finies d'après le thm sur un corps}
using~\ref{morphisme_constant} again we deduce that $\varphi_L$ is constant equal to 0. 
\end{proof}

\begin{remarque}\rm
 In the case~(b), we can give a more precise statement. Let $L$ be an invertible sheaf on~$G$ and let~$x\in G(S)$.
Let us denote by~$\chi_L : F \to \gm$ the character of~$L$ and by~$T_x$ the~$F$-torsor corresponding to the point~$x : S \to BF$.
Then we can prove that the line bundle $\mu_x^*L\otimes L^{-1}$ is isomorphic to~$\pi^*\Lc(\chi_L, T_x)$ where $\pi : G \to S$ is the structural morphism of~$G$ and~$\Lc(\chi_L, T_x)$ is the line bundle on $S$ corresponding to the~$\gm$-torsor~$T_x\wedge^{F,\chi_L}\gm$. 
\end{remarque}

\begin{cor}
\label{lemme_translation_triviale}
 \label{thm_du_carre}
Let $G$ be an abelian stack (resp. the classifying stack~$BF$ of a 
commutative group scheme~$F$).
\begin{itemize}
 \item[(1)] For any $x\in G(S)$, the translation $\mu_x : G\to G$ induces the identity on~$\picto_{G/S}$ (resp. on~$\Pic_{G/S}$).
\item[(2)] Let $x,y\in G(S)$. There is a functorial collection of isomorphisms
$$\xymatrix{\delta_L : (\mu_{x+y}^*L)\otimes L \ar[r]^{\sim} & (\mu_x^*L)\otimes(\mu_y^*L)}$$
for all line bundles $L$ on $G$ such that~$[L]\in\picto_{G/S}(S)$ (resp. for all line bundles~$L$ on~$G$).
\end{itemize}
\end{cor}
\begin{proof}
 (1) is a reformulation of~\ref{thm_phi_L_nul}. Let us prove~(2). Let $x,y\in G(S)$ and let~$L$ be a line bundle on~$G$ such that~$\varphi_L=0$. Then~$[\mu_x^*L\otimes L^{-1}]=0$ hence there is a line bundle~$N$ on~$S$ with an isomorphism $\psi : \mu_x^*L\otimes L^{-1} \to \pi^*N$, where $\pi : G\to S$ is the structural morphism. Since~$\pi\mu_y=\pi$, there is a canonical isomorphism
$$c(N) : \mu_y^*\pi^*N \lto \pi^*N\, .$$
We deduce an isomorphism~$\gamma(L) := 
\psi^{-1}c(N)\mu_y^*(\psi)$ from~$\mu_y^*(\mu_x^*L\otimes 
L^{-1})$ to~$\mu_x^*L\otimes L^{-1}$. This isomorphism does 
not depend on the choice of~$N$ or~$\psi$ since~$c(N)$ is 
functorial. Moreover it is functorial in~$L$. \emph{Via} the 
canonical isomorphism~$\mu_y^*\mu_x^*\simeq \mu_{x+y}^*$, it 
induces the desired~$\delta_L$. 
\end{proof}

\section{The Albanese torsor}
\label{section_Albanese_torsor}

Let $X$ be an algebraic stack over a base scheme $S$ and denote by $f : X\to S$ 
its structural morphism. Assume that $\Oc_S \to f_*\Oc_X$ is universally an 
isomorphism, i.e. for any morphism $T\to S$, the morphism $\Oc_T\to 
(f_T)_*\Oc_{X_T}$ is an isomorphism, where $f_T : X_T\to T$ is the morphism 
obtained from $f$ after the base change $T\to S$. Assume also that $f$ 
locally has sections in the \emph{fppf} topology.
\details{This latter assumption might be replaced by the assumption that the \emph{fppf} sheaf $\shExt^2(\Pic_{X/S}, \gm)$ is zero. We will not use this variant in the sequel.}%
Then we have an exact sequence of group stacks (see~\cite[2.3]{Brochard_Picard}):
$$\xymatrix{0 \ar[r] & B\gm \ar[r]& \champic(X/S) \ar[r] & \Pic_{X/S} \ar[r]& 0\, .}$$
Since $f$ locally has sections, so does 
$D(f^*) : D(\champic(X/S))\to D(B\gm)$, hence $D(f^*)$ is 
an \emph{fppf} epimorphism.
By~\ref{exactness_of_the_dual_sequence} a), the dual sequence is exact: 
$$\xymatrix{0 \ar[r] & D(\Pic_{X/S})\ar[r]^-j  & D(\champic(X/S)) \ar[r]^-{\pi}& \Z \ar[r]& 0}.$$
By~\ref{equiv_extensions_torseurs} this sequence corresponds to a~$D(\Pic_{X/S})$-torsor over~$S$.

\begin{defi}
\label{def_torseur_canonique}
\begin{itemize}
 \item[(i)] The Albanese stack of~$X$ is the commutative group stack
$$ A^0(X):=D(\Pic_{X/S})\, .$$
\item[(ii)] The Albanese torsor of~$X$ is the $A^0(X)$-torsor corresponding to 
the above sequence via Proposition~\ref{equiv_extensions_torseurs}. It is 
denoted by~$A^1(X)$.
\end{itemize}
\end{defi}

If $x$ is an object of $X(U)$ for some $S$-scheme $U$, we 
still denote by $x$ the induced section $U\to X_U:=X\times_S 
U$. Then the pullback $x^*$ defines a morphism of commutative group 
stacks from $\champic(X_U/U)$ to $\champic(U/U)\simeq 
(B\gm)_U$. This is functorial, hence this defines a 
natural morphism of stacks~$\varphi : X \to 
D(\champic(X/S)).$ Let us compute the composition $\pi \circ 
\varphi : X\to \Z$.
Let $x : S\to X$ be an $S$-point of $X$. If we identify $\champic(X/S)$ 
with the stack $Mor(X,B\gm)$ of morphisms from $X$ to $B\gm$ (and $B\gm$ 
with $Mor(S,B\gm)$), then $\varphi(x)$ is the ``evaluation at $x$'' morphism 
$ev_x : Mor(X,B\gm) \to B\gm$ defined functorially by $\tau \mapsto \tau\circ 
x$. Under the identification of $\Hom(B\gm,B\gm)$ with $\Z$, by definition 
$\pi$ is the dual of the canonical morphism $f^* : B\gm \to \champic(X/S)$, 
hence $\pi(\varphi(x))$ is the section of $\Z$ that corresponds to the morphism 
$ev_x\circ f^* : B\gm\to B\gm$. The latter morphism maps an $S$-point $t : S\to 
B\gm$ of $B\gm$ to the composite $t\circ f\circ x$. But $f\circ x=\id_S$, so 
that $ev_x\circ f^*$ is the identity of $B\gm$, which corresponds to the 
section 
$1\in \Z$. Hence $\pi\circ\varphi$ is constant equal to 1. This means that 
$\varphi$ factorizes through the 
open and closed substack $A^1(X)$ of $D(\champic(X/S))$.

\begin{defi}
 \label{def_morphisme_canonique_vers_le_torseur_canonique}
The induced morphism from~$X$ to~$A^1(X)$ is called the Albanese morphism of~$X$ and is denoted by
$$a_X : X \lto A^1(X)\, .$$
\end{defi}

\begin{remarque} \rm
 \label{fonctorialite_Albanese}
Note that the Albanese morphism is functorial. Let $g : X \to Y$ be a morphism between algebraic stacks satisfying the above assumptions. Let us denote by $A^0(g)$ the dual of $g^* : \Pic_{Y/S} \to \Pic_{X/S}$. Then the morphism $D(g^*)$ from $D(\champic(X/S))$ to $D(\champic(Y/S))$ induces an $A^0(g)$-equivariant morphism of torsors $A^1(g) : A^1(X) \to A^1(Y)$. Moreover, there is a canonical 2-isomorphism making the natural square
$$\xymatrix@R=1pc{
X \ar[r]^-{a_X} \ar[d]_{g} &
A^1(X) \ar[d]^{A^1(g)} \\
Y \ar[r]_-{a_Y} & A^1(Y)
}$$
2-commutative.
\details{Le 2-isomorphisme est donné par la collection des 2-isomorphismes $(g\circ x)^* \simeq x^*\circ g^*$ pour tout $U$ et tout $x\in X(U)$.}
\end{remarque}

%

\begin{remarque}\rm
\label{lien_champetre_aG_eG}
 Note that, for any commutative group stack~$G$, the diagram
$$\xymatrix@R=1pc@C=0.5pc{
G \ar[rr]^-{a_G} \ar[dr]_{e_G} &&
D(\champic(G/S)) \ar[dl]^{D(\omega)} \\
&DD(G)
}$$
is 2-commutative. Indeed, if $g$ is an $S$-point of~$G$, then 
$D(\omega)(a_G(g))=a_G(g)\circ \omega=g^*\circ\omega$ maps an element~${\psi\in 
D(G)}$ to the $S$-point $\xymatrix@C=1pc{g^*(\psi) : S \ar[r]^-g & 
G\ar[r]^-{\psi} & \bgm}$ of $\bgm$. The latter is equal to $\psi(g)$, hence 
$D(\omega)\circ a_G=e_G$. In the above diagram, be careful that $a_G$ is not a 
homomorphism of commutative group stacks (it does not even map 0 to 0).
\end{remarque}

\begin{remarque}\rm
For any morphism of commutative group stacks~$A\to \Pic_{X/S}$, we can 
extend the scalars along the dual morphism $A^0(X)\to D(A)$  
and get a morphism $X\to 
A^1(X)\wedge^{A^0(X)}D(A)$ from~$X$ to a~$D(A)$-torsor. In 
particular, we will denote as follows the duals of the 
neutral and torsion components:
$$A_0^0(X):=D(\Pic^0_{X/S}), \\
A_{\tau}^0(X):=D(\picto_{X/S})$$
and the resulting torsors will be denoted by~$A^1_0(X)$ and~$A^1_{\tau}(X)$ and will also be called \emph{Albanese torsor} if no confusion can arise. We have natural morphisms
$$X\lto A^1(X) \lto A^1_{\tau}(X)\lto A^1_0(X).$$
Note that the torsor $A^1_0(X)$ (and similarly for $A^1_{\tau}(X)$) corresponds \emph{via} \ref{equiv_extensions_torseurs} to the exact sequence
$$\xymatrix{
0 \ar[r] & D(\Pic^0_{X/S})\ar[r]  & D(\champic^0(X/S)) \ar[r]& \Z \ar[r]& 0\, .
}$$
\end{remarque}

We prove below (\ref{torseur_canonique_dun_torseur}) that if~$G$ is an abelian stack, or the classifying stack of a multiplicative group, then any $G$-torsor is the Albanese torsor of some stack (actually, it is the Albanese torsor of itself).

\begin{prop}
\label{Picard_dun_torseur}
 Let $G$ be a commutative group stack and let $T$ be a $G$-torsor.
\begin{itemize}
 \item[(a)] If~$G$ is an abelian stack, then there is a canonical isomorphism
$$\xymatrix{\picto_{T/S} \ar[r]^{\sim} & \picto_{G/S}}.$$
\item[(b)] If $G$ is the classifying stack of a commutative group scheme, then 
there is a canonical isomorphism
$$\xymatrix{\Pic_{T/S} \ar[r]^{\sim} & \Pic_{G/S}}.$$
\item[(c)] The isomorphisms of~(a) and~(b) are functorial in the following sense. Let $c_0 : G\to G'$ be a morphism of abelian stacks and let $c_1 : T\to T'$ be a~$c_0$-equivariant morphism of torsors. Then the diagram
$$\xymatrix{\picto_{T'/S} \ar[r]^{(a)} \ar[d]_{c_1^*} &
    \picto_{G'/S} \ar[d]^{c_0^*} \\
\picto_{T/S} \ar[r]_{(a)} &
\picto_{G/S}
}$$
commutes (and analogue statement for~(b)).
\end{itemize}
\end{prop}
\begin{proof}
 Let $t_0\in T(S')$ be an $S'$-point of $T$ where $S'\to S$ is an \emph{fppf} cover. Such a point gives rise to an isomorphism of stacks $\varphi_{t_0} : G_{S'} \to T_{S'}$ mapping $g$ to $g.t_0$, which in turn induces an isomorphism $\varphi_{t_0}^* : \Pic^{\tau}_{T_{S'}/S'} \to \Pic^{\tau}_{G_{S'}/S'}$. By Corollary~\ref{lemme_translation_triviale}, the translation $\mu_x : G \to G$ by an~$S$-point $x$ of $G$ induces the identity on $\Pic^{\tau}_{G/S}$. This proves that $\varphi_{t_0}^*$ does not depend on the choice of $t_0$. By descent this yields the isomorphism~(a). Similarly we get~(b). Let us prove~(c). The statement is \emph{fppf}-local on~$S$, so we may assume that~$T$ has an~$S$-point~$t_0$. Then we have to prove that $c_0^*\circ(\varphi_{c_1(t_0)}^*)=\varphi_{t_0}^*\circ c_1^*$. But $c_1$ is equivariant hence $c_1\circ \varphi_{t_0}=\varphi_{c_1(t_0)} \circ c_0$ and the result follows.
\end{proof}

\begin{prop}
\label{torseur_canonique_dun_torseur}
 Let $G$ be a commutative group stack over~$S$ and let $T$ be a $G$-torsor.
\begin{itemize}
 \item[(a)] Assume that~$G$ is an abelian stack. Let us 
denote by $f_0$ the following composition of isomorphisms:
$$\xymatrix{f_0 : G\ar[r]^{e_G}&
DD(G) \ar[r]^{D(\ref{comp_dual_picto})^{-1}}&
 D(\picto_{G/S})
\ar[r]^-{D(\ref{Picard_dun_torseur})} &
D(\Pic_{T/S}^{\tau})=A^0_{\tau}(T)}.$$
Then the canonical morphism
$$a_T : T \lto A^1_{\tau}(T)$$
is an $f_0$-equivariant isomorphism of torsors.
\item[(b)] Assume that~$G$ is the classifying stack of a Cartier group scheme. Let us denote by $f_0$ the following composition of isomorphisms:
$$\xymatrix{f_0 : G\ar[r]^{e_G}&
DD(G) \ar[r]^{D(\ref{comp_dual_pic})^{-1}} &
 D(\Pic_{G/S}) 
\ar[r]^-{D(\ref{Picard_dun_torseur})} &
D(\Pic_{T/S})=A^0(T)}.$$
Then the canonical morphism
$$a_T : T \lto A^1(T)$$
is an $f_0$-equivariant isomorphism of torsors.
\end{itemize}
\end{prop}
\begin{remarque}\rm
\label{aG_isom}
 In the case $T=G$, it follows 
from~\ref{lien_champetre_aG_eG} that the isomorphism~$f_0$ 
in (a) (resp.~(b)) coincides with the composition 
$\xymatrix@C=1pc{G\ar[r]^-{a_G} & A^1_{\tau}(G) 
\ar[r]^-{t_{-e^*}}& A^0_{\tau}(G)}$ (resp.  with 
$\xymatrix@C=1pc{G\ar[r]^-{a_G} & A^1(G) \ar[r]^-{t_{-e^*}}& 
A^0(G)}$) where~$e$ is a neutral object of~$G$ and 
$t_{-e^*}$ maps a point $\lambda$ to $\lambda-e^*$.
\end{remarque}
\begin{proof} We prove (a) only, (b) is very similar. Since 
a morphism of torsors is always an isomorphism, it suffices 
to prove that $a_T$ is $f_0$-equivariant. Let us first do it 
locally on $S$. Then we can assume that $T$ is the group $G$ 
with its action by translations. We have to find a 
2-isomorphism $\alpha$ making the diagram
$$\xymatrix@C=0,1pc@R=0,1pc{G\times_S G \ar[rr]^-{\mu} \ar[dd]_{f_0\times a_T} &&
G \ar[dd]^{a_T} \\
& ^\alpha\FlecheNE& \\
A^0_{\tau}(G)\times_S A^1_{\tau}(G) \ar[rr] &&A^1_{\tau}(G)}$$
2-commutative. For any two objects $x,y$ of $G(S)$, we need 
a functorial isomorphism between $a_T(x+y)$ and 
$f_0(x).a_T(y)$. Both are objects of 
$\shHom(\champic^{\tau}(G/S), \bgm)$. Let us describe them. 
Identifying $\bgm$ with the Picard stack of $S$ over itself, 
we recall that, by definition, $a_T(y)$ is the pullback 
functor $y^*$, that is, the morphism from 
$\champic^{\tau}(G/S)$ to $\bgm$ that maps a line bundle 
$\Lc$ to $y^*\Lc$. Using~\ref{aG_isom}, we see that the 
morphism $f_0(x).a_T(y)$ is 2-isomorphic to the morphism 
that maps a line bundle $\Lc$ to 
$(x^*\Lc)\otimes(e^*\Lc)^{-1}\otimes(y^*\Lc)$. Hence, to get 
the expected~$\alpha$, we need a functorial isomorphism 
between $(x+y)^*\Lc$ and 
$(x^*\Lc)\otimes(e^*\Lc)^{-1}\otimes(y^*\Lc)$, for any line 
bundle $\Lc$ in $\champic^{\tau}(G/S)$. This is provided by 
corollary~\ref{thm_du_carre}~(2). Since the local 
2-isomorphisms~$\alpha$ are canonical, they glue together 
and yield a global~$\alpha$ over~$S$.
\end{proof}

For further use, we record here a lemma that ensures the compatibility of different isomorphisms introduced so far.

\begin{lem}
 \label{compat_ev_Pic}
Let~$X$ be an algebraic stack over a base scheme~$S$. Assume that $f : X\to S$ locally has sections in the \emph{fppf} topology and that $\Oc_S\to f_*\Oc_X$ is universally an isomorphism.
\begin{itemize}
 \item[(a)] Assume that $P:=\picto_{X/S}$ is a duabelian 
group. Then the composite morphism
$$\xymatrix{P \ar[r]^-{e_P} & DD(P) \ar[r]^-{(\ref{comp_dual_picto})} &
\picto_{A_{\tau}^0(X)/S} \ar[r]^-{(\ref{Picard_dun_torseur}a)^{-1}} &
\picto_{A_{\tau}^1(X)/S} \ar[r]^-{a_{X,\tau}^*} &P
}$$
is the identity of $P$. In particular $a_{X,\tau}^*$ is an isomorphism.
\item[(b)] Assume that $P:=\Pic_{X/S}$ is Cartier (hence $A^0(X)=BP^D$). Then the morphism 
$$\xymatrix{P \ar[r]^-{e_P} & DD(P) \ar[r]^-{(\ref{comp_dual_pic})} &
\Pic_{A^0(X)/S} \ar[r]^-{(\ref{Picard_dun_torseur}b)^{-1}} &
\Pic_{A^1(X)/S} \ar[r]^-{a_X^*} &P
}$$
is the identity of $P$. In particular $a_X^*$ is an isomorphism.
\end{itemize}
\end{lem}
\begin{proof}
 Let us prove (b). It suffices to prove the statement for $S$-points of $P$ (base change). Let~$\lambda\in P(S)$. Since the statement is \emph{fppf} local, we can assume that $X$ has an $S$-point $x_0$ and that~$\lambda$ is induced by an invertible sheaf $\Lc$ on $X$. Then $a_X^*\circ (\ref{Picard_dun_torseur}b)^{-1}$ is induced by the pullback of invertible sheaves along the morphism $\varphi_{x_0^*}^{-1} \circ a_X : X \to A^0(X)\subset D(\champic(X/S))$ that maps a point $x\in X(U)$ to $x^*-x_0^*$. Hence $a_X^*\circ (\ref{Picard_dun_torseur}b)^{-1}(\omega(e_P(\lambda)))$ is the class in $P(S)$ of the invertible sheaf corresponding to the morphism
 $$\xymatrix{X\ar[r]^-{\varphi_{x_0^*}^{-1}\circ a_X}& A^0(X) \ar[r]^-{e_P(\lambda)}& \bgm}.$$
The latter morphism maps a point $x\in X(U)$ to 
$e_P(\lambda)(x^*-x_0^*)=(x^*-x_0^*)(\lambda)=(x^*\Lc)
\otimes(x_0^*\Lc)^{-1}$. Hence it corresponds to the 
invertible sheaf $\Lc\otimes (f^*x_0^*\Lc)^{-1}$ on $X$, 
whose class in $P(S)$ is equal to~$\lambda$. The last 
assertion is obvious since $e_P$ and $(\ref{comp_dual_pic})$ 
are isomorphisms. The proof of (a) is very similar and left 
to the reader. Note by the way that both (a) and (b) are 
related to the following fact. Let us denote by $a_X$ the 
canonical morphism $X\to D(\Pc)$ where $\Pc=\champic(X/S)$. 
Then the composition
$$\xymatrix{\Pc\ar[r]^-{e_P} & DD(\Pc) \ar[r]^-{\omega}& \champic(D(\Pc)) \ar[r]^-{a_X^*} & \Pc}$$
is the identity of $\Pc$ (without any assumption on $f$). 
\end{proof}

\section{Universal properties}
\label{section_PU}

The following theorem generalizes~FGA~VI, théorème~3.3~(iii)~\cite[exp. 236]{FGA}.
Note that an explicit comparison is made in 
Corollary~\ref{lien_avec_variete_Albanese}.
\begin{thm}
 \label{PU_albanese}
Let~$X$ be an algebraic stack over a base scheme~$S$. Assume 
that the structural morphism $f : X \to S$ locally has 
sections in the \emph{fppf} topology, that $\Oc_S \to 
f_*\Oc_X$ is universally an isomorphism, and that the Picard 
functor $\picto_{X/S}$ is a duabelian group 
(see~\ref{def_duabelian_groups}). Then the Albanese morphism
$$a_X : X \lto A^1_{\tau}(X)$$
is initial among maps to torsors under abelian stacks (\ref{def_abelian_stacks}), in the following sense. For any triple~$(B,T,b)$ where $B$ is an abelian stack, $T$ is a~$B$-torsor and~$b : X \to T$ is a morphism of algebraic stacks, there is a triple~$(c_0,c_1,\gamma)$ where $c_0 : A^0_{\tau}(X) \to B$ is a homomorphism of commutative group stacks, $c_1 : A^1_{\tau}(X) \to T$ is a $c_0$-equivariant morphism, and $\gamma$ is a 2-isomorphism $c_1\circ a_X \Rightarrow b$. Such a triple~$(c_0,c_1,\gamma)$ is unique up to a unique isomorphism. 
\end{thm}

The proof of this theorem occupies most of this section. 
Actually we will prove a slightly more precise statement 
(Theorem \ref{phi_equivalence} below).
Keeping the assumptions of~\ref{PU_albanese}, let us first 
define the 2-category $\Tc$ of maps from~$X$ to torsors 
under abelian stacks. An object of $\Tc$ is a triple 
$(B,T,b)$ as in~\ref{PU_albanese}. A morphism from 
$(B,T,b)$ to another object~$(B', T', b')$ is a 
triple~$(c_0,c_1,\gamma)$ where $c_0 : B\to B'$ is a 
homomorphism of abelian stacks, $c_1 : T \to T'$ is a 
$c_0$-equivariant morphism, and $\gamma$ is a 2-isomorphism 
$c_1\circ b \Rightarrow b'$. A 2-morphism from~$(c_0, c_1, 
\gamma)$ to~$(d_0,d_1,\delta)$ is a pair~$(\eps_0, \eps_1)$ 
where $\eps_0 : c_0 \Rightarrow d_0$ is a 2-isomorphism of 
additive morphisms, and $\eps_1 : c_1 \Rightarrow d_1$ is a 
2-isomorphism of \emph{equivariant} morphisms, \emph{i.e.} 
we require that $\eps_1$ is compatible to the 2-isomorphisms 
that make~$c_1$ and $d_1$ equivariant. We also require 
$\eps_1$ to be compatible with $\gamma$ and~$\delta$, 
\emph{i.e.} for any object~$x$ of $X$, 
$\eps_1^{b(x)}=\delta_x^{-1}\circ \gamma_x$.

\begin{lem}
 \label{T_est_discrete}
The automorphism groups of 1-morphisms in $\Tc$ are trivial. Hence, we may 
regard the 2-category~$\Tc$ as an ordinary category (that we still denote 
by~$\Tc$).
\end{lem}
\begin{proof}
 Let $(B,T,b)$ and $(B',T',b')$ be two objects of~$\Tc$ and 
let~$(c_0,c_1,\gamma)$ be a morphism from the one to the 
other. We have to prove that any automorphism~$(\eps_0, 
\eps_1)$ of~$(c_0,c_1,\gamma)$ is trivial. Since~$D(B)$ 
and~$D(B')$ are sheaves (\ref{comp_dual_picto}), any 
automorphism of~$D(c_0)$ is trivial and since $D(.)$ is a 
2-antiequivalence, $\eps_0=\id_{c_0}$. Now let us prove that 
for an object~$t$ of~$T$, $\eps_1^t$ is the identity 
of~$c_1(t)$. The question is local on~$S$ so we may assume 
that $X(S)$ is nonempty. Let $x\in X(S)$. Since $T$ is 
a~$B$-torsor, there exist an object $g$ of $B$ and an 
isomorphism~$\varphi : g.b(x) \to t$. Now using the various 
compatibility conditions on $\eps_1$ and its functoriality, 
we check successively that $\eps_1^{b(x)}$, 
$\eps_1^{g.b(x)}$ and~$\eps_1^{t}$ are trivial.
\end{proof}

Now let~$\Gc$ be the category of pairs~$(A,a)$ where~$A$ is a duabelian group and $a : A\to \picto_{X/S}$ is a morphism of group algebraic spaces. A morphism in~$\Gc$ from~$(A,a)$ to $(A',a')$ is a homomorphism $d : A \to A'$ of group schemes such that $a'd=a$. There is a natural functor~$\Phi : \Tc \to \Gc$ defined as follows. For an object~$(B,T,b)$ of~$\Tc$, $\Phi(B,T,b)$ is the pair~$(A,a)$ where $A=D(B)$ and the homomorphism~$a$ is the composition:
$$\xymatrix{a : D(B) \ar[r]^{(\ref{comp_dual_picto})} &
\picto_{B/S} \ar[r]^{(\ref{Picard_dun_torseur}a)^{-1}}&
\picto_{T/S} \ar[r]^{b^*}&
\picto_{X/S}\, .
}$$
The functor $\Phi$ maps a morphism $(c_0, c_1, \gamma)$ in~$\Tc$  to~$D(c_0)$. This is indeed a morphism in~$\Gc$ using the functoriality of the isomorphisms~\ref{comp_dual_picto} and~\ref{Picard_dun_torseur}.

Theorem~\ref{PU_albanese} exactly means that the 
object~$(A^0_{\tau}(X), A^1_{\tau}(X), a_X)$ is initial in 
the category~$\Tc$. Using~\ref{compat_ev_Pic}~(a), we see 
that $\Phi$ maps this object 
to~$(DD(\picto_{X/S}),e_{\picto_{X/S}}^{-1})$.  Since the 
latter is final in~$\Gc$, Theorem~\ref{PU_albanese} is a 
consequence of~\ref{phi_equivalence} below.

\begin{thm}
\label{phi_equivalence}
 The functor~$\Phi: \Tc \to \Gc$ is an anti-equivalence of categories.
\end{thm}
\begin{proof}
 Let us first prove that~$\Phi$ is essentially surjective. 
Let $(A,a)$ be an object of~$\Gc$. Let $B:=D(A)$. The dual 
of~$a$ is a homomorphism from $D(\picto_{X/S})$ to $B$. 
Let~$T$ be the~$B$-torsor $A^1_{\tau}(X) 
\wedge^{A^0_{\tau}(X),D(a)} B$ obtained from $A^1_{\tau}(X)$ 
by extension of scalars along~$D(a)$ 
(see~\ref{prop_produit_contracte}). There is a natural 
$D(a)$-equivariant morphism $c_1 : A^1_{\tau}(X) \to T$. Let 
$b=c_1\circ a_X$. The triple $(B,T,b)$ is an object 
of~$\Tc$. Let us denote by $(DD(A), \alpha)$ its image 
by~$\Phi$. Using various functorialities 
(\ref{comp_dual_picto}, \ref{Picard_dun_torseur} and 
\ref{fonctorialite_evaluation_map}) and~\ref{compat_ev_Pic}, 
we check that the diagram
$$\xymatrix@!C=0.1pc@R=1.1pc{A\ar[rr]^-{e_A} \ar[rd]_-{a} &&
DD(A) \ar[dl]^-{\alpha}\\
&\picto_{X/S}
}$$
commutes. Hence~$e_A$ is an isomorphism in~$\Gc$ from~$(A,a)$ to~$\Phi(B,T,b)$.

Now let us prove that $\Phi$ is fully faithful.
Let~$\beta=(B,T,b)$ and~$\beta'=(B',T',b')$ be two objects 
of~$\Tc$.
For any~$S$-scheme~$U$,
let $X_U$ denote $X\times_S U$ and let $\Tc_U$ (resp. 
$\Gc_U$) denote the analog of $\Tc$ (resp. $\Gc$) over the 
base scheme $U$. In other words $\Tc_U$ is the category of 
maps from $X_U$ to torsors under abelian stacks (over $U$), 
and similarly for $\Gc_U$.
Let $H_{\Tc}$ and 
$H_{\Gc}$ be the presheaves defined 
by~$H_{\Tc}(U)=\Hom_{\Tc_{U}}(\beta_U, \beta'_U)$ 
and~$H_{\Gc}(U)=\Hom_{\Gc_{U}}(\Phi(\beta'_U),
\Phi(\beta_U))$.  The functor~$\Phi$ extends to a morphism 
of presheaves $H_{\Phi} : H_{\Tc} \to H_{\Gc}$ and saying 
that $\Phi$ is fully faithful precisely means that 
$H_{\Phi}(S)$ is bijective. We will prove that~$H_{\Phi}$ is 
an isomorphism. 
Since $U\mapsto \Hom(D(B')_U, D(B)_U)$ is a sheaf (resp. 
$U\mapsto \Hom(B_U, B'_U)$ is a stack), the presheaf
$H_{\Gc}$ (resp. $H_{\Tc}$)
is actually an \emph{fppf sheaf}, so the 
question is local on~$S$ in the \emph{fppf} topology. We can 
then assume that $X(S)$ is nonempty and we are back to prove 
that~$\Phi$ is fully faithful. Let us fix an object $x_0 \in 
X(S)$. We can now asssume that $T=B$ and that the image $e$ 
of $x_0$ by $b$ is a neutral element of~$B$. 

For the faithfulness, let~$(c_0, c_1, \gamma)$ and~$(d_0,d_1,\delta)$ be two morphisms from $\beta$ to $\beta'$ such that $D(c_0)=D(d_0)$. We want to find a 2-isomorphism $(\eps_0,\eps_1)$ from $(c_0, c_1, \gamma)$ to~$(d_0,d_1,\delta)$. Since $D(.)$ is a 2-antiequivalence, we already know that there is a unique 2-isomorphism of additive morphisms $\eps_0 : c_0 \Rightarrow d_0$ such that $D(\eps_0)$ is the identity of~$D(c_0)$. It remains to construct~$\eps_1 : c_1\Rightarrow d_1$. Let $g$ be an object of $B$. We choose an isomorphism $\varphi : g.e \to g$ and we define $\eps_1^g$ by the commutative diagram:
\begin{equation}
\xymatrix@C=3pc{c_1(g.e) \ar[d]_{c_1(\varphi)} &
c_0(g).c_1(e)\ar[r]^-{\eps_0^g.(\delta_{x_0}^{-1}\circ \gamma_{x_0})}
   \ar[l]_-{\sigma_g^{e}} &
d_0(g).d_1(e) \ar[r]^-{\tau_g^{e}}
   & 
d_1(g.e) \ar[d]^{d_1(\varphi)} \\
c_1(g) \ar[rrr]^{\eps_1^g\  (\textrm{def}^o)}&&& 
d_1(g)
}  \tag{*}
\end{equation}

where $\sigma$ and $\tau$ are the 2-isomorphisms that make~$c_1$ and~$d_1$ equivariant (see~\ref{def_action}~(ii)). 
We claim that~$\eps_1^g$ does not depend on the choice of~$\varphi$. To see this, it suffices to prove that for any automorphism $\psi$ of $g.e$ the following diagram of solid arrows commutes:
$$\xymatrix@C=3pc{c_1(g.e) \ar[d]_{c_1(\psi)} &
c_0(g).c_1(e)\ar[r]^-{\eps_0^g.(\delta_{x_0}^{-1}\circ \gamma_{x_0})}
   \ar[l]_-{\sigma_g^{e}} \ar@{.>}[d]|{c_0(\lambda).\id_{c_1(e)}} &
d_0(g).d_1(e) \ar[r]^-{\tau_g^{e}}
   \ar@{.>}[d]|{d_0(\lambda).\id_{d_1(e)}}& 
d_1(g.e) \ar[d]^{d_1(\psi)} \\
c_1(g.e) &
c_0(g).c_1(e)\ar[r]_-{\eps_0^g.(\delta_{x_0}^{-1}\circ \gamma_{x_0})}
   \ar[l]^-{\sigma_g^{e}} &
d_0(g).d_1(e) \ar[r]_-{\tau_g^{e}}& 
d_1(g.e)
}$$
There is a unique automorphism $\lambda$ of~$g$ such that 
$\lambda.\id_{e}=\psi$. Then the commutativity of the above 
diagram follows from the functoriality of $\sigma$, $\eps_0$ 
and $\tau$. Since $\eps_1^g$ does not depend on 
the choice of~$\varphi$, we see that the collection of 
the~$\eps_1^g$ for all objects~$g$ of~$B$ is functorial, 
\emph{i.e.} it defines a 2-isomorphism $\eps_1 : 
c_1\Rightarrow d_1$. Moreover, for all objects~$t$ and~$g$ 
of~$B$, the following diagram commutes:
$$\xymatrix{c_0(g).c_1(t) \ar[r]^{\sigma_g^t}
\ar[d]_{\eps_0^g.\eps_1^t} &
c_1(g.t) \ar[d]^{\eps_1^{g.t}} \\
d_0(g).d_1(t) \ar[r]_{\tau_g^t} & d_1(g.t)\, .
}$$
(Choose an isomorphism $\varphi : t.e \to t$ to define 
$\eps_1^t$ with the diagram (*) above. Then choose the 
isomorphism $(g.t).e\stackrel{can.}{\to}g.(t.e) 
\stackrel{\id_g.\varphi}{\to} g.t$ to define 
$\eps_1^{g.t}$. To see that the resulting diagram is 
commutative, use among other facts the functoriality of
$\sigma, \tau, \eps_0$ and the hexagon 
condition~\ref{def_action}~(ii).)
This means that $\eps_1$ is a 2-isomorphism of 
\emph{equivariant} morphisms. It remains to prove that for 
any object $x$ of~$X$, $\eps_1^{b(x)}=\delta_x^{-1}\circ 
\gamma_x$. Using the above diagram with $g=t=e$, we see 
that $\eps_0^e.(\delta_{x_0}^{-1}\circ 
\gamma_{x_0})=\eps_0^e.\eps_1^e$ from which we deduce that 
$\eps_1^{b(x_0)}=\delta_{x_0}^{-1}\circ \gamma_{x_0}$. Now, 
the map $x\mapsto \delta_x\circ \eps_1^{b(x)} 
\circ\gamma_x^{-1}$ is functorial and defines a morphism of 
algebraic stacks from~$X$ to $H^{-1}(B')$. By 
Lemma~\ref{morphisme_constant}, this morphism must be 
constant and this proves the desired equality for all~$x$. 
This finishes the proof of the fact that~$\Phi$ is faithful.

To prove that~$\Phi$ is full, we keep the same notations for the objects $\beta$ and~$\beta'$. Let $(D(B),a)$ and $(D(B'), a')$ be their images in~$\Gc$, and let $d : D(B') \to D(B)$ be a morphism in~$\Gc$. In particular $ad=a'$. Since $D(.)$ is an antiequivalence, we know that there is a homomorphism $c_0 : B\to B'$ such that $D(c_0)=d$. We take $c_1=c_0 : B\to B'$ which is clearly $c_0$-equivariant, and to conclude the proof it suffices to check that the morphisms $b'$ and $c_0\circ b$ from $X$ to $B'$ are isomorphic. By~\ref{fonctorialite_Albanese} there is a commutative diagram
$$\xymatrix@R=1pc{X \ar[r] \ar[d]_b &
A^1_{\tau}(X) \ar[d]^{A^1_{\tau}(b)} \\
B \ar[r] &A^1_{\tau}(B).
}$$
With the notations of~\ref{torseur_canonique_dun_torseur} the bottom horizontal arrow is an $f_0$-equivariant isomorphism of torsors. Hence we deduce a morphism of torsors $A^1_{\tau}(X) \to B$ which is equivariant under the homomorphism
$$\xymatrix{
f_0^{-1}\circ D(b^*) : D(\picto_{X/S}) \ar[r]^-{D(b^*)}&
D(\picto_{B/S}) \ar[r]^-{D(\ref{Picard_dun_torseur})^{-1}}&
D(\picto_{B/S}) \ar[r]^-{D(\ref{comp_dual_picto})}&
DD(B) \ar[r]^-{e_G^{-1}} &B
}$$
and through which $b$ factorizes (here, since $T=B$, the 
isomorphism~(\ref{Picard_dun_torseur}) is actually the identity). Choosing a 
trivialization of the torsor 
$A^1_{\tau}(X)$, we find a morphism $u : X \to D(\picto_{X/S})$ such that~$b$ 
factorizes as $f_0^{-1}\circ D(b^*)\circ u$. Similarly~$b'$ factorizes through 
the same $u$ as $f_0'^{-1}\circ D(b'^*)\circ u$. It now suffices to prove that 
the homomorphisms of commutative group stacks $c_0f_0^{-1}D(b^*)$ and 
$f_0'^{-1}D(b'^*)$ are isomorphic. Applying~$D(.)$, this is equivalent to 
$$D(f_0^{-1}D(b^*)) \circ d =D(f_0'^{-1}\circ D(b'^*)).$$ 
Using~\ref{fonctorialite_evaluation_map}, we see that $D(f_0^{-1}D(b^*))$ is 
equal to $e_{\picto_{X/S}}\circ a$ and 
$D(f_0'^{-1}D(b'^*))$ to~$e_{\picto_{X/S}}\circ a'$. Hence the desired equality follows from the assumption~$ad=a'$.
\end{proof}

\begin{cor}
 \label{lien_avec_variete_Albanese}
Under the assumptions of~\ref{PU_albanese}, assume that $\Pic_{X/S}$ has an abelian subscheme~$A$, the underlying subset of which is $\Pic^0_{X/S}$. Let us denote by $\Alb^1(X)$ the torsor obtained from $A^1(X)$ by extension of scalars along $D(\Pic_{X/S})\to D(A)$, and by $u$ the  composed morphism: 
$$\xymatrix{
u : X\ar[r]^{a} & A^1(X) \ar[r] &\Alb^1(X)\, .
}$$
Then:
\begin{itemize}
 \item[(i)] $u$ is initial among morphisms from $X$ to torsors under abelian schemes. In particular this proves that $u$ coincides with the classical Albanese morphism of~FGA~VI, théorème~3.3~(iii) \cite[exp. 236]{FGA}.
\item[(ii)] \emph{Via} the canonical morphism $A^1_{\tau}(X) \to \Alb^1(X)$, the classical Albanese torsor  $\Alb^1(X)$ is the coarse moduli space of the Albanese stack~$A^1_{\tau}(X)$. If $\Pic_{X/S}/\picto_{X/S}$ is a twisted lattice, then $\Alb^1(X)$ is also the coarse moduli space of $A^1(X)$.
\end{itemize}
\end{cor}
\begin{proof}
Let $Q$ denote the quotient sheaf $\picto_{X/S}/A$. By 
Artin's representability theorem \cite[10.4]{LMB} it is an algebraic space. 
It is proper and flat because so is $\picto_{X/S}$. We 
first prove that any morphism~$B\to \picto_{X/S}$ where $B$ 
is an abelian scheme factorizes through $A$. If $S$ is the 
spectrum of a field then $A$ is equal to 
$(\Pic^0_{X/S})_{\textrm{réd}}$ and the claim follows from 
the universal property of the reduced subscheme. In the 
general case it suffices to prove that the 
composition~$\lambda : B\to Q$ is zero. But this morphism is 
constant on the fibers (by the case where~$S$ is a field) 
hence it is constant by~\cite[6.1]{Mumford_GIT}. Since it 
maps 0 to 0 it is then the trivial morphism.

Let us now prove that $Q$ is finite. It only remains to prove that it is 
quasi-finite, so for this question we may assume that $S$ is the spectrum of a 
field. By assumption there is an exact sequence $0\to B\to \picto_{X/S}\to F \to 
0$ where $F$ is a finite flat commutative group scheme and $B$ is an abelian 
scheme. By the above $B\to \picto_{X/S}$ factorizes through $A$, and we get a 
morphism $F\to Q$. This morphism is proper, smooth and surjective because so 
are the morphisms $\picto_{X/S}\to F$ and $\picto_{X/S}\to Q$, whence 
the assertion.

Let us prove (i). Let $b : X\to T$ be a morphism from $X$ to 
a $B$-torsor, where $B$ is an abelian scheme. 
By~\ref{PU_albanese}, there exists a homomorphism $c_0 : 
A_{\tau}^0(X)\to B$ and a $c_0$-equivariant morphism of 
torsors $c_1 : A_{\tau}^1(X)\to T$ such that $c_1\circ 
a_{X,\tau}=b$. The dual $D(c_0) : B^t \to \picto_{X/S}$ 
factorizes (uniquely) through $A$, by the above. Dualizing, 
we get a morphism $\ov{c_0} : D(A) \to B$ through which 
$c_0$ factorizes. Then there is a unique 
$\ov{c_0}$-equivariant morphism $\ov{c_1} : \Alb^1(X) \to B$ 
such that $\ov{c_1}u=b$
\details{(La question est locale... on peut donc supposer que $X$ a un $S$-point et alors les torseurs sont triviaux donc c'est évident.)}
and this concludes the proof.

Now let us prove (ii). The assertion is equivalent to saying that $A^1_{\tau}(X) \to \Alb^1(X)$ is an \emph{fppf}-gerbe. This question is \emph{fppf}-local on $S$ so we may assume that $X$ has an $S$-point. Then the torsors are trivial and we have to prove that $A^0_{\tau}(X)$ is a gerbe over $A^t$. But, since $Q$ is finite and flat, the exact sequence $0\to A\to \picto_{X/S}\to Q\to 0$ induces by~\ref{exactness_of_the_dual_sequence} an exact sequence $0\to BQ^D\to A^0_{\tau}(X) \to A^t\to 0$. If $M:=\Pic_{X/S}/\picto_{X/S}$ is a twisted lattice, then we see similarly that $A^0(X)$ is an $M^D$-gerbe over $A^0_{\tau}(X)$ and the last assertion follows.
\end{proof}

\begin{exemple}\rm
 Let $k$ be a field and let $C$ be a geometrically integral smooth projective curve over~$k$. Then $\Pic_{C/k}^{\tau}=\Pic^0_{C/k}$ and it is an abelian variety.
\details{(Voir par exemple~\cite{FGA_Explained}~9.5.4, 9.5.19 et~9.6.21.)}%
Hence by~\ref{lien_avec_variete_Albanese} the Albanese stack~$A^1_{\tau}(X)$ is isomorphic to the classical Albanese torsor~$\Alb^1(C/k)$. 
\end{exemple}

\begin{exemple}\rm
\label{exemple:root_stack}
 Let $f : X\to S$ be a proper and flat morphism of schemes, 
with $f_*\Oc_X$ universally isomorphic to $\Oc_S$ and $S$ 
noetherian. Let $\Lc$ be an invertible sheaf on $X$ and $r$ 
a positive integer. Let us denote by~$\Xc=\sqrt[r]{\Lc}$ the 
stack that classifies $r$-th roots of $\Lc$ 
(see~\cite[2.2.6]{Cadman_USTITCOC} 
or~\cite[5.3]{Brochard_Picard} for the precise definition). 
Let us describe the Albanese morphism of $\Xc$. The 
stack $\Xc$ is a $\mu_r$-gerbe over $X$. By 
\cite[5.3]{Brochard_Picard}, there is an exact sequence of 
sheaves
$0\to \Pic_{X/S} \to \Pic_{\Xc/S}\to \Z/r\Z \to 0.$
By~\ref{exactness_of_the_dual_sequence} the dual sequence
$$0\lto B\mu_r \lto A^0(\Xc)\lto A^0(X)\lto 0$$
is exact hence $A^0(\Xc)$ is a $\mu_r$-gerbe over $A^0(X)$. Let us now assume that $\Pic^{\tau}_{X/S}$ is an abelian scheme over $S$ (hence equal to $\Pic^0_{X/S}$).
Let $Q$ be the image \emph{fppf} sheaf of 
$\Pic^{\tau}_{\Xc/S}$ in $\Z/r\Z$. Assume that $Q$ is an 
open and closed group subscheme of $\Z/r\Z$. 
By~\cite[3.3.2]{Brochard_finiteness} there is an exact 
sequence
$0\to \picto_{X/S} \to \picto_{\Xc/S}\to Q \to 0.$ Then by~\ref{exactness_of_the_dual_sequence}, the dual sequence 
 $$0\to D(Q) \to A^0_{\tau}(\Xc)\to A^0_{\tau}(X)\to 0$$
is exact hence $A^0_{\tau}(\Xc)$ is a $Q^D$-gerbe over $A^0_{\tau}(X)$. To describe the morphism $\Xc \to A^1_{\tau}(\Xc)$, let us assume that $\Xc$ has an $S$-point $x_0$ (this is true \emph{fppf}-locally on $S$). Then using the point $a_{\Xc}(x_0)=x_0^*$, the torsor trivializes and we will describe the resulting morphism $\Xc\to A^0_{\tau}(\Xc)$. There is a commutative diagram
$$\xymatrix@R=1pc@C=1pc{\Xc
\ar[r]^-{a_{\Xc}} \ar[d]_-{\pi} & A^0_{\tau}(\Xc)
\ar[d]^-{A^0_{\tau}(\pi)} \\
X \ar[r]_-{a_X} &A^0_{\tau}(X)
}$$
in which, by~\ref{lien_avec_variete_Albanese}, the bottom 
map $a_X$ is the classical Albanese morphism, mapping 
$\pi(x_0)$ to $0$. Let $\Xc_0$ denote the fiber 
$\pi^{-1}(\pi(x_0))$ and $\Ac_0\simeq D(Q)$ the kernel of 
$A^0_{\tau}(\pi)$, and let us compute the induced morphism 
$\Xc_0\to \Ac_0$. The commutative diagram
$$\xymatrix@R=1pc@C=1pc{\Xc_0
\ar[r]^-{a_{\Xc_0}} \ar[d]_-{i} & A^0_{\tau}(\Xc_0)
\ar[d]^-{A^0_{\tau}(i)} \\
\Xc \ar[r]_-{a_{\Xc}} &A^0_{\tau}(\Xc)
}$$
induces a factorization of the morphism $\Xc_0\to \Ac_0$ 
through $a_{\Xc_0}$. But the stack $\Xc_0$ is a trivial 
$\mu_r$-gerbe over $S$. Hence $a_{\Xc_0}$ is an isomorphism 
(see~\ref{lien_champetre_aG_eG} and~\ref{aG_isom}). In the 
end, through the above-mentioned identifications, the 
morphism $\Xc_0\to \Ac_0$ is the dual of the inclusion of 
$Q$ into $\Z/r\Z$.

To conclude this example, let us discuss a little bit this assumption that $Q$ 
is an open and closed group subscheme of~$\Z/r\Z$. If $[\Lc]\in 
\picto_{X/S}$ or if $\Lc$ has an $r$-th root on $X$, then 
$Q=\Z/r\Z$ so the condition holds. If $X$ is 
smooth and projective over $S$, then $Q$ is automatically a 
closed subscheme of $\Z/r\Z$. However, even with that 
latter assumption, it seems that $Q$ does not need to be 
open in $\Z/r\Z$. Its openness is equivalent to the flatness 
of $\picto_{\Xc/S}$ over $S$. (Indeed, since the morphism $Q\to \Z/r\Z$ is a 
monomorphism of finite type, it is an open immersion if and only if it is 
flat, i.e. if and only if $Q$ is flat over $S$. Then the claim follows from 
the fact that flatness is local on the source, and that $\picto_{\Xc/S}\to Q$ 
is an \emph{fppf} epimorphism.) In particular $Q$ is open if
$\Pic_{X/S}$ is flat over $S$.
\details{
Plus de détails (beaucoup plus...) dans C2 p.25 et suivantes.
}
\end{exemple}

\begin{exemple}\rm
\label{exemple:twisted_curves}
 Very similarly, let us consider the case of a smooth twisted curve $\Cc \to C$ as defined by Abramovich and Vistoli~\cite{Abramovich_Vistoli_CSSM}. By Cadman~\cite[2.2.4 and 4.1]{Cadman_USTITCOC}, $\Cc$ can be described as a root stack $\sqrt[r_1]{(\Lc_1,s_1)}\times_S\dots \times_S \sqrt[r_n]{(\Lc_n, s_n)}$, where $\Lc_i$ is an invertible sheaf on $C$ and $s_i$ is a global section of $\Lc_i$. Then by~\cite[5.4]{Brochard_Picard}
 $\Pic_{\Cc/S}$ is an extension of $\Z/r_1\Z\times_S \dots \times_S \Z/r_n\Z$ by $\Pic_{C/S}$. Considering the dual sequence, which is exact by~\ref{exactness_of_the_dual_sequence}, we see as above that $A^0(\Cc)$ is a gerbe over $A^0(C)$, banded by $\mu_{r_1}\times\dots \times \mu_{r_n}$. Let $Q$ denote the image of $\picto_{\Cc/S}$ in $\prod \Z/r_i\Z$. It is an open and closed group subscheme (because $C$ is a smooth and projective family of curves). The group $\picto_{C/S}$ is an abelian scheme, hence the exact sequence $0\to \picto_{C/S}\to \picto_{\Cc/S}\to Q\to 0$ shows that $A^0_{\tau}(\Cc)$ is an abelian stack (it is a $Q^D$-gerbe over the abelian scheme $(\Pic^0_{C/S})^t$). 
\end{exemple}

\begin{thm}
 \label{PU_cas_des_groupes_de_tm}
Let~$X$ be an algebraic stack over a base scheme~$S$. Assume that the structural morphism $f : X \to S$ locally has sections in the \emph{fppf} topology, that $\Oc_S \to f_*\Oc_X$ is universally an isomorphism, and that the Picard functor $\Pic_{X/S}$ is a Cartier group scheme (see~\ref{Cartier_groups}). Then the Albanese morphism of~$X$ 
$$a_X : X \lto A^1(X)$$
is initial among maps to torsors under classifying stacks of Cartier group schemes (that is, maps to gerbes banded by Cartier group schemes). 
\end{thm}
\begin{proof}
 The proof of~\ref{PU_albanese} holds verbatim, 
\emph{mutatis mutandis}. More precisely, replace everywhere 
$\Pic^{\tau}$ with $\Pic$, duabelian with Cartier, 
$A^i_{\tau}$ with $A^i$, ``abelian stack'' with ``classiying 
stack of a Cartier group scheme'', \ref{comp_dual_picto} 
with~\ref{comp_dual_pic}, \ref{Picard_dun_torseur} (a) with 
\ref{Picard_dun_torseur} (b), and~\ref{compat_ev_Pic} (a) 
with \ref{compat_ev_Pic} (b). You also need a variant 
of~Lemma~\ref{morphisme_constant} where the affine target 
$Y$ is replaced with a Cartier group scheme. To prove it, 
we proceed as follows. If $Y$ is a finite flat group 
scheme, or of multiplicative type, 
then~\ref{morphisme_constant} itself applies. If $Y$ is a 
twisted lattice then any $S$-morphism $X\to Y$ must be 
constant because the geometric fibers of $X$ are connected.  
For an arbitrary Cartier group scheme, the result 
follows by induction on the number of extensions. 
Notice in particular that the analogs in this context 
of~\ref{T_est_discrete} and~\ref{phi_equivalence} hold.
\details{Ici il faut aussi vérifier que le lemme~\ref{morphisme_constant} est encore valable lorsque $Y$ n'est plus un schéma affine sur $S$ mais est un groupe de Cartier. Pour les groupes finis plat ou les groupes de type multiplicatif on peut utiliser~\ref{morphisme_constant} tel quel. Pour les groupes constants tordus ce n'est pas très difficile. Pour un groupe de Cartier quelconque on le fait par récurrence sur le nombre d'extensions nécessaires pour dévisser le groupe de Cartier.}
\end{proof}

\section{Application to rational varieties}
\label{section_rational_varieties}

The ``universal torsors'' and the so-called ``elementary obstruction'' were introduced by Colliot-Thélène and Sansuc in a series of Notes (\cite{Colliot-Thelene_Sansuc_Note1, Colliot-Thelene_Sansuc_Note2, Colliot-Thelene_Sansuc_Note3}) and in the foundational article~\cite{Colliot-Thelene_Sansuc_La_Descente_Duke}. These are the key ingredients of a general method which proved to be very useful in the study of rational points of some algebraic varieties. One of the main tools is the following fundamental exact sequence. Let $f : X\to S$ be a morphism of schemes. Throughout this section we assume that $\Oc_S\to f_*\Oc_X$ is universally an isomorphism and that $f$ locally has sections in the \emph{fppf} topology.

\begin{lem}[{\cite[prop. 1]{Colliot-Thelene_Sansuc_Note2}, \cite[2.0]{Colliot-Thelene_Sansuc_La_Descente_Duke}}]
\label{suite_exacte_fondamentale_Colliot}
Let $G$ be an $S$-group scheme of multiplicative type and of finite type. Then there is a functorial exact sequence:
$$\xymatrix@C=1.42pc{0\ar[r] & H^1(S,G) \ar[r]^-{i_1}&
H^1(X,G) \ar[r]^-{\chi} &
\Hom_{S-gp}(G^D, \Pic_{X/S}) \ar[r]^-{\partial} &
H^2(S,G) \ar[r]^-{i_2}&
H^2(X,G)
}$$
\end{lem}
\details{
If $X$ has a $k$-rational point, it defines retractions for $i_1$ and $i_2$. In particular, the above sequence for $G_0$ induces a short exact sequence:
$$\xymatrix{0\ar[r] & H^1(k,G_0) \ar[r]^-{i_1}&
H^1(X,G_0) \ar[r]^-{\chi} &
\Hom_{k-gp}(G_0^D, \Pic_{X/k}) \ar[r] & 0
}$$
}

\begin{remarque}\rm
 In this sequence, $i_1$ and $i_2$ are given by pullback along~$f$, and the cohomology groups are computed with respect to the \emph{fppf} topology. Note that if $G$ is smooth they coincide with the étale ones.
\end{remarque}

Assume moreover that $\Pic_{X/S}$ is representable by a twisted lattice\footnote{This is the case if $S$ is the spectrum of a field $k$, and $X$ is proper and smooth over $k$, and $\ov{k}$-rational.} (\emph{i.e.} étale-locally $\Pic_{X/S}$ is the constant sheaf associated with an ordinary free abelian group of finite rank). Its Cartier dual~$G_0:=\Pic_{X/S}^D$ is of multiplicative type and of finite type. According to~\cite{Colliot-Thelene_Sansuc_Note2}, a \emph{universal torsor} is by definition a $G_0$-torsor $T$ over $X$ (or, by a slight abuse of notation, its class in $H^1(X,G_0)$) such that $\chi(T)$ is the canonical isomorphism $\lambda_0 : \Pic_{X/S}^{DD} \to \Pic_{X/S}$. 
For each $x\in X(S)$, the unique universal torsor $T_x\in H^1(X,G_0)$  whose pullback along $x$ is trivial is called the \emph{universal torsor associated with $x$}. Note that, since $x^*T_x$ is trivial, the $S$-point $x$ is in the image of $T_x(S)\to X(S)$. The \emph{elementary obstruction} is the class $\partial(\lambda_0)$ in the group $H^2(S, G_0)$. This is an obstruction to the existence of universal torsors. If the elementary obstruction vanishes, then there is a universal torsor, and some natural questions about the $S$-points of $X$ reduce to the same questions on $T$, which in general is arithmetically simpler than $X$.

In this section, we will relate these constructions to the Albanese torsor of section~\ref{section_Albanese_torsor}, giving by the way a \emph{geometric} description of the elementary obstruction (in terms of a gerbe). Let $a : X \to A^1(X)$ be the canonical morphism of~\ref{def_morphisme_canonique_vers_le_torseur_canonique}. 
Note that by~\ref{particular_case_of_a_sheaf}, $A^1(X)$ is a $G_0$-gerbe.

Let $G$ be an $S$-group of multiplicative type and of finite type. We will 
define maps $\chi'$ from $H^1(X,G)$ to $\Hom_{S-gp}(G^D, \Pic_{X/S})$ and 
$\partial'$ from $\Hom_{S-gp}(G^D, \Pic_{X/S})$ to $H^2(S,G)$ and compare them 
to the $\chi$ and $\partial$ of~\ref{suite_exacte_fondamentale_Colliot}. Let $T$ 
be a $G$-torsor over $X$. Then $T$ corresponds to a morphism $\varphi_T : X\to 
BG$. By \ref{PU_cas_des_groupes_de_tm} there is up to isomorphism a unique 
triple $(c_0, c_1, \gamma)$ such that $c_0 : D(\Pic_{X/S}) \to BG$ is a morphism 
of commutative group stacks, $c_1 : A^1(X) \to BG$ is a $c_0$-equivariant 
morphism, and $\gamma$ makes the diagram
$$\xymatrix{X \ar[r]^-a \ar[rd]_{\varphi_T} &A^1(X)\ar[d]^{c_1}\\
& BG}$$
2-commutative. We define $\chi'(T)$ to be the composition of the dual $D(c_0)$ with the canonical isomorphisms as follows:
$$\xymatrix{\chi'(T) : G^D \ar[r]^-{(\ref{particular_case_of_a_classifying_stack})^{-1}}&
D(BG) \ar[r]^-{D(c_0)} &
DD(\Pic_{X/S}) \ar[r]^-{e_{\Pic_{X/S}}^{-1}} &
\Pic_{X/S}.
}$$
Note that since the commutative group stacks involved here are sheaves, 
$\chi'(T)$ does not depend on the choice of~$c_0$ into its isomorphism class.
Now let $\lambda\in \Hom_{S-gp}(G^D, \Pic_{X/S})$. We define $\partial'(\lambda)$ to be the class in $H^2(S, G)$ of the $G$-gerbe (that is, the $BG$-torsor), obtained from the $D(\Pic_{X/S})$-torsor $A^1(X)$ by extension of scalars (see~\ref{prop_produit_contracte} b)) along the composed morphism
$$\xymatrix{D(\Pic_{X/S})  \ar[r]^-{D(\lambda)} &
 D(G^D)\ar[r]^-{D(\ref{particular_case_of_a_classifying_stack})} &
DD(BG) \ar[r]^-{e_{BG}^{-1}} &
BG.
}$$

\begin{prop}
\label{prop_description_morphismes_CTS} \ 
 \begin{itemize}
  \item[(i)] $\chi'$ and $\partial'$ are homomorphisms.
\item[(ii)] $\chi'=\chi$.
\item[(iii)] The sequence
 $$\xymatrix@C=1.42pc{0\ar[r] & H^1(S,G) \ar[r]^-{i_1}&
H^1(X,G) \ar[r]^-{\chi'} &
\Hom_{S-gp}(G^D, \Pic_{X/S}) \ar[r]^-{\partial'} &
H^2(S,G) \ar[r]^-{i_2}&
H^2(X,G)
}$$
 is exact.
\item[(iv)] The elementary obstruction $\partial(\lambda_0)$ vanishes if and only if the gerbe~$A^1(X)$ is trivial (equivalently, if and only if $\partial'(\lambda_0)=0$).
 \end{itemize}
\end{prop}
\begin{proof}
The map $\partial'$ is a homomorphism due to~\ref{dual_dun_produit_de_morphismes} and~\ref{prop_extension_scalaires_multiplicative_en_phi}. 
To prove that $\chi'$ is a homomorphism, we can invoke (ii), 
or proceed directly as follows. Let $T$ and $T'$ be two 
$G$-torsors over $X$ and $(c_0, c_1), (c_0', c_1')$ the 
associated pairs of morphisms as above. The 
product map $c_1.c_1'$ from~$A^1(X)$ to~$BG$ is 
equivariant under $c_0.c_0'$ and maps the object $a\in 
A^1(X)(X)$ to $T\wedge^G T'$. Looking at the definition of 
$\chi'$ and using~\ref{dual_dun_produit_de_morphismes}, we 
deduce the expected relation $\chi'(T\wedge^G 
T')=\chi'(T).\chi'(T')$.

Let us prove (ii). Let $T$ be a $G$-torsor over~$X$. Let us prove that the morphisms $\chi(T)$ and $\chi'(T)$ are equal. By \cite[1.5.2 (ii)]{Colliot-Thelene_Sansuc_La_Descente_Duke}, the morphism~$\chi(T) : G^D \to \Pic_{X/S}$ is the morphism that maps a character $\varphi : G \to \gm$ to the point of $\Pic_{X/S}$ corresponding to the $\gm$-torsor $T\wedge^{G,\varphi}\gm$. By construction, the morphism~$c_0$ is such that $e^{-1}_{\Pic_{X/S}}\circ D(c_0)=\varphi_T^*\circ (\ref{Picard_dun_torseur})^{-1}\circ (\ref{comp_dual_pic})$ (see the proof of~\ref{PU_cas_des_groupes_de_tm}). Here (\ref{Picard_dun_torseur}) is the identity, hence $\chi'(T)=\varphi_T^*\circ(\ref{comp_dual_pic})\circ (\ref{particular_case_of_a_classifying_stack})^{-1}$ and it suffices to prove that the following diagram commutes:
$$\xymatrix{
&&G^D \ar[rrd]^{\chi(T)} \ar[d]^{(\ref{Picard_classifiant})^{-1}}\\
D(BG)\ar[rr]_{(\ref{comp_dual_pic})}
   \ar[urr]^{(\ref{particular_case_of_a_classifying_stack})} &&
\Pic_{BG/S} \ar[rr]_{\varphi_T^*}&&
\Pic_{X/S}
}$$
The left triangle commutes because of~\ref{rem_dual_classifiant_via_Picard}. It remains to prove that the right triangle commutes. Since all the constructions commute with base change it suffices to do it on $S$-points. Let $\varphi : G\to \gm$ be a character of $G$. By construction of the morphism~(\ref{Picard_classifiant})$^{-1}$, it maps $\varphi$ to the class of an invertible sheaf $\Lc(\varphi)$ such that $\varphi_T^*\Lc(\varphi)$ is the class of the $\gm$-torsor $T\wedge^{G,\varphi}\gm$ in~$\Pic_{X/S}(S)$, as desired.

 Now let us prove the exactness of the sequence~(iii) in 
$\Hom_{S-gp}(G^D, \Pic_{X/S})$. Let $T$ be a $G$-torsor over 
$X$. The morphism $c_1 : A^1(X)\to BG$ in the above 
construction of $\chi'(T)$ is equivariant under $c_0$. But 
using~\ref{fonctorialite_evaluation_map} we see that~$c_0$ 
is equal to the morphism 
$\xymatrix@C=1pc{D(\Pic_{X/S})  \ar[r]&  BG}$ along which we 
extend the gerbe $A^1(X)$ to define $\partial'(\chi'(T))$. 
This proves that the gerbe we get by extension of scalars is 
trivial, hence $\partial'(\chi'(T))=0$. Conversely, if 
$\partial'(\lambda)=0$ for some morphism $\lambda : G^D \to 
\Pic_{X/S}$, this means that the gerbe defining this class 
is trivial. We have a morphism $A^1(X) \to BG$ that is 
equivariant under $e_{BG}^{-1}\circ 
D(\ref{particular_case_of_a_classifying_stack}) \circ 
D(\lambda)$. Composing with $a : X \to A^1(X)$, we get a 
morphism $X \to BG$ which corresponds to a $G$-torsor $T$ 
and we check that $\chi'(T)=\lambda$. Similarly, the 
exactness in~$H^2(S, G)$ is a consequence of the 
universal property~\ref{PU_cas_des_groupes_de_tm}.

Obviously $\partial(\lambda_0)=0$ if and only if $\partial'(\lambda_0)=0$, by~(iii). But the map along which we extend the scalars to construct the gerbe defining~$\partial'(\lambda_0)$ is an isomorphism.
\details{Par~\ref{compatibility_of_can_isoms_with_ev_maps} et \ref{fonctorialite_evaluation_map} c'est l'isomorphisme canonique de~\ref{particular_case_of_a_sheaf}.}%
Hence $\partial'(\lambda_0)=0$ if and only if the gerbe~$A^1(X)$ is trivial.
\end{proof}

\begin{remarque}\rm
 As one of the referees points out, it is certainly possible, but perhaps 
painful and not worth the effort, to determine whether $\partial$ 
and $\partial'$ are actually equal or not (they might be opposite). Since both 
maps are functorial (covariant) in $G$, this is equivalent to 
$\partial(\lambda_0)=\partial'(\lambda_0)$. In other words, $\partial=\partial'$ 
if and only if the elementary obstruction $\partial(\lambda_0)$ is the class in 
$H^2(S, G_0)$ of the gerbe $A^1(X)$.
\end{remarque}

Now let us describe the universal torsors $T_x$ in terms of 
the Albanese morphism $a : X\! \to\! A^1(X)$. Let $\eta_0 : 
BG_0\to A^0(X)$ be the canonical isomorphism 
of~\ref{particular_case_of_a_sheaf}. If $x\in X(S)$ is an 
$S$-point of $X$, its image $a(x)$ induces an 
$\eta_0$-equivariant isomorphism of torsors (that is, a 
trivialization of the $G_0$-gerbe $A^1(X)$)
$\xymatrix@C=1pc{BG_0 \ar[r]^-{\sim} & A^1(X)}$
which maps the trivial $G_0$-torsor to $a(x)$. Let $T'_x$ denote the torsor over $X$ corresponding to the composition
$$\xymatrix{X\ar[r]^-a & A^1(X)& BG_0 \ar[l]_-{\sim}}.$$

\begin{prop}
 The torsor $T'_x$ is (isomorphic to) the universal torsor $T_x$ associated with $x$.
\end{prop}
\begin{proof}
 By construction $T'_x$ is a $G_0$-torsor over $X$ whose 
pullback along~$x$ is trivial. Hence is suffices to prove 
that $\chi(T'_x)\in \Hom_{S-gp}(G_0^D, \Pic_{X/S})$ is equal 
to $\lambda_0$. Owing 
to~\ref{prop_description_morphismes_CTS} (ii), $\chi(T'_x)$ 
is equal to the composition \vskip-8mm
$$\xymatrix{\Pic_{X/S}^{DD}=G_0^D \ar[r]^-{(\ref{particular_case_of_a_classifying_stack})^{-1}} &
D(BG_0) \ar[r]^-{D(\eta_0)^{-1}} &
DD(\Pic_{X/S}) \ar[r]^-{e_{\Pic_{X/S}}^{-1}} &
\Pic_{X/S}
}$$
which is equal to $\lambda_0$ due to Lemma~\ref{compatibility_of_can_isoms_with_ev_maps}.
\end{proof}

\section{Applications to Grothendieck's section conjecture}
\label{section_CS}

Let $X$ be an algebraic stack over a field~$k$. Borne and 
Vistoli define a Nori fundamental gerbe for $X$ as 
follows (see~\cite[5.1]{Borne_Vistoli_Fundamental_Gerbe}). 
It is a gerbe $\Pi_{X/k}$ with a $k$-morphism $X\to 
\Pi_{X/k}$ that is universal for morphisms to finite stacks 
(a finite stack is an algebraic stack $\Gamma$ over $k$ 
with finite diagonal, which admits a flat surjective map 
$U\to \Gamma$, where $U$ is a finite scheme over $k$). It is unique 
if it exists \cite[5.2]{Borne_Vistoli_Fundamental_Gerbe}. 
By~\cite[5.7]{Borne_Vistoli_Fundamental_Gerbe} an algebraic 
stack $X$ has a fundamental gerbe if and only if it is 
inflexible. (See~\cite[5.3]{
Borne_Vistoli_Fundamental_Gerbe} for the definition of 
inflexible stacks. 
By~\cite[5.5]{Borne_Vistoli_Fundamental_Gerbe}, if $X$ is a 
geometrically connected and geometrically reduced algebraic 
stack of finite type over~$k$, then it is inflexible.)
This formalism allows to reformulate Grothendieck's section 
conjecture as follows: the traditional ``section map'' is 
bijective if and only if the natural morphism $X\to 
\Pi_{X/k}$ induces a bijection on isomorphism classes of 
$k$-rational points.

\begin{prop}
\label{exemples_CS}
 Let $P$ be a Severi-Brauer variety over a field $k$. Let 
$r$ be its exponent. Then $\Pic(P)$ is generated by an 
invertible sheaf of degree $r$, which we call $\Oc_P(r)$. 
Let $f : X\to P$ be a morphism, where $X$ is an inflexible 
algebraic stack. Assume that there exists a prime $p$ 
dividing $r$ and an invertible sheaf $\Lambda$ on $X$, such 
that $\Lambda^{\otimes p}\simeq f^*\Oc_P(r)$. Then 
$\Pi_{X/k}(k)=\emptyset$.
\end{prop}
\begin{proof}
 Let $Y$ denote the stack of $p$-th roots of $\Oc_P(r)$. For 
any $k$-scheme $S$, $Y(S)$ is the category of triples $(x, 
M, \varphi)$ where $x\in P(S)$, $M$ is an invertible sheaf 
on $S$ and $\varphi$ is an isomorphism from~$M^{\otimes p}$ 
to $x^*\Oc_P(r)$. There is a canonical forgetful morphism $g : Y\to P$ and 
there is on $Y$ a canonical $p$-th 
root $\Omega$ of $g^*\Oc_P(r)$, defined by $u^*\Omega=M$ if $u : S\to Y$ 
is the morphism corresponding to some object $(x,M,\varphi)$ of $Y(S)$. The 
sheaf $\Lambda$ defines a 
morphism from $X$ to $Y$, hence a morphism from~$\Pi_{X/k}$ 
to~$\Pi_{Y/k}$ and it suffices to prove that $\Pi_{Y/k}$ 
does not have any $k$-point.

The stack $Y$ is proper, smooth and geometrically integral 
over $k$. In particular it has an Albanese torsor 
$A^1_{\tau}(Y)$ together with a morphism $a_Y : Y\to 
A^1_{\tau}(Y)$ (see section~\ref{section_Albanese_torsor}), 
where $A^1_{\tau}(Y)$ is a torsor under 
$A^0_{\tau}(Y)=D(\picto_{Y/k})$.
 The invertible sheaf $\Omega\otimes g^*\Oc_P(\frac{r}{p})^{-1}$ (which 
is defined after some extension of $k$ that trivializes $P$) defines a section 
of $\picto_{Y/k}$ that is mapped to the section $1\in 
\Z/p\Z$ under the canonical morphism 
$\Pic_{Y/k}\to \Z/p\Z$. This proves that the induced morphism $\picto_{Y/k}\to 
\Z/p\Z$ is surjective. Since $\Pic_{P/k}$ has a trivial torsion 
component it is also injective and we get an isomorphism 
$\picto_{Y/k}\simeq\Z/p\Z$. Hence $A^0_{\tau}(Y)\simeq 
B\mu_p$ and $A^1_{\tau}(Y)$ is a $\mu_p$-gerbe. In 
particular it is a finite gerbe, so by definition of 
$\Pi_{Y/k}$ there is a morphism $\Pi_{Y/k}\to A^1_{\tau}(Y)$ 
and to conclude the proof it suffices to prove that 
$A^1_{\tau}(Y)$ has no $k$-points.

Assume that $A^1_{\tau}(Y)$ has a $k$-point. Now we claim that any $k$-point of 
$\picto_{Y/k}$ is induced by a genuine invertible sheaf on $Y$. This yields a 
contradiction because then the invertible sheaf $g^*\Oc_P(\frac{r}p)\otimes 
\Omega^{-1}$ is defined over $k$, hence so is $\Oc_P(\frac{r}p)$. To prove the 
claim, notice that a $k$-point of~$A^1_{\tau}(Y)$ is a retraction of the natural 
morphism $\bgm \to \champic^{\tau}(Y/S)$. Then the natural sequence
$$0\lto \bgm \lto \champic^{\tau}(Y/k)\lto \picto_{Y/k}\lto 0$$
splits so the projection $\champic^{\tau}(Y/k)\to \picto_{Y/k}$ has a section and this yields the assertion.
\end{proof}

\begin{remarque}\rm\ 
\begin{enumerate}
\item Borne and Vistoli give an elegant independent proof of 
Proposition~\ref{exemples_CS}, using the dual Severi-Brauer 
variety (see~\cite[13.2]{Borne_Vistoli_Fundamental_Gerbe}).
 \item Proposition~\ref{exemples_CS} provides a way to 
generate examples of smooth projective geometrically 
connected curves that satisfy Grothendieck's section 
conjecture, over any field with non-trivial Brauer group 
(see~\cite[section 13]{Borne_Vistoli_Fundamental_Gerbe} for 
more details).
\end{enumerate}
\end{remarque}

\begin{remarque}\rm
 \label{lien_avec_gerbe_fondamentale}
For any inflexible algebraic stack~$X$ over a field~$k$, one can define the 
abelianization~$\Pi_{X/k}^{\textrm{ab}}$ through a rigidification process. It 
comes with a natural morphism~$\Pi_{X/k} \to \Pi_{X/k}^{\textrm{ab}}$ and the 
composition~$X\to \Pi_{X/k}^{\textrm{ab}}$ is universal for morphisms from~$X$ 
to a gerbe banded by a finite abelian group 
(see~\cite[App. A]{Biswas_Borne_Tamely_ramified_torsors} or 
\cite{Borne_Vistoli_Fundamental_Gerbe_2} for more details about the abelianized 
fundamental gerbe).
The role played by the Albanese torsor~$A^1_{\tau}(Y)$ in 
the proof of~\ref{exemples_CS} is not surprising. Indeed, it 
turns out that~$A^1_{\tau}(Y)$ is isomorphic 
to~$\Pi_{Y/k}^{\textrm{ab}}$: it is a finite gerbe (hence profinite), and it 
follows from~\ref{PU_albanese} that the Albanese 
morphism~$Y\to A^1_{\tau}(Y)$ is universal for maps from~$Y$ 
to a gerbe banded by a finite abelian group. Hence it 
satisfies the same universal property as $Y\to 
\Pi_{Y/k}^{\textrm{ab}}$ and in particular 
$A^1_{\tau}(Y)\simeq \Pi_{Y/k}^{\textrm{ab}}$.
\end{remarque}

\section{Some vanishing results for Ext sheaves}
\label{section_Ext}

In this section, we recall some vanishing or 
representability theorems for sheaves of the form 
$\shExt^i(G, \gm)$. In the following results, $S$ is a base 
scheme, and $G$ is a commutative group scheme over $S$. The 
sheaves $\shExt^i$ are computed as derived functors in the 
abelian category of \emph{fppf} sheaves of commutative 
groups. Note that if $G$ and $H$ are such sheaves, then 
$\shExt^i(G,H)$ is also the $\emph{fppf}$ sheaf associated 
with the presheaf $T\mapsto \Ext^i_T(G\times_S T, H\times_S 
T)$. With this description, it is clear that forming 
$\shExt^i(G,H)$ commutes with any base change $T\to S$. As 
mentioned on page~\pageref{notations}, $\shExt^i(G, \gm)$ 
will be denoted by~$E^i(G)$.

\begin{thm}[{\cite[VIII~prop.~3.3.1]{SGA7}}]
\label{vanishing_ext_finite_or_multiplicative}
 Assume that $G$ is finite and locally free over $S$, or that it is of finite type and of multiplicative type. Then $E^1(G)=0$.
\end{thm}

\begin{thm}[\cite{Breen_Theoreme_annulation}, main theorem and Remark~3]
\label{vanishing_ext23_finite}
 Assume that $G$ is finite and flat over $S$, and that 2 is invertible in $S$. Then $E^2(G)=E^3(G)=0.$
\end{thm}
\begin{remarque}\rm
 If $S$ is the spectrum of a separably closed field of 
characteristic~2, then it is known that $E^2(\alpha_2)\neq 
0$ (see~\cite{Breen_Nontrivial_higher_extension}).
\end{remarque}

\begin{thm}
\label{dual_and_ext1_abelian}
 If $G$ is an abelian scheme over $S$, then~$E^1(G)$ is 
representable by the dual abelian scheme $G^t$ 
\cite[17.6]{Oort_Commutative_Group_Schemes}. On the other 
hand, the Cartier dual $G^D$ is zero (obvious, \emph{e.g.} 
use~\ref{morphisme_constant}).
\end{thm}

\begin{cor}
\label{vanishing_ext23_abelian}
Let $A$ be an abelian scheme and $G$ a 
duabelian group scheme over $S$. Then:
\begin{itemize}
 \item[(i)] For any $n\in \N^*$, the multiplication by $n$ 
in $E^2(A)$ is a monomorphism. If $2\in \Oc_S^{\times}$, the 
multiplication by $n$ in $E^i(A)$ is an isomorphism for 
$i=2, 3$ and a monomorphism for $i=4$.
\item[(ii)] If $F$ is a finite and locally free commutative 
group scheme, then any morphism from $F$ to~$E^2(A)$ is 
trivial. If $2\in \Oc_S^{\times}$, this is also true for 
$E^3(A)$, $E^4(A)$ and $E^2(G)$.
\item[(iii)] If $S$ is regular and $2\in \Oc_S^{\times}$, 
then $E^2(A)=E^3(A)=0$.
\end{itemize}
\end{cor}
\begin{proof}
The exact sequence of \emph{fppf} sheaves
$$\xymatrix{0\ar[r]& _nA \ar[r]& A \ar[r]^n &A \ar[r]& 0}$$
induces an exact sequence:
  $$\xymatrix{E^{i-1}(_nA) \ar[r]& E^{i}(A) \ar[r]^n &E^{i}(A)\ar[r]&E^{i}(_nA)}$$
But the scheme $_nA$ is finite and locally free over $S$, 
hence $E^1(_nA)=0$ by~%
\ref{vanishing_ext_finite_or_multiplicative}. If $2\in 
\Oc_S^{\times}$, by~\ref{vanishing_ext23_finite} the 
groups $E^{2}(_nA)$ and $E^3(_nA)$ also vanish. This 
gives (i). If $S$ is regular 
then $E^2(A)$ and $E^3(A)$ are torsion 
by~\cite[\S~7]{Breen_Extensions_abelian}, whence (iii). In 
(ii), the question is local on $S$ so 
we may assume that $F$ is free of order $n$. But then it is 
killed by $n$ (\cite[\S 1]{Oort_Tate_Group_Schemes}) and 
the result follows from (i) for the morphisms to $E^2(A)$, 
$E^3(A)$ or $E^4(A)$. Since $G$ is duabelian, 
by~\ref{prop_duabelian_groups} 
and~\ref{vanishing_ext23_finite} the sheaf $E^2(G)$ is 
isomorphic to $E^2(A')$ where $A'$ is an abelian scheme. 
Hence any morphism from $F$ to $E^2(G)$ is trivial by the 
previous case.
\end{proof}

\begin{thm}
 \label{vanishing_ext_constant}
Let $G$ be a finitely generated twisted constant group over~$S$ 
(see~\ref{Cartier_groups}). Then $E^i(G)=0$ for all $i>0$.
\end{thm}
\begin{proof}
The question is local on~$S$ so we may assume that $G$ is constant, associated 
to an ordinary abelian group~$M$ of finite type. It suffices to consider the 
cases $M=\Z$ and $M=\Z/n\Z$. If $M=\Z$, the functor $\shHom(G, .)$ is the 
identity hence $E^i(G)=0$ for all $i>0$.
If $M=\Z/n\Z$, the result follows by considering the long exact sequence associated with the short exact sequence $0\to \Z \to \Z \to \Z/n\Z\to 0$ (for $i=1$ the statement also follows from~\ref{vanishing_ext_finite_or_multiplicative}).   
\end{proof}

\bibliographystyle{plain}
\bibliography{../../../../texmf/tex/latex/monlatex/mabiblio}
\end{document}

%% file: macros.tex
\newcommand{\be}{\begin{enumerate}}
\newcommand{\ee}{\end{enumerate}}
\newcommand{\beqn}{\begin{eqnarray*}}
\newcommand{\eeqn}{\end{eqnarray*}}

\newcommand{\incl}[1][r]
      {\ar@<-0.2pc>@{^(-}[#1] \ar@<+0.2pc>@{-}[#1]}

\def\N{{\mathbb N}}

\def\Z{{\mathbb Z}}
\def\Ac{{\mathcal A}}

\def\Cc{{\mathcal C}}

\def\Fc{{\mathcal F}}
\def\Gc{{\mathcal G}}
\def\Hc{{\mathcal H}}

\def\Lc{{\mathcal L}}

\def\Oc{{\mathcal O}}
\def\Pc{{\mathcal P}}

\def\Tc{{\mathcal T}}

\def\Xc{{\mathcal{X}}}

\def\mgo{{\mathfrak m}}

\def\gm{{\mathbb G}_{\rm{m}}}

\def\bgm{{\rm B}{\mathbb G}_{\rm{m}}}

\def\eps{\varepsilon}

\def\inj{\hookrightarrow}

\def\ov{\overline}
\newcommand{\cartesien}{\ar@{}[dr]|{\square}}
\def\FlecheNE{\rotatebox[origin=c]{45}{\ensuremath \Rightarrow}}

\def\FlecheSO{\rotatebox[origin=c]{225}{\ensuremath \Rightarrow}}

\newcommand{\double}{\ar@<2pt>[r] \ar@<-2pt>[r]}

\def\Alb{\textrm{\rm Alb}}

\def\Aut{\hbox{\rm Aut}}

\def\ch{\hbox{\rm ch}\,}
\def\champic{\mathcal{P}ic}

\def\Coker{\hbox{\rm Coker}\,}

\DeclareMathOperator{\Ext}{Ext}

\def\fExt{{\mathcal{E}xt}\,}

\def\Hom{\hbox{\rm Hom}}

\def\id{\hbox{\rm id}}

\def\Isom{\hbox{\rm Isom}}

\def\Ker{\hbox{\rm Ker}\,}

\def\Pic{\hbox{\rm Pic}\,}
\def\picto{\Pic^{\tau}}

\def\RHom{\hbox{\rm RHom}}

\def\shExt{{\mathcal{E}xt}\,}
\def\shHom{{\mathcal{H}om}\,}
\def\shIsom{{\mathcal I\!som}\,}

\DeclareMathOperator{\Spec}{Spec}

	{\begin{pmatrix}}%
	{\end{pmatrix}}
	{\begin{vmatrix}}%
	{\end{vmatrix}}

	{\ensuremath{\left \{ \hskip -1.5 mm \begin{array}{#1@{\ =\ }l}}}%
	{\end{array}\right.}

	{\ensuremath{\left \{ \hskip -1.5 mm%
	 \begin{array}{#1@{\ }c@{\ }l}}}%
	{\end{array}\right.}

	{\ensuremath{\left \{ \hskip -1.5 mm \begin{array}{#1@{\quad}l}}}%
	{\end{array}\right.}

\renewcommand{\setminus}{\smallsetminus}
\newcommand{\details}[1]{}

%% file: englishsectionnement.tex
\newtheorem{cor}[subsection]{Corollary}

\newtheorem{prop}[subsection]{Proposition}

\newtheorem{thm}[subsection]{Theorem}

\newtheorem{defi}[subsection]{Definition}

\newtheorem{lem}[subsection]{Lemma}

\newtheorem{remarque}[subsection]{Remark}

\newtheorem{exemple}[subsection]{Example}

\newtheorem{conjecture}[subsection]{Conjecture}
\newtheorem{sousconjecture}[subsubsection]{Conjecture}

\newtheorem{question}[subsection]{Question}



%% file: courbetordue.pstex_t
\begin{picture}(0,0)%
\includegraphics{courbetordue.pstex}%
\end{picture}%
\setlength{\unitlength}{1243sp}%
\begingroup\makeatletter\ifx\SetFigFont\undefined%
\gdef\SetFigFont#1#2#3#4#5{%
  \reset@font\fontsize{#1}{#2pt}%
  \fontfamily{#3}\fontseries{#4}\fontshape{#5}%
  \selectfont}%
\fi\endgroup%
\begin{picture}(6129,3898)(2596,-7127)
\put(2611,-6991){\makebox(0,0)[lb]{\smash{{\SetFigFont{9}{10.8}{\rmdefault}{\mddefault}{\updefault}{\color[rgb]{0,0,0}$X$}%
}}}}
\put(2611,-5281){\makebox(0,0)[lb]{\smash{{\SetFigFont{9}{10.8}{\familydefault}{\mddefault}{\updefault}{\color[rgb]{0,0,0}$\mathcal{X}$}%
}}}}
\put(5716,-5281){\makebox(0,0)[lb]{\smash{{\SetFigFont{9}{10.8}{\rmdefault}{\mddefault}{\updefault}{\color[rgb]{0,0,0}$\pi$}%
}}}}
\put(6211,-6586){\makebox(0,0)[lb]{\smash{{\SetFigFont{9}{10.8}{\familydefault}{\mddefault}{\updefault}{\color[rgb]{0,0,0}$p$}%
}}}}
\put(7111,-3796){\makebox(0,0)[lb]{\smash{{\SetFigFont{9}{10.8}{\rmdefault}{\mddefault}{\updefault}{\color[rgb]{0,0,0}$\mu_r$}%
}}}}
\end{picture}%